\documentclass[11pt,a4paper]{article}
\usepackage{authblk}
\usepackage{bm}
\usepackage{geometry}
\usepackage{fancyhdr}
\usepackage{tikz}
\usetikzlibrary{shapes.geometric}
\usepackage{booktabs}
\usepackage{amsmath,amsthm,amssymb,amsfonts,mathrsfs,amscd}
\usepackage{enumerate}
\newtheorem{thm}{Theorem}[section]
\newtheorem{prop}[thm]{Proposition}
\newtheorem{lem}[thm]{Lemma}
\newtheorem{cor}[thm]{Corollary}

\newenvironment{abs}
{%
  \par
  \noindent\textbf{Abstract.}\enspace
  \normalfont
  \ignorespaces
}
{
  \par
}
\theoremstyle{definition}
\newtheorem{definition}[thm]{Definition}
\newtheorem{example}[thm]{Example}

\newtheorem{remark}[thm]{Remark}
\numberwithin{equation}{section}

\begin{document}
	
\title{Carleson Interpolation, Boundary Rigidity, and Classification of Spectral Transformations Preserving Diagonal Operator Orbit Frames}

\author{Jian Wu\thanks{Corresponding author:
\texttt{xingxingwu2022@163.com}}}

\affil{}

	\renewcommand*{\Affilfont}{\small\it}
	\renewcommand\Authands{ and }
	\date{}

	\maketitle

\begin{abs}
 Preservation of frame generator sets for single
orbits of diagonal operators leads to a geometric mapping
problem in the unit disk. Let \(\mathcal C\) denote the class of
Carleson interpolating sequences in \(\mathbb D\), hereafter simply
called Carleson sequences. We study arbitrary
maps \(\Phi:\mathbb D\to\mathbb D\) such that, for every sequence
\(\Lambda\subset\mathbb D\),
\(
\Lambda\in\mathcal C
\Longleftrightarrow
\Phi(\Lambda)\in\mathcal C,
\)
and
\(
1-|\Phi(z)|^2\asymp1-|z|^2,
\ z\in\mathbb D.
\)
Equivalently, they preserve Riesz sequences of normalized
Szeg\H{o} kernels in both directions, and the norms of the
unnormalized kernels at \(z\) and \(\Phi(z)\) are uniformly
comparable.
Without continuity, measurability, or analyticity assumptions on
\(\Phi\), bidirectional Carleson sequence preservers are highly
flexible; to our knowledge, no general structural classification
is available. For single orbits, equality of frame generator sets
forces the boundary defect estimate through the coordinate condition
\(|f_j|^2\asymp1-|\lambda_j|^2\), making it intrinsic to the
orbit frame preservation problem rather than an additional
hypothesis on Carleson sequence preservers.
Equivalently, the resulting preserver class is characterized by
uniform preservation of the Carleson measure estimates in both
directions for finite positive discrete measures together with
bidirectional preservation of pairwise pseudohyperbolic separation.

We prove that every such \(\Phi\) admits a canonical boundary shadow
\(
h_\Phi(\zeta)
=
\lim_{r\to1^-}\Phi(r\zeta),
\
\zeta\in\mathbb T,
\)
where the convergence is uniform, and that
\(
h_\Phi\in\operatorname{BiLip}(\mathbb T).
\)
Writing \(\mathcal P=\mathcal P_1\) for the preserver class, we give
an intrinsic geometric characterization of its boundary shadow kernel
\(
\mathcal K
=
\bigl\{
\Psi\in\mathcal P:
h_\Psi=\operatorname{id}_{\mathbb T}
\bigr\}
\)
and prove the unique factorization
\(
\Phi=\Psi\circ E_{h_\Phi},
\ \Psi\in\mathcal K,
\)
where \(E_h(0)=0\) and \(E_h(r\zeta)=rh(\zeta)\). Consequently,
\(
\mathcal P
\cong
\mathcal K\rtimes\operatorname{BiLip}(\mathbb T)
\)
as monoids.

For every finite \(m\geq1\), the preservation class for \(m\) orbits
satisfies \(\mathcal P_m=\mathcal P\). For countably many orbits,
we prove
\(
\operatorname{Aut}(\mathbb D)\subsetneq\mathcal P_\omega
\subsetneq\mathcal P.
\)
Every preserver for countably many orbits is a uniform homeomorphism
in the pseudohyperbolic metric. For mappings satisfying
\(|\Phi(z)-\varphi(z)|=o(1-|z|)\) uniformly as \(|z|\to1\),
where \(\varphi\) is a disk automorphism, membership in
\(\mathcal P_\omega\) is equivalent to being a disk homeomorphism.
We also characterize arbitrary preservers for countably many
orbits by tests in both directions.
The holomorphic members of \(\mathcal P\) and
\(\mathcal P_\omega\) are precisely disk automorphisms.


\end{abs}
	
\section{Introduction}

Let \(H^\infty(\mathbb D)\) denote the algebra of bounded holomorphic
functions on the unit disk. A sequence
\(\Lambda=\{\lambda_j\}_{j\geq1}\subset\mathbb D\)
is called a Carleson interpolating sequence, or briefly a
Carleson sequence, if every bounded scalar sequence can be
interpolated on \(\Lambda\) by a function in
\(H^\infty(\mathbb D)\). Sequences are understood as indexed families; in particular, image
sequences retain repetitions. Carleson's interpolation theorem \cite{Carleson} gives several equivalent geometric
characterizations of such sequences, notably the uniform separation
condition; equivalently, one may require pairwise separation in the
pseudohyperbolic metric together with a Carleson measure condition;
for the connection with Carleson measure estimates and further
developments, see
\cite{Carleson1962,CarlesonGarnett1975,Duren1970,Garnett,Seip}.
We denote the class of all Carleson sequences by \(\mathcal C\).
Shapiro and Shields \cite{ShapiroShields1961} showed that these are
precisely the interpolating sequences for the normalized interpolation
problem in \(H^2(\mathbb D)\). This gives an equivalent reproducing
kernel formulation. For \(\lambda\in\mathbb D\), let
\(
k_\lambda(z)
:=
\frac{1}{1-\overline{\lambda}z},
\
\widehat{k}_\lambda
:=
\sqrt{1-|\lambda|^2}\,k_\lambda
\)
denote the Szeg\H{o} kernel of \(H^2(\mathbb D)\) and its normalized
form, respectively. In terms of these kernels, the characterization
states that
\(
\Lambda=\{\lambda_j\}_{j\geq1}\in\mathcal C
\Longleftrightarrow
\{\widehat{k}_{\lambda_j}\}_{j\geq1}
\text{ is a Riesz sequence in }H^2(\mathbb D);
\)
see also \cite{Garnett,Seip}.
Thus the geometry of Carleson sequences is equivalently the Riesz
geometry of normalized Szeg\H{o} kernels.

Related interpolation results hold in normalized complete Pick spaces.
Hartz \cite[Theorems~1.2 and~4.5]{Hartz2023} proves that every such
space satisfies the column-row property with constant one and uses
it to give a new proof of the characterization of interpolating
sequences for its multiplier algebra by weak separation and a
Carleson measure condition. This provides a broader context for
the Szeg\H{o} kernel formulation used here.

These characterizations naturally lead us to ask which mappings
preserve Carleson interpolating sequences and what geometric
constraints such preservation imposes. The characterization by
pseudohyperbolic separation and a Carleson measure condition
connects this question with the transformation of Carleson measures.

For Carleson measures, Semmes \cite{Semmes1988} obtained weighted
pullback estimates for quasiconformal self-homeomorphisms of the disk
that are bilipschitz in the Poincar\'e metric and have strongly
quasisymmetric boundary values. Zinsmeister \cite{Zinsmeister1989}
studied weighted pullback and pushforward operators induced by
conformal maps, relating their boundedness to BMOA and boundary
geometry. Wei and Zinsmeister \cite{WeiZinsmeister2018} developed
this theory further, including results for vanishing Carleson
measures and chord-arc domains.

For Carleson interpolating sequences, a characterization of
preserving maps is known within the quasiconformal category.
Astala and Zinsmeister \cite{AstalaZinsmeister} proved that a
quasiconformal self-homeomorphism of the upper half-plane maps
Carleson sequences to Carleson sequences whenever its boundary map
is strongly quasisymmetric. Gonz\'alez and Nicolau
\cite{GonzalezNicolau} subsequently established the converse.
Since the inverse of a strongly quasisymmetric homeomorphism is
strongly quasisymmetric, applying this characterization to the
inverse also yields bidirectional preservation. The corresponding
formulation in the disk follows by conformal equivalence.

Beyond the quasiconformal setting, a natural problem is to
characterize arbitrary mappings
\(
\Phi:\mathbb D\longrightarrow\mathbb D
\)
such that
\(
\Lambda\in\mathcal C
\Longleftrightarrow
\Phi(\Lambda)\in\mathcal C
\)
for every sequence \(\Lambda\subset\mathbb D\). To the best of our
knowledge, no general structural classification is known for such
maps without additional assumptions.

 Our problem arises from the preservation of frame generator sets for diagonal operator orbits, which also forces uniform comparability of boundary defects. We prove that these two requirements together yield a canonical bilipschitz boundary trace and a structural decomposition of the preserver class, without assuming continuity, measurability, or analyticity.

The results on Carleson measures cited above concern weighted
pullback and pushforward operators. Our formulation uses ordinary
pushforward of finite positive discrete measures, retaining each
atom's mass.

We briefly describe the setting in operator theory. Dynamical sampling
studies the recovery of an evolving state from measurements taken
along an operator orbit. If
\(\mathsf E\in\mathcal B(H)\) is an operator describing evolution in discrete time,
\(x\in H\) is an unknown initial state, and \(g\in H\) is a sampling
vector, then
\(
\langle \mathsf E^n x,g\rangle
=
\langle x,(\mathsf E^*)^n g\rangle,
\ n\geq0.
\)
Thus stable recovery of every initial state is equivalent to the frame
property of
\(
\{(\mathsf E^*)^ng\}_{n\geq0}.
\)
For background on dynamical sampling and a recent survey, see
\cite{AldroubiCabrelliMolterTang2017,DynamicalSamplingSurvey}.
Related work on stable reconstruction, nonuniform sampling, relevant sampling, frame localization, and Banach frames includes
\cite{AdcockHansen2012,AldroubiBaskakovKrishtal2008,
AldroubiGroechenig2001,BassGroechenig2013,Groechenig2004}.
Operator representations of frames, iterated systems, and frame
properties of operator orbits were studied in
\cite{OperatorRepresentations,IterativeFrameProperties,
OperatorOrbitFrames,DynamicalFrameRepresentations}.
Generalized operator orbits and Bessel orbits of normal operators were
considered in \cite{ChristensenHasannasab2023} and
\cite{PhilippBesselOrbits2017}, respectively.
Most directly related to the present work are the descriptions of frames formed by several orbits in terms of model spaces and geometry in
\cite{CabrelliMolterSuarez2021,CabrelliMolterSuarez2023}
and the recent structural results on Carleson frames in
\cite{ChristensenHasannasabPhilippStoeva2024,Krishtal}.

Throughout the paper, Hilbert spaces are complex, and inner products
are linear in the first variable. By a diagonal operator we mean
a bounded operator that is diagonal with respect to an orthonormal
basis.
Recall that a sequence \(\{f_i\}_{i\in I}\subset H\) is a frame for
\(H\) if there exist constants \(0<a\leq b<\infty\) such that
\(
a\|x\|^2
\leq
\sum_{i\in I}|\langle x,f_i\rangle|^2
\leq
b\|x\|^2\) for \(x\in H\).
A sequence \(\{f_i\}_{i\in I}\subset H\) is called a Riesz sequence
if there exist constants \(0<a\leq b<\infty\) such that
\(
a\sum_{i\in I}|c_i|^2
\leq
\left\|
\sum_{i\in I}c_i f_i
\right\|^2
\leq
b\sum_{i\in I}|c_i|^2
\)
for every finitely supported scalar family \(\{c_i\}_{i\in I}\).
If, in addition,
\(
\overline{\operatorname{span}}\{f_i:i\in I\}=H,
\)
then \(\{f_i\}_{i\in I}\) is called a Riesz basis for \(H\).
We refer to \cite{ChristensenBook} for the standard theory of frames,
Riesz sequences, and Riesz bases.

Fix \(m\in\mathbb N^+\), let \(H\) be a separable Hilbert space of infinite dimension, and consider a diagonal
operator
\(
A=\sum_{j\geq1}\lambda_jP_j\in\mathcal B(H),
\ \lambda_j\in\mathbb D,
\)
where the \(\lambda_j\)'s are pairwise distinct, the \(P_j\)'s are
pairwise orthogonal nonzero projections, and
\(
\sum_{j\geq1}P_j=I_H
\)
in the strong operator topology. We write
\(
E_j:=P_jH
\)
for the spectral subspace associated with \(\lambda_j\). For
\(
F=(f_1,\ldots,f_m)\in H^m,
\)
define
\(
\mathcal O_m(A,F)
:=
\{A^nf_i\}_{n\geq0,\ 1\leq i\leq m}
\)
and
\(
\mathcal G_m(A)
:=
\left\{
F\in H^m:
\mathcal O_m(A,F)\text{ is a frame for }H
\right\}.
\)
Given an arbitrary mapping
\(\Phi:\mathbb D\to\mathbb D\), by a harmless abuse of notation, we use the same symbol $\Phi$ for
the point mapping and the induced spectral transformation
\(
\Phi(A):=\sum_{j\ge1}\Phi(\lambda_j)P_j.
\)
If several values $\Phi(\lambda_j)$ coincide, the corresponding
spectral subspace of $\Phi(A)$ is understood as the orthogonal direct
sum of the associated spaces $P_jH$.

Let \(\mathcal P_m\) denote the class of mappings
\(\Phi:\mathbb D\to\mathbb D\) such that
\(
\mathcal G_m(A)
=
\mathcal G_m\bigl(\Phi(A)\bigr)
\)
for every diagonal operator \(A\) of the above form. We write
\(
\mathcal P:=\mathcal P_1.
\)
If \(\mathcal G_1(A)\neq\varnothing\), then every spectral subspace
\(E_j\) has dimension one \cite{AldroubiCabrelliMolterTang2017}. Hence the nontrivial part of the problem involving a single orbit automatically reduces to the case of a simple spectrum.
Whenever an orthonormal basis
\(
\{e_j\}_{j\geq1}
\)
of \(H\) and a sequence
\(
\Lambda=\{\lambda_j\}_{j\geq1}\subset\mathbb D
\)
are fixed, we write \(A_\Lambda\) for the diagonal operator defined by
\(
A_\Lambda e_j=\lambda_j e_j,
\ j\geq1.
\) Here \(\Lambda=\{\lambda_j\}_{j\geq1}\) is an indexed diagonal
sequence and may contain repetitions; equal diagonal values are
grouped into a single spectral subspace. We write
\(
\mathcal G:=\mathcal G_1
\).
If
\(
f=\sum_{j\geq1}f_je_j
\),
then
\(
\{A_\Lambda^nf\}_{n\geq0}
\)
is a frame for \(H\) if and only if
\(
\Lambda\in\mathcal C\) and
\(|f_j|^2\asymp1-|\lambda_j|^2.
\)
If \(\Lambda\notin\mathcal C\), then
\(
\mathcal G(A_\Lambda)=\varnothing.
\)
These facts follow from the characterization of frames generated by orbits of normal operators in
\cite{AldroubiCabrelliMolterTang2017,IterativeNormalOperators,
CabrelliMolterPaternostroPhilipp2020}.

This characterization leads to the first main result of the paper. We
prove that
\(
\Phi\in\mathcal P
\)
if and only if the following two conditions hold:
\(
\Lambda\in\mathcal C
\Longleftrightarrow
\Phi(\Lambda)\in\mathcal C
\)
for every sequence \(\Lambda\subset\mathbb D\), and
\(
1-|\Phi(z)|^2
\asymp
1-|z|^2,
\ z\in\mathbb D,
\)
with constants independent of \(z\). By the normalized kernel characterization of interpolating sequences
and the identity
\(
\|k_z\|_{H^2}^2=(1-|z|^2)^{-1},
\)
the same theorem admits a formulation purely in terms of kernels:
\(\Phi\in\mathcal P\) if and only if it preserves, in both directions,
Riesz sequences of normalized Szeg\H{o} kernels and satisfies
\(
\|k_{\Phi(z)}\|_{H^2}
\asymp
\|k_z\|_{H^2}\) for \( z\in\mathbb D\).

Carleson's characterization provides a useful interpretation of
these two conditions. A sequence
\(\Lambda=\{\lambda_j\}\subset\mathbb D\) is interpolating if and
only if it is pairwise separated in the pseudohyperbolic metric
and its canonical discrete measure
\(
\mu_\Lambda
=
\sum_j(1-|\lambda_j|^2)\delta_{\lambda_j}
\)
is a Carleson measure. In this characterization, the measure
condition involves the canonical boundary weights and is coupled
with pairwise separation. The boundary defect comparability
\(
1-|\Phi(z)|^2\asymp1-|z|^2\) for \(z\in\mathbb D\)
ensures that the canonical measure \(\mu_{\Phi(\Lambda)}\) is
comparable to the ordinary pushforward \(\Phi_{\#}\mu_\Lambda\),
with constants independent of \(\Lambda\).

Theorem~\ref{thm:carleson-measure-separation-characterization}
establishes an equivalent characterization of the
preserver class: a uniform two-sided Carleson norm estimate for
the ordinary pushforward of every finitely supported positive
measure, together with bidirectional preservation of pairwise
pseudohyperbolic separation.

Neither bidirectional preservation of Carleson sequences nor comparability of boundary defects alone characterizes the class \(\mathcal P\). Boundary defect comparability contains essentially
radial information and does not control the angular distribution of a
sequence; compare Naftalevi\v{c}'s rotation theorem
\cite{Naftalevic}. On the other hand, preservation of Carleson
sequences alone does not control the coordinate weights that determine
the frame generators. The main geometric problem is therefore to
understand the interaction between these two conditions.
We prove that their conjunction forces a strong boundary rigidity
phenomenon. For every \(\Phi\in\mathcal P\), the radial boundary trace
\(
h_\Phi(\zeta)
:=
\lim_{r\to1^-}\Phi(r\zeta),
\ \zeta\in\mathbb T,
\)
exists for every \(\zeta\in\mathbb T\), with convergence uniform in
\(\zeta\), and satisfies an error estimate of first order. Moreover,
\(
h_\Phi\in\operatorname{BiLip}(\mathbb T).
\)

Unlike the quasiconformal characterization of Gonz\'alez and Nicolau,
the present result derives a bilipschitz boundary trace from preservation of frames generated by operator orbits without assuming continuity,
quasiconformality, or surjectivity of \(\Phi\); discontinuous and
nonsurjective behavior may nevertheless persist in the interior.

The main difficulty in the boundary theorem is to prove the
injectivity of \(h_\Phi\). We combine an argument using regular double points
with rescaling in the half plane model and the bidirectional preservation, for sequences approaching the
boundary, of the property that pseudohyperbolic distances
tend to zero. The resulting
limit mappings are continuous, injective, and proper mappings of the half plane into itself. By invariance of domain and properness, their
images are both open and closed; hence the limit mappings are
surjective, from which the required contradiction follows.

After identifying the boundary action, we separate it from the part of the mapping with trivial boundary action. For an arbitrary mapping
\(
h:\mathbb T\to\mathbb T,
\)
define
\(
E_h(0):=0,
\
E_h(r\zeta):=rh(\zeta),
\
0<r<1.
\)
We prove that
\(
E_h\in\mathcal P
\Longleftrightarrow
h\in\operatorname{BiLip}(\mathbb T).
\)
We then introduce the boundary shadow kernel
\(
\mathcal K
:=
\left\{
\Psi\in\mathcal P:
h_\Psi=\operatorname{id}_{\mathbb T}
\right\}.
\)
The class \(\mathcal K\) is characterized intrinsically by
injectivity, proper behavior at the boundary in both directions, a uniform
pseudohyperbolic displacement bound strictly below \(1\) near the
boundary, and the bidirectional preservation, for sequences approaching the
boundary, of the property that pseudohyperbolic distances
tend to zero.

Every \(\Phi\in\mathcal P\) admits the unique factorization
\(
\Phi
=
\Psi\circ E_{h_\Phi},
\
\Psi
=
\Phi\circ E_{h_\Phi^{-1}}
\in\mathcal K.
\)
Consequently,
\(
\mathcal P
\cong
\mathcal K\rtimes
\operatorname{BiLip}(\mathbb T)
\)
as monoids. This yields the announced structural classification of \(\mathcal P\).

For finitely many orbits, we prove that
\(
\mathcal P_m=\mathcal P_1,
\ m\in\mathbb N^+.
\)
This is not a formal consequence of the result for a single orbit. If
\(m\) operator orbits form a frame, then every eigenspace has
dimension at most \(m\), while the set of distinct eigenvalues is a
union of at most \(m\) Carleson sequences
\cite[Theorem~2.3]{CabrelliMolterPaternostroPhilipp2020}; these
conditions describe spectral multiplicity and spectral geometry,
respectively. Using a Szeg\H{o} kernel model with vector values and Hartmann's decomposition \cite[Proposition~1.1]{Hartmann}, we show that the associated synthesis operator is bounded below if and only if its restrictions to the finite Hartmann clusters are uniformly bounded below, and that this condition is invariant under \(\Phi\).

As a further geometric consequence, every
\(\Phi\in\mathcal P\) preserves the Carleson index of an arbitrary
spectral set, namely the least number of Carleson sequences required
in a finite decomposition:
\(
\operatorname{ind}_{\mathcal C}\bigl(\Phi(\Lambda)\bigr)
=
\operatorname{ind}_{\mathcal C}(\Lambda).
\)
Thus, \(\Phi\) preserves in both directions not only the class of
finite unions of Carleson sequences, but also their optimal
decomposition number. The holomorphic members of every preservation class for finitely many orbits are precisely the disk automorphisms.

We then introduce the preservation class for countably many orbits
\(\mathcal P_\omega\). Every member belongs to \(\mathcal P\) and is
a uniform homeomorphism of \(\mathbb D\) in the pseudohyperbolic
metric. Disk automorphisms preserve countable orbit frames by the
unitary covariance of normalized Szeg\H{o} kernels. The same
holomorphic rigidity holds in the preservation class for countably many orbits. We also obtain an
explicit criterion for boundary perturbations of disk automorphisms.
If \(\varphi\in\operatorname{Aut}(\mathbb D)\) and
\(|\Phi(z)-\varphi(z)|=o(1-|z|)\) uniformly as \(|z|\to1\),
then \(\Phi\in\mathcal P_\omega\) if and only if \(\Phi\) is a
disk homeomorphism; see
Corollary~\ref{cor:countable-automorphism-perturbation}.
The proof uses an estimate for perturbations of kernel synthesis
operators that is independent of the dimension of the coefficient
space, together with observation localization on the spectrum. In particular, among disk
homeomorphisms, changing a mapping inside a compact subdisk does not
affect membership in \(\mathcal P_\omega\).
For \(0<\varepsilon<\frac{1}{2}\), the radial homeomorphism
\(K_\varepsilon(z)=[1-\varepsilon(1-|z|)^2]z\) therefore belongs
to \(\mathcal P_\omega\), although it is nonholomorphic on every
boundary annulus; see
Example~\ref{ex:countable-nonholomorphic-preserver}.
On the other hand,
Example~\ref{ex:countable-conjugation-not-preserver} shows that
\(J(z)=\bar z\) belongs to \(\mathcal P\setminus\mathcal P_\omega\),
although it is a pseudohyperbolic isometry. Thus
\(\mathcal P_\omega\) is a proper subclass of \(\mathcal P\) even
among disk homeomorphisms. Consequently,
\(
\operatorname{Aut}(\mathbb D)\subsetneq\mathcal P_\omega
\subsetneq\mathcal P.
\)

For an arbitrary mapping \(\Phi:\mathbb D\to\mathbb D\),
Theorem~\ref{thm:countable-uniform-matrix-characterization}
characterizes membership in \(\mathcal P_\omega\) by bounds in both directions for finite orbit frame matrices. The constants depend only
on the input frame bounds and on \(\Phi\), and are independent of
dimension, spectral points, and multiplicities. Necessity follows
by assembling finite configurations into one countable orbit frame;
sufficiency follows by finite spectral compression. This gives an
exact analytic criterion for the preservation class for countably many orbits, while an
explicit geometric classification of the full class remains open.

The paper is organized as follows. Section~2 establishes equivalent characterizations of preservation for single orbits in terms of operator orbits, Szeg\H{o} kernels, and Carleson geometry.
Section~3 constructs the boundary shadow of a preserver and
derives quantitative Lipschitz bounds and estimates preventing its collapse.
Section~4 proves the injectivity and bilipschitz rigidity of the
boundary shadow, derives the equivalent characterization in terms of
Carleson measure estimates in both directions and pairwise pseudohyperbolic
separation, and characterizes the radial extensions belonging to the
preserver class. Section~5 gives an intrinsic geometric
characterization of the boundary shadow kernel and establishes the unique factorization into kernel and angular components and the resulting
semidirect product structure. Section~\ref{sec:finite-orbits} proves
\(\mathcal P_m=\mathcal P_1\) for every finite \(m\), derives
invariance of the minimal number of generators and the Carleson
index, and establishes holomorphic rigidity for these classes.
Section~\ref{sec:countable-orbits} studies the preservation class for countably many orbits.
It proves that its members are uniform homeomorphisms in the
pseudohyperbolic metric, identifies its holomorphic members,
establishes an explicit criterion for boundary perturbations of
disk automorphisms, proves invariance under compact modifications
among disk homeomorphisms, gives nonholomorphic preservers, and
applies the uniform matrix characterization to complex conjugation.
Appendix~\ref{app:geometric-estimates-compactness} collects auxiliary
geometric estimates and a compactness argument.
Appendix~\ref{app:finite-cluster-proofs} contains the detailed proofs of localization and degeneration for finite clusters.
Appendix~\ref{app:countable-observation-localization} supplies the
observation localization lemmas used in the proof of
Theorem~\ref{thm:countable-boundary-perturbation}.

 \section{Preservation of Operator Orbit Frames}
This section establishes equivalent geometric and kernel
characterizations of the preservation of frame generator
sets associated with diagonal operator orbits. We show that
this property is equivalent to the bidirectional preservation of
Carleson sequences together with uniform comparability of boundary
defects, or equivalently, to the corresponding preservation property
for Szeg\H{o} kernels.

For \(z,w\in\mathbb D\), let
\(
\rho(z,w)
:=
\left|
\frac{z-w}{1-\overline{z}w}
\right|
\)
denote the pseudohyperbolic distance on \(\mathbb D\). We first recall
the definition of a Carleson sequence.

\begin{definition}\label{def:carleson-sequence}
Let
\(
\Lambda=\{\lambda_j\}_{j\geq1}\subset\mathbb D.
\)
The sequence \(\Lambda\) is called a Carleson interpolating
sequence, or briefly a Carleson sequence, if
\( \delta_{\Lambda}:=
\inf_{j\geq1}
\prod_{\substack{i\geq1\\ i\neq j}}
\rho(\lambda_i,\lambda_j)
>0.
\)
A Carleson sequence is also said to be \emph{uniformly separated}
with respect to the pseudohyperbolic metric. Let \(\mathcal C\)
denote the collection of all Carleson sequences.
\end{definition}
We also allow finite Carleson sequences, by which we mean finite
indexed families of pairwise distinct points in \(\mathbb D\).
We shall use the following standard characterization of Carleson sequences; see, for example,
\cite{Carleson,Carleson1962,ShapiroShields1961,Garnett,Seip}.

\begin{lem}
\label{thm:carleson-interpolation}
Let
\(
\Lambda=\{\lambda_j\}_{j\ge1}\subset\mathbb D
\). Then \(\Lambda\in\mathcal C\) if and only if the following two
conditions hold:
\begin{enumerate}
    \item[(i)] The sequence \(\Lambda\) is pairwise separated; that is, there exists a
    constant \(\delta>0\) such that for $j\neq k$
    \(
    \rho(\lambda_j,\lambda_k)\geq \delta.
    \)

    \item[(ii)] The discrete measure
    \(
    \mu_\Lambda
    :=
    \sum_{j\ge1}
    \bigl(1-|\lambda_j|^2\bigr)\delta_{\lambda_j}
    \)
    is a Carleson measure on \(\mathbb D\), where
    \(\delta_{\lambda_j}\) denotes the Dirac point mass at
    \(\lambda_j\).
\end{enumerate}

Equivalently, condition~\textup{(ii)} means that there exists a constant
\(C_\Lambda>0\) such that
\(
\mu_\Lambda\bigl(S(I)\bigr)
\leq C_\Lambda |I|
\)
for every arc \(I\subset\mathbb T\), where
\(
S(I)
:=
\left\{
re^{i\theta}\in\mathbb D:
e^{i\theta}\in I,\quad 1-r\leq |I|
\right\}
\)
is the Carleson square associated with \(I\), and \(|I|\) denotes the
normalized arc length of \(I\).
\end{lem}

For \(\lambda\in\mathbb D\), let
\(
k_\lambda(z):=(1-\overline{\lambda}z)^{-1},
\ z\in\mathbb D,
\)
and
\(
\widehat{k}_\lambda
:=
\sqrt{1-|\lambda|^2}\,k_\lambda.
\)
Then
\(
\|k_\lambda\|_{H^2(\mathbb D)}^2
=
(1-|\lambda|^2)^{-1}
\)
and
\(
\|\widehat{k}_\lambda\|_{H^2(\mathbb D)}=1.
\)
We shall use the following standard reproducing kernel
characterization of Carleson interpolating sequences; see, for example,
\cite{ShapiroShields1961,Garnett,Seip}.

\begin{lem}
\label{lem:carleson-riesz-kernel-characterization}
Let
\(
\Lambda=\{\lambda_j\}_{j\geq1}\subset\mathbb D.
\)
Then
\(
\Lambda\in\mathcal C
\)
if and only if
\(
\{\widehat{k}_{\lambda_j}\}_{j\geq1}
\)
is a Riesz sequence in \(H^2(\mathbb D)\).

This characterization remains valid for indexed sequences; in
particular, the presence of a repeated point prevents the normalized
kernel family from being a Riesz sequence.
\end{lem}

\begin{definition}
\label{def:szego-kernel-preservation-property}
Let
\(
\Phi:\mathbb D\to\mathbb D
\)
be an arbitrary mapping. We say that \(\Phi\) has the
\emph{Szeg\H{o} kernel preservation property} if the following two
conditions hold:
\begin{enumerate}
\item[\textup{(i)}]
For every indexed sequence
\(
\Lambda=\{\lambda_j\}_{j\geq1}\subset\mathbb D,
\)
the family
\(
\{\widehat{k}_{\lambda_j}\}_{j\geq1}
\)
is a Riesz sequence in \(H^2(\mathbb D)\) if and only if
\(
\{\widehat{k}_{\Phi(\lambda_j)}\}_{j\geq1}
\)
is a Riesz sequence in \(H^2(\mathbb D)\).

\item[\textup{(ii)}]
There exist constants \(0<m_\Phi\leq M_\Phi<\infty\) such that
\[
m_\Phi\|k_z\|_{H^2(\mathbb D)}
\leq
\|k_{\Phi(z)}\|_{H^2(\mathbb D)}
\leq
M_\Phi\|k_z\|_{H^2(\mathbb D)},
\qquad
z\in\mathbb D.
\]
\end{enumerate}
\end{definition}

The following lemma characterizes frames generated by single orbits of diagonal operators; see \cite[Theorem~3.16]{AldroubiCabrelliMolterTang2017}.

\begin{lem}
\label{thm:normal-operator-orbit-frame-characterization}
 Let \(A\in\mathcal B(H)\) be a diagonal operator and let \(f\in H\). Then
\(
\{A^n f\}_{n\ge0}
\)
is a frame for  \(H\) if and only if there exist an orthonormal basis
\(\{e_j\}_{j\geq 1}\) of \(H\), a Carleson sequence
\(
\Lambda=\{\lambda_j\}_{j\geq 1}\subset\mathbb D
\),
and constants \(0<\alpha\leq\beta<\infty\) satisfying
\[ Ae_j=\lambda_j e_j \ \text{and}\
\alpha\bigl(1-|\lambda_j|^2\bigr)
\leq
\bigl|\langle f,e_j\rangle\bigr|^2
\leq
\beta\bigl(1-|\lambda_j|^2\bigr),
\qquad j\geq 1.
\]
\end{lem}

As an immediate consequence, if \(\Lambda=\{\lambda_j\}_{j\geq 1}\in\mathcal C\) and \(f_j=\langle f,e_j\rangle\) then
\(
\mathcal G(A_\Lambda)
=
\{
f=\sum_{j\geq 1}f_j e_j\in H:
|f_j|^2\asymp 1-|\lambda_j|^2
\}.
\)
Here,
\(
|f_j|^2\asymp 1-|\lambda_j|^2
\)
means that there exist constants \(0<c\leq C<\infty\), independent of
\(j\), such that
\(
c\bigl(1-|\lambda_j|^2\bigr)
\leq
|f_j|^2
\leq
C\bigl(1-|\lambda_j|^2\bigr)\) for \(j\geq 1.\)
If \(\Lambda\notin\mathcal C\), then
\(
\mathcal G(A_{\Lambda})=\varnothing.
\)

The following proposition gives a necessary and
sufficient condition for
\(
\mathcal G(A_\Lambda)
=
\mathcal G\bigl(\Phi(A_\Lambda)\bigr)
\)
when the Carleson sequence \(\Lambda\) is fixed.
\begin{prop}
\label{prop:fixed-spectrum-preservation}
Let
\(
\Lambda=\{\lambda_j\}_{j\geq 1}\in\mathcal C
\), and let
\(\Phi:\mathbb D\rightarrow\mathbb D\) be an arbitrary mapping. Then
\(
\mathcal G(A_\Lambda)
=
\mathcal G\bigl(\Phi(A_\Lambda))
\)
if and only if the following two conditions hold:
\begin{enumerate}
    \item[\textup{(i)}]
    The image sequence
    \(
    \Phi(\Lambda)=\{\Phi(\lambda_j)\}_{j\geq 1}\in\mathcal C
    \).

    \item[\textup{(ii)}]
      There exist constants \(c_\Lambda,C_\Lambda>0\) such that
    \[
    c_\Lambda\bigl(1-|\lambda_j|^2\bigr)
    \leq
    1-|\Phi(\lambda_j)|^2
    \leq
    C_\Lambda\bigl(1-|\lambda_j|^2\bigr),
    \qquad j\geq 1.
    \]
\end{enumerate}
\end{prop}

\begin{proof}
Assume first that
\(
\mathcal G(A_\Lambda)
=
\mathcal G\bigl(\Phi(A_\Lambda)\bigr)
\). Since \(\Lambda\in\mathcal C\), we have
\(
\sum_{j\geq 1}\bigl(1-|\lambda_j|^2\bigr)<\infty.
\)
Define
\(
f
:=
\sum_{j\geq1}
\sqrt{1-|\lambda_j|^2}\,e_j.
\)
Then \(f\in H\), and for $j\geq1$,
\(
|\langle f,e_j\rangle|^2
=
1-|\lambda_j|^2.
\)
By the diagonal characterization following
Lemma~\ref{thm:normal-operator-orbit-frame-characterization},
\(
f\in\mathcal G(A_\Lambda).
\)
Then
\(
f\in\mathcal G\bigl(\Phi(A_\Lambda)\bigr).
\)
Applying
Lemma~\ref{thm:normal-operator-orbit-frame-characterization}
to \(\Phi(A_\Lambda)\), we conclude that all spectral subspaces of
\(\Phi(A_\Lambda)\) have dimension one and that its eigenvalue
sequence is a Carleson sequence. Hence
\(\Phi(\Lambda)\in\mathcal C\). This proves
condition~\textup{(i)}.

The same characterization also yields constants \(c,C>0\) such that
\[
c\bigl(1-|\Phi(\lambda_j)|^2\bigr)
\leq|\langle f,e_j\rangle|^2
=
1-|\lambda_j|^2
\leq
C\bigl(1-|\Phi(\lambda_j)|^2\bigr),
\qquad j\geq 1.
\]
Equivalently,
\(
\frac{1}{C}\bigl(1-|\lambda_j|^2\bigr)
\leq
1-|\Phi(\lambda_j)|^2
\leq
\frac{1}{c}\bigl(1-|\lambda_j|^2\bigr).
\)
Thus condition~\textup{(ii)} holds with
\(
c_\Lambda=C^{-1}
\)
and
\(
C_\Lambda=c^{-1}
\).

We next prove the sufficiency. Assume that
conditions~\textup{(i)} and~\textup{(ii)} hold. For an arbitrary vector
\(
f=\sum_{j\geq 1}f_j e_j
\in H
\),
Lemma~\ref{thm:normal-operator-orbit-frame-characterization} gives
\( f\in\mathcal G(A_\Lambda) \Longleftrightarrow |f_j|^2\asymp 1-|\lambda_j|^2. \)
By condition~\textup{(ii)}, \( 1-|\lambda_j|^2 \asymp 1-|\Phi(\lambda_j)|^2.\) Hence \( |f_j|^2\asymp 1-|\lambda_j|^2 \Longleftrightarrow |f_j|^2\asymp 1-|\Phi(\lambda_j)|^2. \)
Since \(\Phi(\Lambda)\in\mathcal C\), another application of Lemma~\ref{thm:normal-operator-orbit-frame-characterization} yields \( |f_j|^2\asymp 1-|\Phi(\lambda_j)|^2  \Longleftrightarrow f\in\mathcal G\bigl(\Phi(A_\Lambda)\bigr). \) Therefore, \( \mathcal G(A_\Lambda) = \mathcal G\bigl(\Phi(A_\Lambda)\bigr). \)
\end{proof}

We shall use the following standard properties of Carleson sequences.
For \(0<r<1\), set
\(
\overline{D(0,r)}
:=
\{z\in\mathbb C:|z|\leq r\},
\)
the closed disk of radius \(r\) centered at the origin.

\begin{lem}
\label{lem:basic-properties-carleson-sequences}
Given a sequence
\(
\Lambda=\{\lambda_j\}_{j\geq1}\subset\mathbb D,
\)
the following assertions hold:
\begin{enumerate}
\item[\textup{(i)}]
If there exists a constant \(c\in(0,1)\) such that
\(
\frac{1-|\lambda_{j+1}|}{1-|\lambda_j|}
\leq c,
\
j\geq1,
\)
then \(\Lambda\in\mathcal C\); see \cite{Duren1970}.

\item[\textup{(ii)}]
If \(\Lambda\in\mathcal C\), then, for every \(0<r<1\), the closed
disk \(\overline{D(0,r)}\) contains only finitely many points of
\(\Lambda\). Consequently,
\(
\lim_{j\to\infty}|\lambda_j|=1;
\)
see \cite{ShapiroShields1961}.

\item[\textup{(iii)}]
Let \(\Lambda\in\mathcal C\). Then every finite perturbation of
\(\Lambda\) is again a Carleson sequence. More precisely, if
\(F\subset\Lambda\) is finite and
\(E\subset\mathbb D\setminus\Lambda\) is a finite set of pairwise
distinct points, then both
\(
\Lambda\setminus F
\)
and
\(
\Lambda\cup E
\)
belong to \(\mathcal C\); see \cite{CabrelliMolterPaternostroPhilipp2020}.
\end{enumerate}
\end{lem}

 The following theorem characterizes \(\mathcal P\) in terms of the bidirectional preservation of Carleson sequences and the global comparability of boundary defects. It is the
main result of this section.

\begin{thm}
\label{thm:universal-orbit-frame-preservation}
Let
\(
\Phi:\mathbb D\rightarrow\mathbb D
\)
be an arbitrary mapping. Then \(\Phi\in\mathcal P\) if and only if the
following conditions hold:
\begin{enumerate}
\item[\textup{(i)}]
For every sequence
\(
\Lambda=\{\lambda_j\}_{j\geq 1}\subset\mathbb D,
\)
one has
\(
\Lambda\in\mathcal C
\Longleftrightarrow
\Phi(\Lambda)\in\mathcal C.
\)

\item[\textup{(ii)}]
There exist constants \(0<c_\Phi\leq C_\Phi<\infty\) such that
\[
c_\Phi\bigl(1-|z|^2\bigr)
\leq
1-|\Phi(z)|^2
\leq
C_\Phi\bigl(1-|z|^2\bigr),
\qquad z\in\mathbb D.
\]
\end{enumerate}

Equivalently,
\(
\Phi\in\mathcal P
\)
if and only if \(\Phi\) has the Szeg\H{o} kernel preservation
property of Definition~\ref{def:szego-kernel-preservation-property}.
\end{thm}
\begin{proof}
Let \(\Lambda\in\mathcal C\). Since \(\Phi\in\mathcal P\),
\(
\mathcal G(A_\Lambda)
=
\mathcal G\bigl(\Phi(A_\Lambda)\bigr).
\)
Proposition~\ref{prop:fixed-spectrum-preservation} therefore gives
\(\Phi(\Lambda)\in\mathcal C\).

Conversely, let
\(
\Lambda=\{\lambda_j\}_{j\geq1}\subset\mathbb D
\)
and suppose that
\(
\Phi(\Lambda)\in\mathcal C.
\)
Then we have
\(
\sum_{j\geq 1}\bigl(1-|\Phi(\lambda_j)|^2\bigr)<\infty.
\)
Define
\(
g:=
\sum_{j\geq 1}
\sqrt{1-|\Phi(\lambda_j)|^2}\,e_j.
\)
Then \(g\in H\) and
\(
g\in\mathcal G\bigl(\Phi(A_\Lambda)\bigr)
\)
by Lemma~\ref{thm:normal-operator-orbit-frame-characterization}.
Since \(\Phi\in\mathcal P\),
\(g\in\mathcal G(A_\Lambda).
\)
Another application of
Lemma~\ref{thm:normal-operator-orbit-frame-characterization} yields
\(
\Lambda\in\mathcal C.
\)
This proves condition \textup{(i)}.

It remains to prove condition~\textup{(ii)}. Set
\(
q_\Phi(z):=
\frac{1-|\Phi(z)|^2}{1-|z|^2},
\ z\in\mathbb D.
\)

Fix \(0<R<1\). We first show that \(q_\Phi\) is bounded away from zero
on \(\overline{D(0,R)}\). Suppose, to the contrary, that
\(
\inf_{|z|\leq R}q_\Phi(z)=0.
\)
Then there exists a sequence
\(
\{z_n\}_{n\geq 1}\subset\overline{D(0,R)}
\)
such that
\(
q_\Phi(z_n)\rightarrow 0.
\)
Since
\(
1-|\Phi(z_n)|^2
=
q_\Phi(z_n)\bigl(1-|z_n|^2\bigr)
\leq
q_\Phi(z_n),
\)
we have
\(
|\Phi(z_n)|\rightarrow 1.
\)
After passing to a subsequence, we may assume that
\(
1-|\Phi(z_{n+1})|
\leq
\frac{1}{4}\bigl(1-|\Phi(z_n)|\bigr),
\ n\geq 1.
\)

Set
\(
\Lambda:=\{z_n\}_{n\geq1}\). By Lemma~\ref{lem:basic-properties-carleson-sequences}\textup{(i)},
\(
\Phi(\Lambda)\in\mathcal C.
\)
Condition \textup{(i)} gives \(\Lambda\in\mathcal C\). This contradicts
Lemma~\ref{lem:basic-properties-carleson-sequences}\textup{(ii)}, since
\(\Lambda\subset\overline{D(0,R)}\). Therefore,
\(
\inf_{|z|\leq R}q_\Phi(z)>0.
\)

We next show that \(q_\Phi\) is globally bounded from above. Suppose
that
\(
\sup_{z\in\mathbb D}q_\Phi(z)=\infty.
\)
Choose a sequence \(\{z_n\}_{n\geq 1}\subset\mathbb D\)
such that
\(
q_\Phi(z_n)\rightarrow\infty.
\)
Since
\(
q_\Phi(z)
=
\frac{1-|\Phi(z)|^2}{1-|z|^2}
\leq
\frac{1}{1-|z|^2},
\)
we have
\(
|z_n|\rightarrow 1.
\)
After passing to a subsequence, we may assume that
\(
1-|z_{n+1}|
\leq
\frac{1}{4}\bigl(1-|z_n|\bigr),
\ n\geq 1.
\)
Set \(\Lambda=\{z_n\}_{n\ge 1}\). Lemma~\ref{lem:basic-properties-carleson-sequences}\textup{(i)} implies
that \(\Lambda\in\mathcal C\). Since \(\Phi\in\mathcal P\),
Proposition~\ref{prop:fixed-spectrum-preservation} yields a constant \(C_\Lambda<\infty\) such
that for $n\ge1$,\
\(
q_\Phi(z_n)\leq C_\Lambda,
\)
which contradicts \(q_\Phi(z_n)\to\infty\). Therefore, there exists
\(C_\Phi<\infty\) such that
\(
q_\Phi(z)\leq C_\Phi
\)
for all $z\in\mathbb D$.

Finally, we show that \(q_\Phi\) has a global positive lower bound.
Suppose that
\(
\inf_{z\in\mathbb D}q_\Phi(z)=0.
\)
Choose a sequence \(\{z_n\}_{n\geq 1}\subset\mathbb D\)
such that
\(
q_\Phi(z_n)\rightarrow 0.
\)
The local lower bound already proved implies that
\(
|z_n|\rightarrow 1.
\)
After passing to a subsequence, we may assume that
\(
1-|z_{n+1}|
\leq
\frac{1}{4}\bigl(1-|z_n|\bigr),
\ n\geq 1.
\)
Lemma~\ref{lem:basic-properties-carleson-sequences}\textup{(i)} gives
\(
\Lambda:=\{z_n\}_{n\geq 1}\in\mathcal C.
\)
Since \(\Phi\in\mathcal P\),
Proposition~\ref{prop:fixed-spectrum-preservation} yields a constant
\(c_\Lambda>0\) such that, for \(n\geq1\),
\(
q_\Phi(z_n)\geq c_\Lambda,
\)
which contradicts \(q_\Phi(z_n)\to 0\). Thus, there exists
\(c_\Phi>0\) satisfying
\(
q_\Phi(z)\geq c_\Phi
\) for all $z\in\mathbb D$.
Consequently,
\[
c_\Phi\bigl(1-|z|^2\bigr)
\leq
1-|\Phi(z)|^2
\leq
C_\Phi\bigl(1-|z|^2\bigr),
\qquad z\in\mathbb D.
\]
Thus condition~\textup{(ii)} holds.

Conversely, assume that \textup{(i)} and \textup{(ii)} hold. Let
\(
\Lambda=\{\lambda_j\}_{j\geq 1}\subset\mathbb D
\).
If \(\Lambda\notin\mathcal C\), then condition \textup{(i)} gives
\(
\Phi(\Lambda)\notin\mathcal C.
\)
Lemma~\ref{thm:normal-operator-orbit-frame-characterization} then
gives
\(
\mathcal G(A_\Lambda)
=
\varnothing
=
\mathcal G\bigl(\Phi(A_\Lambda)\bigr).
\)

Suppose now that \(\Lambda\in\mathcal C\). By
condition~\textup{(i)},
\(
\Phi(\Lambda)\in\mathcal C.
\)
Moreover, condition~\textup{(ii)} gives
\(
c_\Phi\bigl(1-|\lambda_j|^2\bigr)
\leq
1-|\Phi(\lambda_j)|^2
\leq
C_\Phi\bigl(1-|\lambda_j|^2\bigr)\) for \(j\geq1\).
Proposition~\ref{prop:fixed-spectrum-preservation} therefore yields
\(
\mathcal G(A_\Lambda)
=
\mathcal G\bigl(\Phi(A_\Lambda)\bigr).
\)
Combining the two cases, this equality holds for every sequence \( \Lambda\subset\mathbb D\). Hence \(\Phi\in\mathcal P\).

It remains only to identify the equivalent kernel formulation.
By
Lemma~\ref{lem:carleson-riesz-kernel-characterization},
condition~\textup{(i)} is equivalent to
\(\{\widehat{k}_{\lambda_j}\}_{j\geq1}
\text{ is a Riesz sequence}
\Longleftrightarrow
\{\widehat{k}_{\Phi(\lambda_j)}\}_{j\geq1}
\text{ is a Riesz sequence}
\)
for every indexed sequence
\(
\Lambda=\{\lambda_j\}_{j\geq1}\subset\mathbb D.
\)
Moreover,
\(
\|k_z\|_{H^2(\mathbb D)}
=
\frac{1}{\sqrt{1-|z|^2}},
\)
and therefore condition~\textup{(ii)} is equivalent to
\(
C_\Phi^{-\frac{1}{2}}\|k_z\|_{H^2(\mathbb D)}
\leq
\|k_{\Phi(z)}\|_{H^2(\mathbb D)}
\leq
c_\Phi^{-\frac{1}{2}}\|k_z\|_{H^2(\mathbb D)}\) for \(z\in\mathbb D\).
Hence conditions~\textup{(i)} and~\textup{(ii)} are precisely the
Szeg\H{o} kernel preservation property of
Definition~\ref{def:szego-kernel-preservation-property}.
\end{proof}

\section{Construction and Quantitative Control of Boundary Shadows}

For each \(\Phi\in\mathcal P\), we construct its boundary shadow and
establish its Lipschitz regularity and quantitative estimates preventing collapse. These results prepare the injectivity and bilipschitz
arguments in the next section.

For later use, we introduce the following definitions and formulas. For \(e^{2\pi i\theta},e^{2\pi i\varphi}\in\mathbb T\), define the normalized
arc length distance on \(\mathbb T\) by
\[
d_{\mathbb T}\bigl(e^{2\pi i\theta},e^{2\pi i\varphi}\bigr)
:=\min_{k\in\mathbb Z}
\left|\theta-\varphi+k\right|
\in[0,\frac12].
\]
Let
\(
z=r\zeta
\)
and
\(
w=s\eta,
\)
where
\(
0<r,s<1
\)
and
\(
\zeta,\eta\in\mathbb T.
\)
If
\(
\alpha:=d_{\mathbb T}(\zeta,\eta),
\)
then
\begin{equation}
\label{eq:pseudohyperbolic-distance-polar-form}
\rho(z,w)^2
=
\frac{
(r-s)^2+4rs\sin^2(\pi\alpha)
}{
(1-rs)^2+4rs\sin^2(\pi\alpha)
}.
\end{equation}
\begin{definition}
\label{def:bi-lipschitz-circle-mappings}
A bijection \(h:\mathbb T\rightarrow\mathbb T\) is called
\emph{bilipschitz} if there exists a
constant \(L\geq 1\) such that
\[
L^{-1}
d_{\mathbb T}(\zeta,\eta)
\leq
d_{\mathbb T}\bigl(h(\zeta),h(\eta)\bigr)
\leq
L
d_{\mathbb T}(\zeta,\eta),
\qquad
\zeta,\eta\in\mathbb T.
\]
The set of all such mappings is denoted by
\(
\operatorname{BiLip}(\mathbb T).
\)
\end{definition}
For \(z\in\mathbb D\), set
\(
\delta(z):=1-|z|
\). We shall use the following lemma.

\begin{lem}
\label{lem:pseudohyperbolic-polar-estimates}
Let
\(
\{z_n\}_{n\geq1}
\)
and
\(
\{w_n\}_{n\geq1}
\)
be sequences in \(\mathbb D\).
Then the following assertions hold.

\begin{enumerate}
\item[\textup{(i)}]
If
\(
\rho(z_n,w_n)\rightarrow0,
\
|z_n|\rightarrow1,
\) and \(
|w_n|\rightarrow1,
\)
then
\begin{equation}
\label{eq:asymptotic-boundary-defect-comparability}
\frac{\delta(z_n)}{\delta(w_n)}
\rightarrow1
\end{equation}
and
\begin{equation}
\label{eq:asymptotic-angular-control}
\frac{
d_{\mathbb T}
\left(
\dfrac{z_n}{|z_n|},
\dfrac{w_n}{|w_n|}
\right)
}{
\delta(z_n)
}
\rightarrow0.
\end{equation}

\item[\textup{(ii)}]
For every \(M\in(0,1)\), there exists a constant
\(A_M<\infty\), depending only on \(M\), such that whenever
\(
\rho(z,w)\leq M,
\
|z|,|w|\geq\frac{1}{2},
\)
one has
\begin{equation}
\label{eq:bounded-pseudohyperbolic-radial-control}
A_M^{-1}\delta(z)
\leq
\delta(w)
\leq
A_M\delta(z)
\end{equation}
and
\begin{equation}
\label{eq:bounded-pseudohyperbolic-angular-control}
d_{\mathbb T}
\left(
\frac{z}{|z|},
\frac{w}{|w|}
\right)
\leq
A_M\delta(z).
\end{equation}
\end{enumerate}
\end{lem}

\noindent\emph{Proof outline.}
The polar formula~\eqref{eq:pseudohyperbolic-distance-polar-form}
separates the radial and angular contributions to the distance.
When the distance tends to zero, both contributions are negligible
relative to \((1-|z||w|)^2\); when it is at most \(M<1\), the
same identity gives uniform bounds for both contributions.
The resulting boundary defect comparisons and angular estimates
are proved in Appendix~\ref{app:geometric-estimates-compactness}.

We shall use the following sparse packet criterion involving two points per packet for constructing Carleson sequences.

\begin{lem}
\label{lem:general-sparse-two-point-packets}
Let
\(
u_n,v_n\in\mathbb D
\)
for \(n\geq1\). Suppose that there exist a sequence
\(
s_n\downarrow0
\)
and constants
\(
0<a\leq b<\infty,
\
\varepsilon>0,
\
q\in(0,1),
\)
such that
\(
a s_n
\leq
\delta(u_n),\delta(v_n)
\leq
b s_n,
\)
\(
s_{n+1}\leq q s_n,
\)
and
\(
\rho(u_n,v_n)\geq\varepsilon.
\)
If \(q\leq\frac{a}{2b}\),
\(
\Gamma
:=
\{u_n,v_n:n\geq1\}
\in\mathcal C.
\)
\end{lem}

\begin{proof}
We verify the pairwise separation and the Carleson measure condition
for \(\Gamma\).
We first prove pairwise separation. Since \(\rho(u_n,v_n)\geq\varepsilon,\)
it remains to compare points belonging to different levels. Let
\(m>n\), and choose
\(
z\in\{u_n,v_n\},
w\in\{u_m,v_m\}.
\)
Write
\(
x:=\delta(z)=1-|z|,
y:=\delta(w)=1-|w|.
\)
By \(s_{n+1}\leq q s_n\), we have
\(
s_m\leq q^{m-n}s_n\leq q s_n.
\)
Consequently,
\(
x\geq a s_n,
y\leq b s_m\leq bq s_n.
\)
Under assumption
\(q\leq\frac{a}{2b}\), we have
\(
y\leq bq s_n\leq\frac{a}{2}s_n<x,
\)
and hence \(|z|<|w|\).

For arbitrary points \(z,w\in\mathbb D\), one has
\begin{equation}
\rho(z,w)
\geq
\frac{\bigl||z|-|w|\bigr|}
{1-|z||w|}.
\label{eq:pseudohyperbolic-distance-dominates-radial-distance}
\end{equation}
Indeed, if \(zw=0\), the inequality is immediate. Otherwise, let
\(r=|z|\), \(t=|w|\), and let \(B\geq0\) denote the angular term in
\eqref{eq:pseudohyperbolic-distance-polar-form}. Then
\(
\rho(z,w)^2
=
\frac{(r-t)^2+B}{(1-rt)^2+B}
\geq
\frac{(r-t)^2}{(1-rt)^2},
\)
because
\(
(1-rt)^2-(r-t)^2
=
(1-r^2)(1-t^2)
\geq0.
\)

Since
\(
\bigl||z|-|w|\bigr|=x-y
\)
and
\(
1-|z||w|
=
1-(1-x)(1-y)
=
x+y-xy,
\)
we obtain from
\eqref{eq:pseudohyperbolic-distance-dominates-radial-distance} that
\(
\rho(z,w)
\geq
\frac{x-y}{x+y-xy}
\geq
\frac{x-y}{x+y}.
\)
Using the bounds for \(x\) and \(y\), we find
\(
\rho(z,w)
\geq
\frac{a s_n-bq s_n}
{b s_n+bq s_n}
=
\frac{a-bq}{b(1+q)}.
\)
Since \(q\leq \frac{a}{2b}\) and \(a\leq b\), one has
\(
a-bq\geq\frac{a}{2},
1+q\leq2.
\)
Therefore,
\(
\rho(z,w)\geq\frac{a}{4b}.
\)
Thus every two distinct points \(z,w\in\Gamma\) satisfy
\(
\rho(z,w)
\geq
\min\left\{
\varepsilon,\frac{a}{4b}
\right\}
>0.
\)
Hence \(\Gamma\) is pairwise separated.

We next verify the Carleson measure condition. Let
\(I\subset\mathbb T\) be an arbitrary arc. If one of the points
\(u_n\) or \(v_n\) belongs to \(S(I)\), then
\(
\delta(u_n)\leq |I|
\ \text{or}\
\delta(v_n)\leq |I|.
\)
In either case,
\(
a s_n\leq |I|,
\)
and therefore
\(
s_n\leq\frac{|I|}{a}.
\)

At the \(n\)th level, the total contribution to the discrete measure
associated with \(\Gamma\) satisfies
\(
\bigl(1-|u_n|^2\bigr)
+
\bigl(1-|v_n|^2\bigr)
\leq
2\delta(u_n)+2\delta(v_n)
\leq
4b s_n.
\)
Since \(s_n\to0\) and \(|I|>0\), the index
\(
N:=\min\left\{n\ge1:s_n\le\frac{|I|}{a}\right\}
\)
is well defined. Every level contributing to \(S(I)\) has index \(n\ge N\). Moreover, for \(k\ge0\) we have
\(
s_{N+k}\leq q^k s_N.
\)
It follows that
\[
\mu_\Gamma\bigl(S(I)\bigr)
\leq
4b\sum_{n=N}^{\infty}s_n
\leq
4b s_N\sum_{k=0}^{\infty}q^k
=
\frac{4b}{1-q}s_N
\leq
\frac{4b}{a(1-q)}|I|.
\]
Thus \(\mu_\Gamma\) is a Carleson measure.

Since \(\Gamma\) is pairwise separated and its associated discrete
measure is a Carleson measure, Lemma~\ref{thm:carleson-interpolation}
yields
\(
\Gamma\in\mathcal C.
\)
\end{proof}

\begin{prop}
\label{prop:basic-properties-of-universal-preservers}
Let \(\Phi\in\mathcal P\). Then the following assertions hold:
\begin{enumerate}
\item[\textup{(i)}]
The mapping \(\Phi\) is injective.

\item[\textup{(ii)}]
For any sequences
\(
\{z_n\}_{n\geq1},
\{w_n\}_{n\geq1}
\subset\mathbb D
\)
satisfying
\(
|z_n|\to1
\)
and
\(
|w_n|\to1,
\)
one has
\[
\rho(z_n,w_n)\to0
\quad\Longleftrightarrow\quad
\rho\bigl(\Phi(z_n),\Phi(w_n)\bigr)\to0.
\]
\end{enumerate}
\end{prop}
\begin{proof}
We first prove \textup{(i)}. Suppose that
\(
\Phi(z)=\Phi(w)
\)
for some distinct \(z,w\in\mathbb D\). Choose a Carleson sequence
\(
\Gamma\subset\mathbb D\setminus\{z,w\}.
\)
By Lemma~\ref{lem:basic-properties-carleson-sequences}(iii),
\(
\Lambda:=\Gamma\cup\{z,w\}\in\mathcal C.
\)
However, \(\Phi(\Lambda)\) contains the repeated point
\(\Phi(z)=\Phi(w)\), and hence
\(
\Phi(\Lambda)\notin\mathcal C.
\)
This contradicts
Theorem~\ref{thm:universal-orbit-frame-preservation}\textup{(i)}.
Therefore, \(\Phi\) is injective.

We next prove \textup{(ii)}. By Theorem~\ref{thm:universal-orbit-frame-preservation}\textup{(ii)}
and the identity
\(
1-|z|^2=(1+|z|)\delta(z),
\)
there exist constants \(A,B>0\) such that
\begin{equation}
\label{eq:boundary-defect-comparability-for-P}
A\delta(z)
\leq
\delta\bigl(\Phi(z)\bigr)
\leq
B\delta(z),
\qquad z\in\mathbb D.
\end{equation}
In particular,
\(
|z_n|\to1
\Longleftrightarrow
|\Phi(z_n)|\to1.
\)
Assume first that
\(
\rho(z_n,w_n)\to0.
\)
Suppose, to the contrary, that
\(
\rho(\Phi(z_n),\Phi(w_n))\not\to0.
\)
After passing to a subsequence, there exists \(\varepsilon>0\) such that
\(
\rho\bigl(\Phi(z_n),\Phi(w_n)\bigr)\geq\varepsilon
\) for \(n\geq1\).
By Lemma~\ref{lem:pseudohyperbolic-polar-estimates}\textup{(i)},
\(
\frac{\delta(w_n)}{\delta(z_n)}\to1.
\)
Hence, after deleting finitely many terms,
\(
\frac{1}{2}\delta(z_n)
\leq
\delta(w_n)
\leq
2\delta(z_n).
\)
Set
\(
s_n:=\delta(z_n).
\)
Using \eqref{eq:boundary-defect-comparability-for-P}, we obtain
\(
\frac{A}{2}s_n
\leq
\delta\bigl(\Phi(z_n)\bigr),
\delta\bigl(\Phi(w_n)\bigr)
\leq
2Bs_n.
\)
Passing to a further subsequence, we may assume that
\(
s_{n+1}\leq q s_n,
\)
where \(q>0\) is sufficiently small for
Lemma~\ref{lem:general-sparse-two-point-packets}. That lemma gives
\(
\{\Phi(z_n),\Phi(w_n):n\geq1\}\in\mathcal C.
\)
On the other hand,
\(
\rho(z_n,w_n)\to0.
\)
Since
\(
\rho\bigl(\Phi(z_n),\Phi(w_n)\bigr)\geq\varepsilon
\)
and \(\Phi\) is injective, one has \(z_n\neq w_n\) for every \(n\).
Hence the sequence
\(
\{z_n,w_n:n\geq1\}
\)
is not pairwise separated and therefore does not belong to
\(\mathcal C\). This contradicts
\(
\{\Phi(z_n),\Phi(w_n):n\geq1\}\in\mathcal C
\)
and the bidirectional preservation of Carleson sequences.
Therefore,
\(
\rho(z_n,w_n)\to0
\Longrightarrow
\rho\bigl(\Phi(z_n),\Phi(w_n)\bigr)\to0.
\)

Conversely, assume that
\(
\rho\bigl(\Phi(z_n),\Phi(w_n)\bigr)\to0.
\)
Suppose that
\(
\rho(z_n,w_n)\not\to0.
\)
After passing to a subsequence, there exists \(\varepsilon>0\) such that
\(
\rho(z_n,w_n)\geq\varepsilon
\) for \(n\geq1\).
Since
\(
|\Phi(z_n)|,|\Phi(w_n)|\to1,
\)
Lemma~\ref{lem:pseudohyperbolic-polar-estimates}\textup{(i)} yields
\(
\frac{
\delta(\Phi(w_n))
}{
\delta(\Phi(z_n))
}
\to1.
\)
Set
\(
s_n:=\delta(\Phi(z_n)).
\)
After deleting finitely many terms and using
\eqref{eq:boundary-defect-comparability-for-P}, we obtain constants
\(a,b>0\), independent of \(n\), such that
\(
a s_n
\leq
\delta(z_n),\delta(w_n)
\leq
b s_n.
\)
Passing to a further subsequence, we may again assume that
\(
s_{n+1}\leq q s_n
\)
with \(q>0\) sufficiently small. Hence
Lemma~\ref{lem:general-sparse-two-point-packets} gives
\(
\{z_n,w_n:n\geq1\}\in\mathcal C.
\)
But
\(
\rho(\Phi(z_n),\Phi(w_n))\to0,
\)
so its image is not pairwise separated and therefore does not belong to
\(\mathcal C\), again a contradiction.

Thus,
\(
\rho(z_n,w_n)\to0
\Longleftrightarrow
\rho\bigl(\Phi(z_n),\Phi(w_n)\bigr)\to0.
\)
\end{proof}
From Proposition \ref{prop:basic-properties-of-universal-preservers}(ii), we readily obtain the following corollary.
\begin{cor}
\label{cor:uniform-small-scale-control}
Let \(\Phi\in\mathcal P\). For every \(\varepsilon>0\), there exist
\(r_\varepsilon\in(0,1)\) and \(\eta_\varepsilon>0\) such that, whenever
\(z,w\in\mathbb D\) satisfy
\(
|z|,|w|>r_\varepsilon,
\) the following implications hold:
\begin{enumerate}
\item[\textup{(i)}]
\(
\rho(z,w)<\eta_\varepsilon
\quad\Longrightarrow\quad
\rho\bigl(\Phi(z),\Phi(w)\bigr)<\varepsilon.
\)
\item[\textup{(ii)}]
\(
\rho\bigl(\Phi(z),\Phi(w)\bigr)<\eta_\varepsilon
\quad\Longrightarrow\quad
\rho(z,w)<\varepsilon.
\)
\end{enumerate}
\end{cor}
For \(z\in\mathbb D\) sufficiently close to \(\mathbb T\),
\eqref{eq:boundary-defect-comparability-for-P} implies that
\(
\Phi(z)\neq 0.
\)
Hence we may define
\(
\widehat{\Phi}(z)
:=
\frac{\Phi(z)}{|\Phi(z)|}
\in\mathbb T.
\)
The following proposition associates with each \(\Phi\in\mathcal P\) a boundary shadow and gives a uniform estimate of first order for its angular behavior near \(\mathbb T\).

\begin{prop}
\label{prop:existence-of-boundary-shadow}
Let
\(
\Phi\in\mathcal P.
\)
Then there exists a unique mapping
\(
h_\Phi:\mathbb T\rightarrow\mathbb T
\)
and there exist constants \(A<\infty\) and \(t_0>0\) such that
\begin{equation}
\label{eq:boundary-shadow-first-order-error}
d_{\mathbb T}
\left(
\widehat{\Phi}\bigl((1-t)\zeta\bigr),
h_\Phi(\zeta)
\right)
\leq At,
\qquad
\zeta\in\mathbb T,\quad 0<t<t_0.
\end{equation}
In particular,
\(
h_\Phi(\zeta)
=
\lim_{t\to0^+}
\widehat{\Phi}\bigl((1-t)\zeta\bigr)
\) for \(\zeta\in\mathbb T\)
and the convergence is uniform with respect to
\(\zeta\in\mathbb T\).
\end{prop}
\begin{proof}
By \eqref{eq:boundary-defect-comparability-for-P}, there exists
\(B>0\) such that
\begin{equation}
\label{eq:boundary-defect-upper-bound-for-shadow}
\delta\bigl(\Phi(z)\bigr)
\leq
B\delta(z),
\qquad z\in\mathbb D.
\end{equation}

Fix \(M=\frac{1}{2}\). By
Corollary~\ref{cor:uniform-small-scale-control}\textup{(i)}, there exist
\(r_*\in(0,1)\) and \(\eta_*>0\) such that
\[
|z|,|w|>r_*,
\qquad
\rho(z,w)<\eta_*
\quad\Longrightarrow\quad
\rho\bigl(\Phi(z),\Phi(w)\bigr)<M.
\]
Choose \(\lambda\in(0,1)\) so close to \(1\) that
\(
\frac{1-\lambda}{\lambda}<\eta_*,
\)
and choose
\(
0<t_0<
\min\left\{1-r_*,\frac{1}{2B}\right\}.
\)

For \(0<t\leq t_0\) and \(\zeta\in\mathbb T\), define
\(
F_t(\zeta)
:=
\widehat{\Phi}\bigl((1-t)\zeta\bigr).
\)
This is well defined, since
\(
\delta\bigl(\Phi((1-t)\zeta)\bigr)
\leq Bt<\frac{1}{2}.
\)

Suppose that
\(
0<a\leq b\leq t_0,
a\geq\lambda b.
\)
Since \((1-a)\zeta\) and \((1-b)\zeta\) lie on the same radius,
\[
\rho\bigl((1-a)\zeta,(1-b)\zeta\bigr)
=
\frac{b-a}{a+b-ab}
\leq
\frac{b-a}{a}
\leq
\frac{1-\lambda}{\lambda}
<
\eta_*.
\]
Hence
\(
\rho
\left(
\Phi((1-a)\zeta),
\Phi((1-b)\zeta)
\right)
<M.
\)
Lemma~\ref{lem:pseudohyperbolic-polar-estimates}\textup{(ii)} and
\eqref{eq:boundary-defect-upper-bound-for-shadow} therefore give
\begin{equation}
\label{eq:boundary-shadow-increment}
d_{\mathbb T}\bigl(F_a(\zeta),F_b(\zeta)\bigr)
\leq
C\,a,
\end{equation}
where \(C>0\) is independent of \(a,b\) and \(\zeta\).

For \(n\geq0\), set
\(
t_n:=\lambda^n t_0.
\)
By \eqref{eq:boundary-shadow-increment},
\(
d_{\mathbb T}
\bigl(F_{t_{n+1}}(\zeta),F_{t_n}(\zeta)\bigr)
\leq
Ct_{n+1}.
\)
Thus, for \(m>n\),
\(
d_{\mathbb T}
\bigl(F_{t_m}(\zeta),F_{t_n}(\zeta)\bigr)
\leq
C\sum_{k=n}^{m-1}t_{k+1}
\leq
\frac{C\lambda}{1-\lambda}t_n.
\)
Hence \(\{F_{t_n}\}\) is uniformly Cauchy on \(\mathbb T\). Since
\(\mathbb T\) is complete, there exists a mapping
\(h_{\Phi}:\mathbb T\to\mathbb T\) such that
\(
h_{\Phi}(\zeta)
=
\lim_{n\to\infty}F_{t_n}(\zeta),
\)
uniformly in \(\zeta\). Letting \(m\to\infty\), we obtain
\begin{equation}
\label{eq:boundary-shadow-geometric-error}
d_{\mathbb T}\bigl(F_{t_n}(\zeta),h_{\Phi}(\zeta)\bigr)
\leq
\frac{C\lambda}{1-\lambda}t_n.
\end{equation}

Now let \(0<t<t_0\), and choose \(n\geq0\) such that
\(
t_{n+1}\leq t\leq t_n.
\)
Since \(t\geq\lambda t_n\), \eqref{eq:boundary-shadow-increment} gives
\(
d_{\mathbb T}\bigl(F_t(\zeta),F_{t_n}(\zeta)\bigr)
\leq Ct.
\)
Moreover, by \eqref{eq:boundary-shadow-geometric-error} and
\(\lambda t_n\leq t\),
\(
d_{\mathbb T}\bigl(F_{t_n}(\zeta),h_{\Phi}(\zeta)\bigr)
\leq
\frac{C}{1-\lambda}t.
\)
Therefore,
\(
d_{\mathbb T}
\left(
\widehat{\Phi}\bigl((1-t)\zeta\bigr),
h_{\Phi}(\zeta)
\right)
\leq
C\left(1+\frac{1}{1-\lambda}\right)t.
\)
This proves the assertion with
\(
A
:=
C\left(1+\frac{1}{1-\lambda}\right).
\)
The same estimate also shows that the convergence
\(
\widehat{\Phi}\bigl((1-t)\zeta\bigr)
\rightarrow
h_{\Phi}(\zeta)
\)
is uniform in \(\zeta\in\mathbb T\) as \(t\to0^+\). Since any mapping satisfying
\eqref{eq:boundary-shadow-first-order-error}
must agree with this radial limit, the mapping \(h_\Phi\) is unique.
\end{proof}
Let
\(
\Phi:\mathbb D\rightarrow\mathbb D.
\)
Suppose that, for every \(\zeta\in\mathbb T\),
\(
\Phi((1-t)\zeta)\neq0
\)
for all sufficiently small \(t>0\), and that the limit
\(
\lim_{t\to0^+}
\widehat{\Phi}\bigl((1-t)\zeta\bigr)
\)
exists. We then define
\(
h_\Phi(\zeta)
:=
\lim_{t\to0^+}
\widehat{\Phi}\bigl((1-t)\zeta\bigr)
\)
and call \(h_\Phi\) the boundary shadow of \(\Phi\).
For mappings in \(\mathcal P\), these conditions follow from
Proposition~\ref{prop:existence-of-boundary-shadow}. This yields the following corollary.

\begin{cor}
\label{cor:radial-boundary-trace}
Let \(\Phi\in\mathcal P\). Then there exist constants \(C>0\) and \(t_0>0\) such that
\[
\left|
\Phi((1-t)\zeta)-h_\Phi(\zeta)
\right|
\leq Ct,
\qquad
\zeta\in\mathbb T,\quad 0<t<t_0.
\]
In particular,
\(
h_\Phi(\zeta)
=
\lim_{t\to0^+}\Phi((1-t)\zeta)
\)
uniformly for \(\zeta\in\mathbb T\). Thus \(h_\Phi\) is also the
canonical radial boundary trace of \(\Phi\).
\end{cor}
\begin{proof}
Let \(A<\infty\) and \(t_1>0\) be as in
Proposition~\ref{prop:existence-of-boundary-shadow}.
Since \(\Phi\in\mathcal P\),
Theorem~\ref{thm:universal-orbit-frame-preservation} gives a
constant \(C_0>0\) such that
\(
1-|\Phi(z)|^2
\leq
C_0\bigl(1-|z|^2\bigr),
\ z\in\mathbb D.
\)
Set
\(
t_0:=\min\left\{t_1,\frac{1}{4C_0}\right\}.
\)
Fix \(\zeta\in\mathbb T\) and \(0<t<t_0\), and set
\(
z_t:=(1-t)\zeta.
\)
Then
\[
1-|\Phi(z_t)|
\leq
1-|\Phi(z_t)|^2
\leq
C_0\bigl(1-|z_t|^2\bigr)
=
C_0(2t-t^2)
\leq
2C_0t
<
\frac12.
\]
In particular, \(\Phi(z_t)\neq0\). Hence
\(
\widehat{\Phi}(z_t)
=
\frac{\Phi(z_t)}{|\Phi(z_t)|}
\in\mathbb T,
\)
and therefore
\[
\left|
\Phi(z_t)-\widehat{\Phi}(z_t)
\right|
=
\left|
|\Phi(z_t)|\widehat{\Phi}(z_t)
-
\widehat{\Phi}(z_t)
\right|
=
1-|\Phi(z_t)|
\leq
2C_0t.
\]

For \(\xi,\eta\in\mathbb T\), one has
\(
|\xi-\eta|
=
2\sin\!\left(
\pi d_{\mathbb T}(\xi,\eta)
\right)
\leq
2\pi d_{\mathbb T}(\xi,\eta).
\)
Thus, by
Proposition~\ref{prop:existence-of-boundary-shadow},
\(
|
\widehat{\Phi}(z_t)-h_\Phi(\zeta)
|
\leq
2\pi
d_{\mathbb T}
\left(
\widehat{\Phi}(z_t),
h_\Phi(\zeta)
\right)
\leq
2\pi At.
\)
Combining the two estimates, we obtain
\[
\begin{aligned}
\left|
\Phi\bigl((1-t)\zeta\bigr)-h_\Phi(\zeta)
\right|
&\leq
\left|
\Phi(z_t)-\widehat{\Phi}(z_t)
\right|
+
\left|
\widehat{\Phi}(z_t)-h_\Phi(\zeta)
\right|\\
&\leq
2C_0t+2\pi At
=
2(C_0+\pi A)t.
\end{aligned}
\]
Hence the conclusion holds with
\(
C:=2(C_0+\pi A).
\)
Since \(C\) is independent of \(\zeta\), the convergence
\(
\Phi\bigl((1-t)\zeta\bigr)
\rightarrow
h_\Phi(\zeta)
\)
is uniform for \(\zeta\in\mathbb T\) as \(t\to0^+\).
\end{proof}

\begin{prop}
\label{prop:boundary-shadow-is-lipschitz}
Let \(\Phi\in\mathcal P\). Then \(h_{\Phi}\) is Lipschitz.
\end{prop}
\begin{proof}
Let \(A>0\) and \(t_0>0\) be as in
Proposition~\ref{prop:existence-of-boundary-shadow}, so that
\begin{equation}
\label{eq:shadow-error-for-lipschitz}
d_{\mathbb T}
\left(
\widehat{\Phi}\bigl((1-t)\zeta\bigr),
h_{\Phi}(\zeta)
\right)
\leq At
\end{equation}
for \(\zeta\in\mathbb T\) and \(0<t<t_0\).
Moreover, by
\eqref{eq:boundary-defect-comparability-for-P}, there exists \(B>0\)
such that for $z\in\mathbb D$
\(
\delta\bigl(\Phi(z)\bigr)\leq B\delta(z).
\)

Fix \(M=\frac{1}{2}\). By
Corollary~\ref{cor:uniform-small-scale-control}\textup{(i)}, there exist
\(r_*\in(0,1)\) and \(\eta_*>0\) such that
\[
|z|,|w|>r_*,
\qquad
\rho(z,w)<\eta_*
\quad\Longrightarrow\quad
\rho\bigl(\Phi(z),\Phi(w)\bigr)<M.
\]
Choose \(K>0\) with
\(
\frac{2\pi}{K}<\eta_*,
\)
and choose
\(
0<\alpha_0<
\frac{1}{K}
\min\left\{
t_0,\,
1-r_*,\,
\frac{1}{2B}
\right\}.
\)

Let
\(
\alpha=d_{\mathbb T}(\zeta,\eta)
\)
and suppose first that
\(
0<\alpha<\alpha_0.
\)
Set
\(
t:=K\alpha,
z:=(1-t)\zeta,
w:=(1-t)\eta.
\)
Then \(|z|,|w|>r_*\), and
\(
\rho(z,w)
\leq
\frac{|z-w|}{1-(1-t)^2}
\leq
\frac{2\pi\alpha}{t}
=
\frac{2\pi}{K}
<
\eta_*.
\)
Hence
\(
\rho(\Phi(z),\Phi(w))<M.
\)

Since
\(
\delta\bigl(\Phi(z)\bigr),
\delta\bigl(\Phi(w)\bigr)
\leq Bt<\frac{1}{2},
\)
Lemma~\ref{lem:pseudohyperbolic-polar-estimates}\textup{(ii)} gives
\(
d_{\mathbb T}
\left(
\widehat{\Phi}(z),
\widehat{\Phi}(w)
\right)
\leq
A_MB t.
\)
Using \eqref{eq:shadow-error-for-lipschitz}, we obtain
\(
d_{\mathbb T}\bigl(h_{\Phi}(\zeta),h_{\Phi}(\eta)\bigr)
\leq
(2A+A_MB)t
=
K(2A+A_MB)\alpha.
\)

If \(\alpha\geq\alpha_0\), then
\(
d_{\mathbb T}\bigl(h_{\Phi}(\zeta),h_\Phi(\eta)\bigr)
\leq
\frac{1}{2}
\leq
\frac{\alpha}{2\alpha_0}.
\)
Therefore,
\[
d_{\mathbb T}\bigl(h_{\Phi}(\zeta),h_{\Phi}(\eta)\bigr)
\leq
L\,d_{\mathbb T}(\zeta,\eta),
\qquad
\zeta,\eta\in\mathbb T,
\]
where
\(
L
:=
\max\left\{
K(2A+A_MB),
\frac{1}{2\alpha_0}
\right\}.
\)
Thus \(h_{\Phi}\) is Lipschitz.
\end{proof}

\begin{prop}
\label{prop:non-collapse-of-arcs}
Let \(\Phi\in\mathcal P\).
There exist constants
\(
\kappa>0
\)
and
\(
\ell_0>0
\)
such that, for every arc
\(
I\subset\mathbb T
\)
with
\(
|I|\leq\ell_0,
\)
one has
\(
\operatorname{diam}_{\mathbb T}\bigl(h_{\Phi}(I)\bigr)
\geq
\kappa |I|.
\)
Here, for \(E\subset\mathbb T\), define
\(
\operatorname{diam}_{\mathbb T}(E)
:=
\sup_{\xi,\eta\in E}
d_{\mathbb T}(\xi,\eta).
\)
\end{prop}
\begin{proof}
Suppose, to the contrary, that the conclusion fails. Choose a decreasing sequence
\(
\{\varepsilon_n\}_{n\geq1}
\)
such that
\(
0<\varepsilon_n\leq\frac{1}{4}\) and
\(\varepsilon_n\rightarrow0.
\) Recursively select arcs
\(I_n\subset\mathbb T\) such that, writing \(\ell_n:=|I_n|\le\frac{1}{4}\),
\(
\operatorname{diam}_{\mathbb T}h_{\Phi}(I_n)
\leq
\varepsilon_n^2\ell_n,
\ell_{n+1}\leq\frac{1}{8}\varepsilon_n\ell_n.
\)
Set
\(
t_n:=\varepsilon_n\ell_n.
\)
Then \(t_{n+1}\leq \ell_{n+1}\leq \frac{t_n}{8}\).

Let \(\xi_n\) be the midpoint of \(I_n\). Choose points
\(
\zeta_{n,1},\ldots,\zeta_{n,N_n}\in I_n
\)
that are equally spaced with respect to normalized arc length, with
successive spacing comparable to \(t_n\). Then
\(
N_n\asymp\frac{\ell_n}{t_n}
=
\frac{1}{\varepsilon_n},
d_{\mathbb T}(\zeta_{n,j},\zeta_{n,k})
\gtrsim |j-k|t_n.
\)
Put
\(
z_{n,k}:=(1-t_n)\zeta_{n,k},
\Lambda:=\{z_{n,k}:n\geq1,\ 1\leq k\leq N_n\}.
\)

If \(j\neq k\), then
\(
d_{\mathbb T}(\zeta_{n,j},\zeta_{n,k})
\gtrsim t_n.
\)
Hence, using the polar formula for the pseudohyperbolic distance and
\(1-(1-t_n)^2\asymp t_n\), we obtain
\(
\rho(z_{n,j},z_{n,k})
\gtrsim 1.
\)
Thus points on the same level are uniformly pseudohyperbolically
separated.
If \(m>n\), then \(t_m\le \frac{t_n}{8}\), and therefore
\(
\rho(z_{n,j},z_{m,k})
\ge
\frac{t_n-t_m}{t_n+t_m-t_nt_m}
\ge
\frac{7}{9}.
\)
Hence points on different levels are also pairwise separated.
Consequently, the family $\Lambda$
is pairwise separated.

Moreover, let \(J\subset\mathbb T\) be an arc. If the \(n\)th level
meets \(S(J)\), then necessarily \(t_n\leq |J|\). Since the points
\(\zeta_{n,k}\) are separated by a constant multiple of \(t_n\), the
number of points of the \(n\)th level whose angular parts lie in \(J\)
is at most
\(
C\left(1+\frac{\min\{|J|,\ell_n\}}{t_n}\right).
\)
As
\(
1-|z_{n,k}|^2
=
1-(1-t_n)^2
\leq 2t_n,
\)
and \(t_n\leq\min\{|J|,\ell_n\}\) whenever the \(n\)th level
contributes, we obtain
\(
\mu_{\Lambda_n}(S(J))
\leq
C\min\{|J|,\ell_n\},
\)
where
\(
\Lambda_n:=\{z_{n,k}:1\leq k\leq N_n\}.
\)

If no \(n\) satisfies \(t_n\leq |J|\), then
\(\mu_\Lambda(S(J))=0\). Otherwise, let
\(
N:=\min\{n:t_n\leq |J|\}.
\)
Then
\(
\mu_\Lambda(S(J))
\leq
C\sum_{n\geq N}\min\{|J|,\ell_n\}.
\)
Since
\(
\ell_{N+1}\leq \frac{t_N}{8}\leq\frac{|J|}{8}
\), after assuming \(\varepsilon_n\leq1\),
\(
\ell_{n+1}\leq\frac{\ell_n}{8},
\)
we have
\(
\sum_{n\geq N}\min\{|J|,\ell_n\}
\leq
|J|+\sum_{n\geq N+1}\ell_n
\leq
|J|+\frac{\ell_{N+1}}{1-\frac{1}{8}}
\leq
\frac{8}{7}|J|.
\)
Hence
\(
\mu_\Lambda(S(J))\leq C|J|.
\)
Thus \(\mu_\Lambda\) is a Carleson measure. Since \(\Lambda\) is
pairwise separated, it follows that
\(
\Lambda\in\mathcal C.
\)

Let \(\xi_n\) be the midpoint of \(I_n\). Since
\(\zeta_{n,k}\in I_n\), Proposition~\ref{prop:existence-of-boundary-shadow}
and the estimate for \(\operatorname{diam}_{\mathbb T}h_{\Phi}(I_n)\) give
\[
\begin{aligned}
d_{\mathbb T}
\left(
\widehat{\Phi}(z_{n,k}),h_{\Phi}(\xi_n)
\right)
&\leq
d_{\mathbb T}
\left(
\widehat{\Phi}(z_{n,k}),h_{\Phi}(\zeta_{n,k})
\right)
+
d_{\mathbb T}
\left(
h_{\Phi}(\zeta_{n,k}),h_{\Phi}(\xi_n)
\right)\\
&\leq
At_n+\varepsilon_n t_n
\leq
(A+1)t_n,
\end{aligned}
\]
after assuming \(\varepsilon_n\leq1\). Moreover, by
\eqref{eq:boundary-defect-comparability-for-P}, there exists \(B>0\)
such that
\(
1-|\Phi(z_{n,k})|
\leq Bt_n.
\)
Choose
\(
C_0\geq\max\{2(A+1),B\},
\)
and, for all sufficiently large \(n\), let \(J_n\) be the arc centered
at \(h_{\Phi}(\xi_n)\) with
\(
|J_n|=C_0t_n.
\)
Then
\(
\widehat{\Phi}(z_{n,k})\in J_n\) and \(1-|\Phi(z_{n,k})|\leq |J_n|,
\)
so that
\(
\Phi(z_{n,k})\in S(J_n)
\).
Clearly, \(|J_n|\asymp t_n\).

Since
\(
\Phi(z_{n,k})\in S(J_n)
\)
for \(1\leq k\leq N_n\), we have
\(
\mu_{\Phi(\Lambda)}(S(J_n))
\geq
\sum_{k=1}^{N_n}
\bigl(1-|\Phi(z_{n,k})|^2\bigr).
\)
By \eqref{eq:boundary-defect-comparability-for-P},
\(
1-|\Phi(z_{n,k})|^2
\gtrsim
1-|z_{n,k}|^2
=
1-(1-t_n)^2
\asymp t_n.
\)
Hence
\(
\mu_{\Phi(\Lambda)}(S(J_n))
\gtrsim
N_nt_n.
\)
Since \(|J_n|\asymp t_n\) and
\(
N_n\asymp\frac{\ell_n}{t_n}
=
\frac{1}{\varepsilon_n},
\)
it follows that
\(
\frac{\mu_{\Phi(\Lambda)}(S(J_n))}{|J_n|}
\gtrsim
N_n
\asymp
\frac{1}{\varepsilon_n}
\rightarrow\infty.
\)
Hence
\(
\Phi(\Lambda)\notin\mathcal C.
\)
This contradicts \(\Lambda\in\mathcal C\) and the bidirectional
preservation of Carleson sequences by \(\Phi\).
Consequently, there exist constants \(\kappa>0\) and \(\ell_0>0\)
such that
\(
\operatorname{diam}_{\mathbb T}h_{\Phi}(I)
\geq
\kappa|I|
\)
for every arc \(I\subset\mathbb T\) with \(|I|\leq\ell_0\).

\end{proof}
Let \(h:\mathbb T\to\mathbb T\) be continuous. A continuous mapping
\(H:\mathbb R\to\mathbb R\) is called a \emph{lift} of \(h\) if for \(x\in\mathbb R\),
\(
h\bigl(e^{2\pi i x}\bigr)=e^{2\pi i H(x)}.
\)
Then there exists an integer \(d\), called the degree of \(h\), such that for \(x\in\mathbb R\),
\(
H(x+1)=H(x)+d.
\)
More generally, if \(I\subset\mathbb R\) is an interval, a continuous
mapping \(H_I:I\to\mathbb R\) satisfying
\(
h\bigl(e^{2\pi i x}\bigr)=e^{2\pi i H_I(x)}\) for \(x\in I\)
is called a \emph{local lift} of \(h\) on \(I\).

\begin{cor}
\label{cor:nonvanishing-derivative-of-local-lifts}
Let \(\Phi\in\mathcal P\). If \(H\) is a local lift of
\(h_{\Phi}\) and \(H'(\theta_0)\) exists, then
\(
|H'(\theta_0)|\geq \kappa,
\)
where \(\kappa\) is the constant in
Proposition~\ref{prop:non-collapse-of-arcs}.
\end{cor}

\begin{proof}
Let \(H\) be a local lift of \(h_{\Phi}\) near \(\theta_0\), and put
\(
a:=H'(\theta_0).
\)
For sufficiently small \(r>0\), let
\(
I_r
:=
\left\{
e^{2\pi i\theta}:
|\theta-\theta_0|\leq r
\right\}.
\)
Then \(|I_r|=2r\leq\ell_0\), so
Proposition~\ref{prop:non-collapse-of-arcs} gives
\(
\operatorname{diam}_{\mathbb T}h_{\Phi}(I_r)
\geq 2\kappa r.
\)

On the other hand, differentiability of \(H\) at \(\theta_0\) implies
that, for every \(\varepsilon>0\) and all sufficiently small \(r\),
\(
|H(\theta)-H(\theta_0)|
\leq
(|a|+\varepsilon)|\theta-\theta_0|,
\ |\theta-\theta_0|\leq r.
\)
Hence, for \(\theta,\varphi\in[\theta_0-r,\theta_0+r]\),
\[
d_{\mathbb T}
\bigl(h_{\Phi}(e^{2\pi i\theta}),h_{\Phi}(e^{2\pi i\varphi})\bigr)
\leq
|H(\theta)-H(\varphi)|
\leq
2(|a|+\varepsilon)r.
\]
Therefore,
\(
2\kappa r
\leq
\operatorname{diam}_{\mathbb T}h_{\Phi}(I_r)
\leq
2(|a|+\varepsilon)r.
\)
Thus \(|a|\geq\kappa-\varepsilon\). Letting
\(\varepsilon\downarrow0\), we obtain
\(
|H'(\theta_0)|\geq\kappa>0.
\)
\end{proof}

\section{Bilipschitz Boundary Rigidity and Radial Extensions}

In this section, we prove that the boundary shadow of every mapping in
\(\mathcal P\) is a bilipschitz mapping of \(\mathbb T\) into itself. To
establish injectivity, we combine an argument using regular double points with
an asymptotic Arzel\`a--Ascoli principle and rescaling in the
half plane model. We then prove the composition law for boundary
shadows and characterize the radial extensions that belong to
\(\mathcal P\).

\begin{lem}
\label{lem:regular-double-point}
Let
\(
h:\mathbb T\rightarrow\mathbb T
\)
be a Lipschitz mapping. Assume that whenever a local lift of \(h\) is
differentiable at a point, its derivative there is nonzero. If \(h\)
is not injective, then there exist distinct points
\(
\zeta_1,\zeta_2\in\mathbb T
\)
such that
\(
h(\zeta_1)=h(\zeta_2),
\)
and \(h\) is differentiable at both points in local angular
coordinates, with nonzero derivatives.
\end{lem}

\begin{proof}
The derivative assumption implies that \(h\) is not constant on any
nondegenerate arc. Indeed, a local lift on a constant subarc would be
constant and would therefore have derivative zero.

Let \(H:\mathbb R\to\mathbb R\) be a Lipschitz lift of \(h\), so that
\(
h(e^{2\pi ix})=e^{2\pi iH(x)},\
H(x+1)=H(x)+d,
\)
where \(d\in\mathbb Z\) is the degree of \(h\). We first show that
there exists a nonempty open arc \(U\subset\mathbb T\) every point of
which has at least two distinct preimages under \(h\).

Suppose first that \(H\) is monotone. If \(d=0\), then periodicity
forces \(H\) to be constant, a contradiction. If \(|d|=1\), then the
absence of constant intervals makes \(H\) strictly monotone. For
\(0\leq x<y<1\), strict monotonicity and
\(H(x+1)=H(x)+d\) imply
\(
H(y)-H(x)\notin\mathbb Z,
\)
so \(h(e^{2\pi ix})\neq h(e^{2\pi iy})\). Thus \(h\) would be
injective, again a contradiction. Hence \(|d|\geq2\).

Assume first that \(d\geq2\); the case \(d\leq-2\) is analogous.
Fix \(a\in\mathbb R\) and choose a nonempty open interval
\(
J\subset\bigl(H(a)+1,H(a)+d\bigr),
\ |J|<1.
\)
For every \(t\in J\), continuity gives \(u_t,v_t\in(a,a+1)\) such that
\(
H(u_t)=t,
H(v_t)=t-1.
\)
Thus
\(
h(e^{2\pi iu_t})=h(e^{2\pi iv_t})=e^{2\pi it}.
\)
Moreover, \(u_t\neq v_t\), and since both lie in an interval of length
\(1\), they are distinct modulo \(1\). Therefore
\(
U:=\{e^{2\pi it}:t\in J\}
\)
is a nonempty open arc every point of which has at least two distinct
preimages under \(h\).

Suppose next that \(H\) is not monotone. Then \(H\) fails to be
monotone on some interval of length less than \(1\).
Consequently, there exist \(x_1<x_2<x_3\) with
\(
x_3-x_1<1
\)
such that either
\(
H(x_2)>\max\{H(x_1),H(x_3)\}
\)
or
\(
H(x_2)<\min\{H(x_1),H(x_3)\}.
\)
The two cases are analogous, so assume the first. Choose a nonempty
open interval
\(
J\subset
\bigl(\max\{H(x_1),H(x_3)\},H(x_2)\bigr),
\ |J|<1.
\)
For every \(t\in J\), the intermediate value theorem gives
\(
u_t\in(x_1,x_2),
v_t\in(x_2,x_3),
H(u_t)=H(v_t)=t.
\)
Since
\(
0<v_t-u_t<x_3-x_1<1,
\)
the points \(u_t\) and \(v_t\) are distinct modulo \(1\). Hence
\(
h(e^{2\pi iu_t})
=
h(e^{2\pi iv_t})
=
e^{2\pi it}.
\)
Thus
\(
U:=\{e^{2\pi it}:t\in J\}
\)
is again a nonempty open arc every point of which has at least two
distinct preimages under \(h\).

Let \(N\subset\mathbb T\) be the set of points at which the local
lifts of \(h\) are not differentiable. By Rademacher's theorem,
\(N\) has measure zero. Since \(h\) is Lipschitz, \(h(N)\) also has
measure zero. Choose
\(
\omega\in U\setminus h(N).
\)
Then there exist distinct \(\zeta_1,\zeta_2\in\mathbb T\) such that
\(
h(\zeta_1)=h(\zeta_2)=\omega.
\)
Since \(\omega\notin h(N)\), neither \(\zeta_1\) nor \(\zeta_2\)
belongs to \(N\). Hence the local lifts of \(h\) are differentiable at
both points, and their derivatives are nonzero by hypothesis.
\end{proof}

The following lemma is a variant, based on exhaustion, of the classical
Arzel\`a--Ascoli theorem; see, for example, \cite{Kelley,Weston}.
Unlike the classical version, the domains \(K_n\) increase with \(n\),
the mappings need not be continuous, and equicontinuity is required only
eventually on each fixed \(K_m\). We include the proof for completeness.

\begin{lem}
\label{lem:asymptotic-arzela-ascoli}
Let \((X,d_X)\) and \((Y,d_Y)\) be metric spaces, and let
\(
K_1\subset K_2\subset\cdots
\)
be compact subsets of \(X\) such that every compact subset of \(X\)
is contained in some \(K_m\). For each \(n\geq1\), let
\(
F_n:K_n\rightarrow Y.
\)
Assume that:
\begin{enumerate}
\item[\textup{(i)}]
For every \(m\geq1\), the set
\(
\bigcup_{n\geq m}F_n(K_m)
\)
is relatively compact in \(Y\).

\item[\textup{(ii)}]
For every \(m\geq1\) and every \(\varepsilon>0\), there exist
\(\delta>0\) and \(N\geq m\) such that, whenever
\(
n\geq N,p,q\in K_m,d_X(p,q)<\delta,
\)
one has
\(
d_Y\bigl(F_n(p),F_n(q)\bigr)<\varepsilon.
\)
\end{enumerate}
Then there exist a subsequence \(\{F_{n_j}\}_{j\geq1}\) and a
continuous mapping
\(
F:X\rightarrow Y
\)
such that \(F_{n_j}\to F\) uniformly on every compact subset of \(X\).
\end{lem}
\noindent\emph{Proof outline.}
A diagonal extraction gives pointwise convergence on a countable
union of finite nets. Eventual equicontinuity upgrades this to
uniform convergence on each compact set of the exhaustion, and
also gives continuity of the limit. The complete proof, including
the passage from the exhaustion to all of \(X\), is in
Appendix~\ref{app:geometric-estimates-compactness}.

We now prove that the boundary shadow is injective.
\begin{thm}
\label{thm:boundary-shadow-is-injective}
Let \(\Phi\in\mathcal P\). Then \(h_{\Phi}\) is injective.
\end{thm}

\begin{proof}
Suppose, to the contrary, that \(h_{\Phi}\) is not injective. By
Proposition~\ref{prop:boundary-shadow-is-lipschitz},
Proposition~\ref{prop:non-collapse-of-arcs},
Corollary~\ref{cor:nonvanishing-derivative-of-local-lifts}, and
Lemma~\ref{lem:regular-double-point}, there exist distinct points
\(
\zeta_i=e^{2\pi i\theta_i},
\ i=1,2,
\)
and a point
\(
\omega=e^{2\pi i\alpha}\in\mathbb T
\)
such that
\(
h_{\Phi}(\zeta_1)=h_{\Phi}(\zeta_2)=\omega.
\)
Moreover, for \(i=1,2\), there is a local lift \(H_i\) of \(h_{\Phi}\),
defined near \(\theta_i\), such that
\(
H_i(\theta_i)=\alpha,
a_i:=H_i'(\theta_i)\neq0.
\)
Let
\(
\mathbb H_R
:=
\left\{
u\in\mathbb C:
\operatorname{Re}u>0
\right\}.
\)
For \(\xi\in\mathbb T,z\in\mathbb D\), define
\(
C_\xi(z)
:=
\frac{\xi-z}{\xi+z}
=
\frac{1-\overline{\xi}z}
     {1+\overline{\xi}z}.
\)
Then \(C_\xi\) maps \(\mathbb D\) conformally onto
\(\mathbb H_R\), sends \(\xi\) to \(0\), and for \(u\in\mathbb H_R\),
\(
C_\xi^{-1}(u)
=
\xi\frac{1-u}{1+u}.
\)
For \(u,v\in\mathbb H_R\), put
\(
\rho_{\mathbb H}(u,v)
:=
\left|
\frac{u-v}{u+\overline v}
\right|.
\)
Then
\begin{equation}
\label{eq:cayley-pseudohyperbolic-isometry}
\rho\bigl(C_\xi^{-1}(u),C_\xi^{-1}(v)\bigr)
=
\rho_{\mathbb H}(u,v),
\end{equation}
and, for every \(t>0\),
\begin{equation}
\label{eq:half-plane-dilation-isometry}
\rho_{\mathbb H}(tu,tv)
=
\rho_{\mathbb H}(u,v).
\end{equation}

Choose a sequence \(t_n\downarrow0\). For \(i=1,2\), define
\begin{equation}
\label{eq:cayley-rescaled-mappings}
F_{i,n}(u)
:=
\frac{1}{t_n}
C_\omega
\left(
\Phi\left(
C_{\zeta_i}^{-1}(t_nu)
\right)
\right),
\qquad u\in\mathbb H_R.
\end{equation}
Set
\(
z_{i,n}(u)
:=
C_{\zeta_i}^{-1}(t_nu)
=
\zeta_i\frac{1-t_nu}{1+t_nu}.
\)
For every compact set \(K\subset\mathbb H_R\),
\(
z_{i,n}(u)\rightarrow\zeta_i
\)
uniformly for \(u\in K\).

A direct calculation gives
\begin{equation}
\label{eq:cayley-source-boundary-defect}
1-\left|z_{i,n}(u)\right|^2
=
\frac{4t_n\operatorname{Re}u}
     {|1+t_nu|^2}.
\end{equation}
Moreover, for \(w\in\mathbb D\),
\begin{equation}
\label{eq:cayley-real-part-formula}
\operatorname{Re}C_\omega(w)
=
\frac{1-|w|^2}{|\omega+w|^2}.
\end{equation}

By Theorem~\ref{thm:universal-orbit-frame-preservation}, there exist
constants \(0<A\leq B<\infty\) such that
\(
A\bigl(1-|z|^2\bigr)
\leq
1-|\Phi(z)|^2
\leq
B\bigl(1-|z|^2\bigr)\) for \(z\in\mathbb D\).
Consequently,
\begin{equation}
\label{eq:exact-cayley-real-part-estimate}
A\,Q_{i,n}(u)\operatorname{Re}u
\leq
\operatorname{Re}F_{i,n}(u)
\leq
B\,Q_{i,n}(u)\operatorname{Re}u,
\end{equation}
where
\(
Q_{i,n}(u)
:=
\frac{4}{
|1+t_nu|^2
\left|
\omega+\Phi\bigl(z_{i,n}(u)\bigr)
\right|^2
}.
\)

We claim that
\begin{equation}
\label{eq:cayley-factor-convergence}
Q_{i,n}(u)\rightarrow1
\end{equation}
locally uniformly on \(\mathbb H_R\). Indeed,
\(
\frac{z_{i,n}(u)}{|z_{i,n}(u)|}
\rightarrow
\zeta_i
\)
locally uniformly. Since \(h_{\Phi}\) is Lipschitz,
\(
h_{\Phi}\left(
\frac{z_{i,n}(u)}{|z_{i,n}(u)|}
\right)
\rightarrow
h_{\Phi}(\zeta_i)
=
\omega
\)
locally uniformly. Proposition~\ref{prop:existence-of-boundary-shadow}
then gives
\(
\widehat{\Phi}\bigl(z_{i,n}(u)\bigr)
\rightarrow
\omega
\)
locally uniformly. Moreover,
\(
1-\left|\Phi\bigl(z_{i,n}(u)\bigr)\right|^2
\leq
B\left(1-\left|z_{i,n}(u)\right|^2\right)
\rightarrow0
\)
locally uniformly. Hence
\(
\Phi\bigl(z_{i,n}(u)\bigr)\rightarrow\omega
\)
locally uniformly, which proves
\eqref{eq:cayley-factor-convergence}.

We next control the imaginary parts. Write
\(
z_{i,n}(u)
=
\left|z_{i,n}(u)\right|
e^{2\pi i\vartheta_{i,n}(u)},
\)
where \(\vartheta_{i,n}(u)\) is chosen near \(\theta_i\). Since
\(
\frac{1-t_nu}{1+t_nu}
=
1-2t_nu+O(t_n^2)
\)
locally uniformly in \(u\), we have
\begin{equation}
\label{eq:cayley-source-angular-expansion}
\frac{\vartheta_{i,n}(u)-\theta_i}{t_n}
\rightarrow
-\frac{\operatorname{Im}u}{\pi}
\end{equation}
locally uniformly on \(\mathbb H_R\).

For every compact set \(K\subset\mathbb H_R\), for all sufficiently
large \(n\) and \(u\in K\), write
\(
\widehat{\Phi}\bigl(z_{i,n}(u)\bigr)
=
\omega e^{2\pi i\beta_{i,n}(u)},
\)
where \(\beta_{i,n}(u)\) is chosen near \(0\). Since \(H_i\) is a local
lift of \(h_{\Phi}\), Proposition~\ref{prop:existence-of-boundary-shadow}
gives
\begin{equation}
\label{eq:lifted-shadow-error-in-cayley-model}
\left|
\beta_{i,n}(u)
-
\bigl(
H_i(\vartheta_{i,n}(u))-\alpha
\bigr)
\right|
\leq
A_0\bigl(1-|z_{i,n}(u)|\bigr)
\end{equation}
locally uniformly.

Since
\(
H_i(\vartheta)
=
\alpha
+
a_i(\vartheta-\theta_i)
+
o\bigl(|\vartheta-\theta_i|\bigr)
\)
as \(\vartheta\to\theta_i\), equations
\eqref{eq:cayley-source-boundary-defect},
\eqref{eq:cayley-source-angular-expansion}, and
\eqref{eq:lifted-shadow-error-in-cayley-model} imply that
\begin{equation}
\label{eq:beta-cayley-estimate}
\frac{\beta_{i,n}(u)}{t_n}
=
-\frac{a_i}{\pi}\operatorname{Im}u
+
O\bigl(\operatorname{Re}u\bigr)
+
o_K(1)
\end{equation}
uniformly for \(u\) in each compact set
\(K\subset\mathbb H_R\).

Put
\(
s_{i,n}(u)
:=
\left|
\Phi\bigl(z_{i,n}(u)\bigr)
\right|.
\)
Then
\(
\Phi\bigl(z_{i,n}(u)\bigr)
=
s_{i,n}(u)\,
\omega e^{2\pi i\beta_{i,n}(u)},
\)
and hence
\(
F_{i,n}(u)
=
\frac{
1-s_{i,n}(u)e^{2\pi i\beta_{i,n}(u)}
}{
t_n\left(
1+s_{i,n}(u)e^{2\pi i\beta_{i,n}(u)}
\right)
}.
\)
Therefore,
\begin{equation}
\label{eq:exact-cayley-imaginary-part}
\operatorname{Im}F_{i,n}(u)
=
-\frac{
2s_{i,n}(u)\sin\bigl(2\pi\beta_{i,n}(u)\bigr)
}{
t_n
\left|
1+s_{i,n}(u)e^{2\pi i\beta_{i,n}(u)}
\right|^2
}.
\end{equation}
Since
\(
s_{i,n}(u)\rightarrow1
\)
and
\(
\beta_{i,n}(u)\rightarrow0
\)
locally uniformly, while
\eqref{eq:beta-cayley-estimate} gives
\(
\frac{\beta_{i,n}(u)}{t_n}=O_K(1),
\)
equation~\eqref{eq:exact-cayley-imaginary-part} yields
\(
\operatorname{Im}F_{i,n}(u)
=
-\pi\frac{\beta_{i,n}(u)}{t_n}
+
o_K(1)
\)
uniformly on each compact set \(K\subset\mathbb H_R\). Combining this
with \eqref{eq:beta-cayley-estimate}, we obtain
\begin{equation}
\label{eq:cayley-rescaled-imaginary-part-bound}
\left|
\operatorname{Im}F_{i,n}(u)
-
a_i\operatorname{Im}u
\right|
\leq
C\operatorname{Re}u+o_K(1)
\end{equation}
locally uniformly on \(\mathbb H_R\), where \(C\) is independent of
\(n\).

We now verify the hypotheses of
Lemma~\ref{lem:asymptotic-arzela-ascoli}. Let
\[
K_m
:=
\left\{
u\in\mathbb H_R:
\frac1m\leq\operatorname{Re}u\leq m,\quad
|\operatorname{Im}u|\leq m
\right\},
\qquad m\geq1.
\]
Then
\(
K_1\subset K_2\subset\cdots,
\)
and every compact subset of \(\mathbb H_R\) is contained in some
\(K_m\).

By \eqref{eq:exact-cayley-real-part-estimate},
\eqref{eq:cayley-factor-convergence}, and
\eqref{eq:cayley-rescaled-imaginary-part-bound}, for each fixed \(m\)
there exist constants
\(
0<c_m\leq C_m<\infty
\)
and \(N_m\) such that
\(
c_m
\leq
\operatorname{Re}F_{i,n}(u)
\leq
C_m,
\left|\operatorname{Im}F_{i,n}(u)\right|
\leq
C_m,
\)
whenever \(u\in K_m\) and \(n\geq N_m\). Hence
\(
\bigcup_{n\geq N_m}F_{i,n}(K_m)
\)
is relatively compact in \(\mathbb H_R\).

It remains to treat the finitely many indices \(m\leq n<N_m\).
For each such \(n\), continuity of the mapping
\(
u\mapsto z_{i,n}(u)
\)
and compactness of \(K_m\) give
\(
d_{m,n}
:=
\min_{u\in K_m}
\left(
1-\left|z_{i,n}(u)\right|^2
\right)
>0.
\)
Thus for \(u\in K_m\), we have
\(
1-
\left|
\Phi\bigl(z_{i,n}(u)\bigr)
\right|^2
\geq
A d_{m,n}.
\)
Consequently, there exists \(R_{m,n}<1\) such that, for every
\(u\in K_m\),
\(
\left|
\Phi\bigl(z_{i,n}(u)\bigr)
\right|
\leq
R_{m,n}.
\)
It follows that \(F_{i,n}(K_m)\) is bounded and that its real parts are
bounded away from zero. Hence each \(F_{i,n}(K_m)\) is relatively
compact in \(\mathbb H_R\). Therefore,
\(
\bigcup_{n\geq m}F_{i,n}(K_m)
\)
is relatively compact.

Finally, by
\eqref{eq:cayley-pseudohyperbolic-isometry} and
\eqref{eq:half-plane-dilation-isometry},
\begin{align}
\rho_{\mathbb H}
\bigl(F_{i,n}(u),F_{i,n}(v)\bigr)
&=
\rho
\left(
\Phi\bigl(z_{i,n}(u)\bigr),
\Phi\bigl(z_{i,n}(v)\bigr)
\right),
\label{eq:cayley-image-distance}\\
\rho
\bigl(z_{i,n}(u),z_{i,n}(v)\bigr)
&=
\rho_{\mathbb H}(u,v).
\label{eq:cayley-source-distance}
\end{align}
Since \(z_{i,n}(u)\to\zeta_i\) uniformly on \(K_m\),
Corollary~\ref{cor:uniform-small-scale-control}\textup{(i)} implies
that \(\{F_{i,n}\}\) is eventually equicontinuous on \(K_m\) with
respect to \(\rho_{\mathbb H}\).

We apply Lemma~\ref{lem:asymptotic-arzela-ascoli} with
\(
X=Y=\mathbb H_R,
\
d_X=d_Y=\rho_{\mathbb H},
\)
and with the exhaustion \((K_m)_{m\geq1}\) defined above. Since
\(\rho_{\mathbb H}\) induces the usual Euclidean topology on
\(\mathbb H_R\), each \(K_m\) is compact with respect to
\(\rho_{\mathbb H}\).
Applying Lemma~\ref{lem:asymptotic-arzela-ascoli} first to
\(\{F_{1,n}\}\), and then to \(\{F_{2,n}\}\) along the resulting
subsequence, we obtain a common subsequence, still denoted by
\(\{t_n\}\), and for \(i=1,2\) continuous mappings
\(
F_i:\mathbb H_R\rightarrow\mathbb H_R
\)
such that
\begin{equation}
\label{eq:cayley-local-uniform-convergence}
F_{i,n}\rightarrow F_i
\end{equation}
locally uniformly on \(\mathbb H_R\).

Fix \(u\in\mathbb H_R\). By
\eqref{eq:cayley-factor-convergence},
\(
Q_{i,n}(u)\rightarrow1.
\)
Passing to the limit in
\eqref{eq:exact-cayley-real-part-estimate}, we obtain the global
estimate
\begin{equation}
\label{eq:cayley-limit-real-part-bounds}
A\operatorname{Re}u
\leq
\operatorname{Re}F_i(u)
\leq
B\operatorname{Re}u,
\qquad
u\in\mathbb H_R.
\end{equation}

We next sharpen the estimate for the imaginary part. For each fixed
\(u\in\mathbb H_R\), equations
\eqref{eq:cayley-source-boundary-defect},
\eqref{eq:cayley-source-angular-expansion}, and
\eqref{eq:lifted-shadow-error-in-cayley-model} give
\(
\limsup_{n\to\infty}
\left|
\frac{\beta_{i,n}(u)}{t_n}
+
\frac{a_i}{\pi}\operatorname{Im}u
\right|
\leq
2A_0\operatorname{Re}u.
\)
Using
\(
\operatorname{Im}F_{i,n}(u)
=
-\pi\frac{\beta_{i,n}(u)}{t_n}+o(1)
\)
and passing to the limit, we obtain
\begin{equation}
\label{eq:cayley-limit-imaginary-part-bound}
\left|
\operatorname{Im}F_i(u)
-
a_i\operatorname{Im}u
\right|
\leq
2\pi A_0\operatorname{Re}u,
\qquad
u\in\mathbb H_R.
\end{equation}

 Suppose that
\(
F_i(u)=F_i(v)
\)
for some \(u,v\in\mathbb H_R\). By
\eqref{eq:cayley-local-uniform-convergence},
\(
\rho_{\mathbb H}
\bigl(F_{i,n}(u),F_{i,n}(v)\bigr)
\rightarrow0.
\)
Equation~\eqref{eq:cayley-image-distance} gives
\(
\rho
\left(
\Phi\bigl(z_{i,n}(u)\bigr),
\Phi\bigl(z_{i,n}(v)\bigr)
\right)
\rightarrow0.
\)
Since both source sequences tend to \(\mathbb T\),
Proposition~\ref{prop:basic-properties-of-universal-preservers}
\textup{(ii)} implies that
\(
\rho
\bigl(z_{i,n}(u),z_{i,n}(v)\bigr)
\rightarrow0.
\)
By \eqref{eq:cayley-source-distance},
\(
\rho_{\mathbb H}(u,v)=0,
\)
so \(u=v\). Thus \(F_i\) is injective.

We next show that \(F_i\) is proper. Let
\(
Q\subset\mathbb H_R
\)
be compact. Then there exist constants
\(
0<\mu_Q\leq M_Q<\infty
\)
such that, for every \(w\in Q\),
\(
\mu_Q
\leq
\operatorname{Re}w
\leq
M_Q,
|\operatorname{Im}w|
\leq
M_Q.
\)
Let
\(
u\in F_i^{-1}(Q).
\)
By \eqref{eq:cayley-limit-real-part-bounds},
\(
\frac{\mu_Q}{B}
\leq
\operatorname{Re}u
\leq
\frac{M_Q}{A}.
\)
Moreover, by
\eqref{eq:cayley-limit-imaginary-part-bound},
\[
|a_i|\,|\operatorname{Im}u|
\leq
\left|\operatorname{Im}F_i(u)\right|
+
2\pi A_0\operatorname{Re}u
\leq
M_Q
+
\frac{2\pi A_0M_Q}{A}.
\]
Since \(a_i\neq0\), it follows that
\(
|\operatorname{Im}u|
\leq
\frac{1}{|a_i|}
\left(
M_Q+\frac{2\pi A_0M_Q}{A}
\right).
\)
Thus \(F_i^{-1}(Q)\) is contained in a compact subset of
\(\mathbb H_R\). Since \(F_i\) is continuous, \(F_i^{-1}(Q)\) is
closed, and therefore compact. Hence \(F_i\) is proper.

By invariance of domain, the injective continuous mapping
\(
F_i:\mathbb H_R\rightarrow\mathbb H_R
\)
has open image. Let
\(
F_i(u_n)\rightarrow q\in\mathbb H_R.
\)
The set
\(
Q
:=
\{q\}\cup\{F_i(u_n):n\geq1\}
\)
is compact. Since \(F_i\) is proper, \(F_i^{-1}(Q)\) is compact.
Passing to a subsequence, we may assume that
\(
u_n\rightarrow u\in\mathbb H_R.
\)
Continuity gives
\(
F_i(u)=q,
\)
and hence \(q\in F_i(\mathbb H_R)\). Thus
\(F_i(\mathbb H_R)\) is closed.

Since \(\mathbb H_R\) is connected and
\(F_i(\mathbb H_R)\) is nonempty, open, and closed, we conclude that
\begin{equation}
\label{eq:cayley-limit-surjectivity}
F_i(\mathbb H_R)=\mathbb H_R,
\qquad
i=1,2.
\end{equation}

Choose \(q\in\mathbb H_R\). By
\eqref{eq:cayley-limit-surjectivity}, there exist
\(u_1,u_2\in\mathbb H_R\) such that
\(
F_1(u_1)=q=F_2(u_2).
\)
It follows from the locally uniform convergence that
\(
\rho_{\mathbb H}
\bigl(F_{1,n}(u_1),F_{2,n}(u_2)\bigr)
\rightarrow0.
\)
Since the same target Cayley mapping \(C_\omega\) is used in both
rescalings,
\(
\rho
\left(
\Phi\bigl(z_{1,n}(u_1)\bigr),
\Phi\bigl(z_{2,n}(u_2)\bigr)
\right)
\rightarrow0.
\)
Proposition~\ref{prop:basic-properties-of-universal-preservers}
\textup{(ii)} therefore yields
\begin{equation}
\label{eq:cayley-source-pairs-converge}
\rho
\bigl(z_{1,n}(u_1),z_{2,n}(u_2)\bigr)
\rightarrow0.
\end{equation}

On the other hand,
\(
z_{1,n}(u_1)\rightarrow\zeta_1,
z_{2,n}(u_2)\rightarrow\zeta_2.
\)
Since \(\zeta_1\neq\zeta_2\), the identity
\(
1-\rho(z,w)^2
=
\frac{
(1-|z|^2)(1-|w|^2)
}{
|1-\overline zw|^2
}
\)
implies that
\(
\rho
\bigl(z_{1,n}(u_1),z_{2,n}(u_2)\bigr)
\rightarrow1,
\)
contradicting
\eqref{eq:cayley-source-pairs-converge}. Therefore \(h_{\Phi}\) is injective.
\end{proof}

\begin{thm}
\label{thm:boundary-shadow-is-bilipschitz}
Let
\(
\Phi\in\mathcal P.
\) Then
\(
h_{\Phi}\in\operatorname{BiLip}(\mathbb T).
\)
\end{thm}

\begin{proof}
By Proposition~\ref{prop:boundary-shadow-is-lipschitz}, there exists
\(L\geq 1\) such that
\(
d_{\mathbb T}\bigl(h_{\Phi}(\zeta),h_{\Phi}(\eta)\bigr)
\leq
L\,d_{\mathbb T}(\zeta,\eta),
\
\zeta,\eta\in\mathbb T.
\)
In particular, \(h_{\Phi}\) is continuous. By
Theorem~\ref{thm:boundary-shadow-is-injective}, \(h_{\Phi}\) is
injective.

Since \(\mathbb T\) is compact and \(\mathbb T\) is Hausdorff,
\(h_{\Phi}\) is a homeomorphism from \(\mathbb T\) onto \(h_{\Phi}(\mathbb T)\).
We claim that
\(
h_{\Phi}(\mathbb T)=\mathbb T.
\)
Indeed, if \(h_{\Phi}(\mathbb T)\neq\mathbb T\), choose
\(\omega\in\mathbb T\setminus h_{\Phi}(\mathbb T)\). Then
\(\mathbb T\setminus\{\omega\}\) is homeomorphic to an open interval
of \(\mathbb R\). Hence \(h_{\Phi}(\mathbb T)\) would be homeomorphic to a
compact connected subset of \(\mathbb R\), and therefore to a closed
interval. This is impossible because \(h_{\Phi}(\mathbb T)\) is homeomorphic
to \(\mathbb T\). Thus
\(
h_{\Phi}(\mathbb T)=\mathbb T,
\)
and \(h_{\Phi}\) is a homeomorphism of \(\mathbb T\).

Let \(\kappa>0\) and \(\ell_0>0\) be the constants in
Proposition~\ref{prop:non-collapse-of-arcs}. Set
\(
\ell_1
:=
\min\left\{
\ell_0,\frac{1}{4L}
\right\}.
\)
We first prove a lower Lipschitz estimate for sufficiently close
points. Let \(\zeta,\eta\in\mathbb T\) satisfy
\(
0<
\alpha
:=
d_{\mathbb T}(\zeta,\eta)
\leq
\ell_1.
\)
Let \(I\subset\mathbb T\) be the shorter closed arc joining
\(\zeta\) to \(\eta\). Then
\(
|I|
=
\alpha
\leq
\ell_1
\leq
\ell_0.
\)
For every \(\xi,\omega\in I\), the Lipschitz property gives
\(
d_{\mathbb T}\bigl(h_{\Phi}(\xi),h_{\Phi}(\omega)\bigr)
\leq
L\,d_{\mathbb T}(\xi,\omega)
\leq
L|I|
\leq
\frac{1}{4}.
\)
Consequently,
\(
\operatorname{diam}_{\mathbb T}h_{\Phi}(I)
\leq
\frac{1}{4}.
\)

Since \(h_{\Phi}\) is a homeomorphism, \(h_{\Phi}(I)\) is a closed arc joining
\(h_{\Phi}(\zeta)\) to \(h_{\Phi}(\eta)\). As its diameter is strictly less than
\(\frac{1}{2}\), it is the shorter arc joining its endpoints. Therefore,
\(
\operatorname{diam}_{\mathbb T}h_{\Phi}(I)
=
d_{\mathbb T}\bigl(h_{\Phi}(\zeta),h_{\Phi}(\eta)\bigr).
\)
By Proposition~\ref{prop:non-collapse-of-arcs},
\(
d_{\mathbb T}\bigl(h_{\Phi}(\zeta),h_{\Phi}(\eta)\bigr)
=
\operatorname{diam}_{\mathbb T}h_{\Phi}(I)
\geq
\kappa |I|
=
\kappa\,d_{\mathbb T}(\zeta,\eta).
\)
Thus
\(
d_{\mathbb T}\bigl(h_{\Phi}(\zeta),h_{\Phi}(\eta)\bigr)
\geq
\kappa\,d_{\mathbb T}(\zeta,\eta)
\)
whenever
\(
d_{\mathbb T}(\zeta,\eta)\leq\ell_1.
\)

It remains to consider pairs satisfying
\(
d_{\mathbb T}(\zeta,\eta)\geq\ell_1.
\)
Define
\[
K
:=
\left\{
(\zeta,\eta)\in\mathbb T\times\mathbb T:
d_{\mathbb T}(\zeta,\eta)\geq\ell_1
\right\}.
\]
The set \(K\) is compact.
Since \(h_{\Phi}\) is injective, the function
\(
(\zeta,\eta)
\mapsto
\frac{
d_{\mathbb T}\bigl(h_{\Phi}(\zeta),h_{\Phi}(\eta)\bigr)
}{
d_{\mathbb T}(\zeta,\eta)
}
\)
is continuous and strictly positive on \(K\). Hence it attains a
positive minimum
\(
m
:=
\min_{(\zeta,\eta)\in K}
\frac{
d_{\mathbb T}\bigl(h_{\Phi}(\zeta),h_{\Phi}(\eta)\bigr)
}{
d_{\mathbb T}(\zeta,\eta)
}
>0.
\)
Therefore,
\(
d_{\mathbb T}\bigl(h_{\Phi}(\zeta),h_{\Phi}(\eta)\bigr)
\geq
m\,d_{\mathbb T}(\zeta,\eta)
\)
whenever
\(
d_{\mathbb T}(\zeta,\eta)\geq\ell_1.
\)

Set
\(
c
:=
\min\{\kappa,m\}>0.
\)
Combining the preceding estimates, we obtain
\[
c\,d_{\mathbb T}(\zeta,\eta)
\leq
d_{\mathbb T}\bigl(h_{\Phi}(\zeta),h_{\Phi}(\eta)\bigr)
\leq
L\,d_{\mathbb T}(\zeta,\eta),
\qquad
\zeta,\eta\in\mathbb T.
\]
Thus \(h_{\Phi}\) is a bilipschitz bijection of \(\mathbb T\), and therefore,
\(
h_{\Phi}\in\operatorname{BiLip}(\mathbb T).
\)
\end{proof}

The preceding boundary rigidity separates the two geometric components of Carleson interpolation: the Carleson measure condition and pseudohyperbolic separation. We record the resulting characterization before turning to the algebraic structure of boundary shadows.

For a positive Borel measure \(\mu\) on \(\mathbb D\), set
\(
\|\mu\|_{\mathrm{CM}}
:=
\sup_{I\subset\mathbb T}
\frac{\mu(S(I))}{|I|}.
\)
For an arbitrary mapping
\(
\Phi:\mathbb D\to\mathbb D
\)
and a finite positive discrete measure
\(
\nu=\sum_{j=1}^N a_j\delta_{z_j},
\)
define
\(
\Phi_{\#}\nu
:=
\sum_{j=1}^N a_j\delta_{\Phi(z_j)}.
\)
This definition does not require any measurability assumption on
\(\Phi\).

To specify the normalization, for a conformal map \(\varphi\) of
\(\mathbb D\) onto \(\Omega\), the weighted pushforward studied in
\cite{Zinsmeister1989,WeiZinsmeister2018} has the form
\(
\varphi_{\#}\bigl(|\varphi'|\,\nu\bigr).
\)
The operator \(\Phi_{\#}\) defined above retains the original atom
masses. The characterization below concerns this ordinary
pushforward together with preservation of pairwise separation;
both conditions enter the equivalence with membership in
\(\mathcal P\).

For an indexed sequence
\(
\Lambda=\{\lambda_j\}_{j\geq1}\subset\mathbb D,
\)
write
\(
\operatorname{sep}(\Lambda)
:=
\inf_{j\neq k}\rho(\lambda_j,\lambda_k).
\)
Thus
\(
\operatorname{sep}(\Lambda)>0
\)
is precisely the pairwise separation condition in
Lemma~\ref{thm:carleson-interpolation}.

\begin{thm}
\label{thm:carleson-measure-separation-characterization}
Let
\(
\Phi:\mathbb D\to\mathbb D
\)
be an arbitrary mapping. Then \(\Phi\in\mathcal P\) if and only if the
following two conditions hold:
\begin{enumerate}
\item[\textup{(i)}]
There exists \(C\geq1\) such that
\(
C^{-1}\|\nu\|_{\mathrm{CM}}
\leq
\|\Phi_{\#}\nu\|_{\mathrm{CM}}
\leq
C\|\nu\|_{\mathrm{CM}}
\)
for every finite positive discrete measure \(\nu\) on \(\mathbb D\).

\item[\textup{(ii)}]
For every indexed sequence
\(
\Lambda\subset\mathbb D,
\)
one has
\(
\operatorname{sep}(\Lambda)>0
\Longleftrightarrow
\operatorname{sep}\bigl(\Phi(\Lambda)\bigr)>0.
\)
\end{enumerate}
\end{thm}

\begin{proof}
Assume first that \(\Phi\in\mathcal P\). We begin with
condition~\textup{(i)}. Write
\(
\delta(z):=1-|z|.
\)
By
\eqref{eq:boundary-defect-comparability-for-P}, there exist
\(a,b>0\) such that
\(
a\,\delta(z)
\leq
\delta\bigl(\Phi(z)\bigr)
\leq
b\,\delta(z)\) for \(z\in\mathbb D.\)
Let \(A<\infty\) and \(t_0>0\) be as in
Proposition~\ref{prop:existence-of-boundary-shadow}, so that
\(
d_{\mathbb T}
\left(
\widehat{\Phi}\bigl((1-t)\zeta\bigr),
h_\Phi(\zeta)
\right)
\leq At\) for \(\zeta\in\mathbb T,\quad 0<t<t_0\).
By Theorem~\ref{thm:boundary-shadow-is-bilipschitz}, there exists
\(L\geq1\) such that both \(h_\Phi\) and \(h_\Phi^{-1}\) are
Lipschitz with constant \(L\).

We first claim that there exists \(M\geq1\), depending only on
\(\Phi\), such that for every arc \(I\subset\mathbb T\) there is an arc
\(I_-\subset\mathbb T\) satisfying
\begin{equation}
\label{eq:inverse-carleson-square-control}
\Phi^{-1}(S(I))
\subset S(I_-),
\qquad
|I_-|\leq M|I|.
\end{equation}
It is enough to prove this for sufficiently small arcs, since for arcs
of length bounded below one may take \(I_-=\mathbb T\) after increasing
\(M\).

Let therefore \(|I|\) be sufficiently small and let
\(
z=(1-t)\zeta\in\Phi^{-1}(S(I)).
\)
Then
\(
\delta(\Phi(z))\leq |I|,
\)
and the lower estimate for the boundary defect gives
\(
t=\delta(z)\leq a^{-1}|I|.
\)
For \(|I|\) small enough, \(t<t_0\). Since the angular component of
\(\Phi(z)\) belongs to \(I\), the boundary estimate of first order yields
\(
d_{\mathbb T}\bigl(h_\Phi(\zeta),I\bigr)
\leq
Aa^{-1}|I|.
\)
Hence \(h_\Phi(\zeta)\) belongs to an enlargement of \(I\) whose length
is at most
\(
(1+2Aa^{-1})|I|.
\)
Since \(h_\Phi^{-1}\) is Lipschitz with constant \(L\), all possible directions
\(\zeta\) lie in an arc of length at most
\(
L(1+2Aa^{-1})|I|.
\)
After enlarging this arc, if necessary, so that its length is also at
least \(a^{-1}|I|\), we obtain
\eqref{eq:inverse-carleson-square-control}.

Similarly, there exists, for every arc \(I\subset\mathbb T\), an arc
\(I_+\subset\mathbb T\) such that
\begin{equation}
\label{eq:forward-carleson-square-control}
\Phi(S(I))
\subset S(I_+),
\qquad
|I_+|\leq M|I|.
\end{equation}
Indeed, for
\(
z=(1-t)\zeta\in S(I)
\)
we have \(t\leq |I|\), and therefore
\(
\delta(\Phi(z))
\leq b|I|.
\)
Moreover,
\(
d_{\mathbb T}
\bigl(
\widehat{\Phi}(z),h_\Phi(\zeta)
\bigr)
\leq A|I|.
\)
Since \(h_\Phi\) is Lipschitz with constant \(L\) and \(\zeta\in I\), the angular
components of points in \(\Phi(S(I))\) lie in an arc of length at most
\(
(L+2A)|I|.
\)
Enlarging this arc, if necessary, so that its length is at least
\(b|I|\), gives
\eqref{eq:forward-carleson-square-control}.

Let now \(\nu\) be a finite positive discrete measure. From
\eqref{eq:inverse-carleson-square-control},
\[
(\Phi_{\#}\nu)(S(I))
=
\nu\bigl(\Phi^{-1}(S(I))\bigr)
\leq
\nu(S(I_-))
\leq
\|\nu\|_{\mathrm{CM}}|I_-|
\leq
M\|\nu\|_{\mathrm{CM}}|I|.
\]
Hence
\(
\|\Phi_{\#}\nu\|_{\mathrm{CM}}
\leq
M\|\nu\|_{\mathrm{CM}}.
\)
Conversely, by
\eqref{eq:forward-carleson-square-control},
\(
\Phi(S(I))\subset S(I_+),
\)
and therefore
\(
S(I)\subset\Phi^{-1}(S(I_+)).
\)
Consequently,
\[
\nu(S(I))
\leq
\nu\bigl(\Phi^{-1}(S(I_+))\bigr)
=
(\Phi_{\#}\nu)(S(I_+))
\leq
\|\Phi_{\#}\nu\|_{\mathrm{CM}}|I_+|
\leq
M\|\Phi_{\#}\nu\|_{\mathrm{CM}}|I|.
\]
Thus
\(
M^{-1}\|\nu\|_{\mathrm{CM}}
\leq
\|\Phi_{\#}\nu\|_{\mathrm{CM}}
\leq
M\|\nu\|_{\mathrm{CM}},
\)
which proves condition~\textup{(i)}.

We next prove condition~\textup{(ii)}. By
Theorem~\ref{thm:universal-orbit-frame-preservation},
boundary defects are comparable under \(\Phi\), while
Proposition~\ref{prop:basic-properties-of-universal-preservers}
shows that \(\Phi\) is injective and that, whenever
\(|z_n|,|w_n|\to1\),
\(
\rho(z_n,w_n)\to0
\Longleftrightarrow
\rho\bigl(\Phi(z_n),\Phi(w_n)\bigr)\to0.
\)
Suppose that pairwise separation failed to be preserved in either
direction. Repetitions are excluded whenever the sequence on the
opposite side is pairwise separated, so the relevant pairs may be
taken distinct. Then there would exist distinct pairs \(z_n,w_n\) in the
nonseparated sequence for which the corresponding pseudohyperbolic
distances tend to zero, while the sequence on the other side is
pairwise separated. Since a pairwise separated sequence has only
finitely many points in each compact subdisk, injectivity and
boundary defect comparability imply, after passing to a subsequence,
that
\(
|z_n|\to1,
|w_n|\to1.
\)
The preceding equivalence then forces the corresponding
pseudohyperbolic distances on the pairwise separated side to tend to
zero, a contradiction. Hence
\(
\operatorname{sep}(\Lambda)>0
\Longleftrightarrow
\operatorname{sep}\bigl(\Phi(\Lambda)\bigr)>0
\)
for every indexed sequence \(\Lambda\).

Conversely, assume conditions~\textup{(i)} and~\textup{(ii)}.
Condition~\textup{(i)} extends from finite to countable positive
discrete measures by applying it to increasing finite partial sums and
using monotone convergence.

Apply condition~\textup{(i)} to the point mass \(\delta_z\). Since
\(
\|\delta_z\|_{\mathrm{CM}}
=
\frac{1}{1-|z|},
\)
we obtain
\(
C^{-1}(1-|z|)
\leq
1-|\Phi(z)|
\leq
C(1-|z|)\) for \(z\in\mathbb D\).
Consequently,
\begin{equation}
\label{eq:boundary-defect-from-carleson-measure-control}
1-|\Phi(z)|^2
\asymp
1-|z|^2,
\qquad z\in\mathbb D.
\end{equation}

It remains to prove bidirectional preservation of Carleson sequences.
Let
\(
\Lambda=\{\lambda_j\}_{j\geq1}.
\)
Suppose first that
\(
\Lambda\in\mathcal C.
\)
By Lemma~\ref{thm:carleson-interpolation},
\(
\operatorname{sep}(\Lambda)>0
\)
and
\(
\mu_\Lambda
:=
\sum_{j\geq1}
(1-|\lambda_j|^2)\delta_{\lambda_j}
\)
is a Carleson measure. Condition~\textup{(ii)} gives
\(
\operatorname{sep}\bigl(\Phi(\Lambda)\bigr)>0.
\)
By condition~\textup{(i)}, extended to countable positive discrete
measures,
\(
\Phi_{\#}\mu_\Lambda
=
\sum_{j\geq1}
(1-|\lambda_j|^2)\delta_{\Phi(\lambda_j)}
\)
is a Carleson measure. By
\eqref{eq:boundary-defect-from-carleson-measure-control},
\(
1-|\Phi(\lambda_j)|^2
\asymp
1-|\lambda_j|^2\) for \( j\geq1\).
Therefore
\(
\mu_{\Phi(\Lambda)}
:=
\sum_{j\geq1}
\bigl(1-|\Phi(\lambda_j)|^2\bigr)
\delta_{\Phi(\lambda_j)}
\)
is a Carleson measure. Lemma~\ref{thm:carleson-interpolation} now yields
\(
\Phi(\Lambda)\in\mathcal C.
\)

The reverse implication is analogous. If
\(
\Phi(\Lambda)\in\mathcal C,
\)
then condition~\textup{(ii)} gives
\(
\operatorname{sep}(\Lambda)>0.
\)
Moreover, \(\mu_{\Phi(\Lambda)}\) is a Carleson measure. By
\eqref{eq:boundary-defect-from-carleson-measure-control}, this is equivalent,
up to uniform constants, to the Carleson property of
\(\Phi_{\#}\mu_\Lambda\). The lower estimate in
condition~\textup{(i)} then gives that \(\mu_\Lambda\) is a Carleson
measure. Hence
\(
\Lambda\in\mathcal C.
\)
We have proved
\(
\Lambda\in\mathcal C\Longleftrightarrow
\Phi(\Lambda)\in\mathcal C
\)
for every indexed sequence \(\Lambda\), together with
\eqref{eq:boundary-defect-from-carleson-measure-control}. Therefore
Theorem~\ref{thm:universal-orbit-frame-preservation} yields
\(
\Phi\in\mathcal P.
\)
\end{proof}

\begin{cor}
\label{cor:borel-carleson-measure-characterization}
Let
\(
\Phi:\mathbb D\to\mathbb D
\)
be Borel measurable. Then \(\Phi\in\mathcal P\) if and only if the
following two conditions hold:
\begin{enumerate}
\item[\textup{(i)}]
There exists \(C\geq1\) such that
\(
C^{-1}\|\mu\|_{\mathrm{CM}}
\leq
\|\Phi_*\mu\|_{\mathrm{CM}}
\leq
C\|\mu\|_{\mathrm{CM}}
\)
for every positive Borel measure \(\mu\) on \(\mathbb D\), where
\(
\Phi_*\mu(E)
:=
\mu\bigl(\Phi^{-1}(E)\bigr).
\)

\item[\textup{(ii)}]
For every indexed sequence
\(
\Lambda\subset\mathbb D,
\)
one has
\(
\operatorname{sep}(\Lambda)>0
\Longleftrightarrow
\operatorname{sep}\bigl(\Phi(\Lambda)\bigr)>0.
\)
\end{enumerate}
\end{cor}

\begin{proof}
Assume first that \(\Phi\in\mathcal P\). The proof of
Theorem~\ref{thm:carleson-measure-separation-characterization} gives,
for every arc \(I\subset\mathbb T\), arcs \(I_-,I_+\subset\mathbb T\)
such that
\(
\Phi^{-1}(S(I))\subset S(I_-),
\Phi(S(I))\subset S(I_+),
\)
with
\(
|I_-|+|I_+|
\lesssim_\Phi |I|.
\)
Since \(\Phi\) is Borel measurable, \(\Phi_*\mu\) is defined for every
positive Borel measure \(\mu\). Hence
\[
(\Phi_*\mu)(S(I))
=
\mu\bigl(\Phi^{-1}(S(I))\bigr)
\leq
\mu(S(I_-))
\lesssim_\Phi
\|\mu\|_{\mathrm{CM}}|I|,
\]
which gives
\(
\|\Phi_*\mu\|_{\mathrm{CM}}
\lesssim_\Phi
\|\mu\|_{\mathrm{CM}}.
\)
On the other hand,
\(
S(I)\subset\Phi^{-1}(S(I_+)),
\)
so
\[
\mu(S(I))
\leq
(\Phi_*\mu)(S(I_+))
\lesssim_\Phi
\|\Phi_*\mu\|_{\mathrm{CM}}|I|.
\]
Taking the supremum over \(I\) yields
\(
\|\mu\|_{\mathrm{CM}}
\lesssim_\Phi
\|\Phi_*\mu\|_{\mathrm{CM}}.
\)
Thus condition~\textup{(i)} holds, while condition~\textup{(ii)}
follows from
Theorem~\ref{thm:carleson-measure-separation-characterization}.

Conversely, condition~\textup{(i)} applies in particular to finite
positive discrete measures. Hence conditions~\textup{(i)} and
\textup{(ii)} of
Theorem~\ref{thm:carleson-measure-separation-characterization} hold,
and that theorem gives
\(
\Phi\in\mathcal P.
\)
\end{proof}

\begin{remark}
\label{rem:boundary-shadow-composition-law}
Let
\(
\Phi,\Theta\in\mathcal P.
\)
By Theorem~\ref{thm:universal-orbit-frame-preservation},
\(
\Phi\circ\Theta\in\mathcal P,
\)
so \(h_{\Phi\circ\Theta}\) is well defined. Then
\(
h_{\Phi\circ\Theta}
=
h_\Phi\circ h_\Theta.
\)

Indeed, fix
\(
\zeta\in\mathbb T
\)
and set
\(
z_t:=\Theta\bigl((1-t)\zeta\bigr).
\)
For all sufficiently small \(t>0\), one has \(z_t\neq0\). Write
\(
z_t=(1-s_t)\eta_t,
s_t:=1-|z_t|,
\eta_t:=\frac{z_t}{|z_t|}\in\mathbb T.
\)

By
Theorem~\ref{thm:universal-orbit-frame-preservation},
there exist constants \(c,C>0\) such that
\[
c\bigl(1-(1-t)^2\bigr)
\leq
1-|z_t|^2
\leq
C\bigl(1-(1-t)^2\bigr).
\]
Since
\(
1-|z_t|^2
=
(1+|z_t|)s_t,
\)
it follows that
\(
s_t\rightarrow0\) as \(t\to0^+\).
Moreover, by the definition of the boundary shadow,
\(
\eta_t
=
\frac{\Theta((1-t)\zeta)}
{\left|\Theta((1-t)\zeta)\right|}
\rightarrow
h_\Theta(\zeta).
\)
Applying
Proposition~\ref{prop:existence-of-boundary-shadow}
to \(\Phi\), we obtain
\(
d_{\mathbb T}
\left(
\frac{\Phi(z_t)}{|\Phi(z_t)|},
h_\Phi(\eta_t)
\right)
\leq
A_\Phi s_t
\rightarrow0.
\)
Since \(h_\Phi\) is continuous,
\(
h_\Phi(\eta_t)
\rightarrow
h_\Phi\bigl(h_\Theta(\zeta)\bigr).
\)
Therefore,
\(
\frac{
(\Phi\circ\Theta)((1-t)\zeta)
}{
\left|
(\Phi\circ\Theta)((1-t)\zeta)
\right|
}
\rightarrow
h_\Phi\bigl(h_\Theta(\zeta)\bigr).
\)
By the uniqueness of the boundary shadow,
\(
h_{\Phi\circ\Theta}(\zeta)
=
h_\Phi\bigl(h_\Theta(\zeta)\bigr).
\)
Since \(\zeta\in\mathbb T\) was arbitrary, the result follows.
\end{remark}

For any mapping
\(
h:\mathbb T\rightarrow\mathbb T,
\)
define its radial extension by
\begin{equation}
E_h(0):=0,
\qquad
E_h(r\zeta):=r h(\zeta),
\qquad
0<r<1,\ \zeta\in\mathbb T.
\label{eq:radial-extension-arbitrary-boundary-map}
\end{equation}
No regularity assumption is imposed on \(h\).

The following theorem shows that the radial extension of a mapping on the unit circle belongs to \(\mathcal P\) precisely when the boundary mapping is bilipschitz.

\begin{thm}
\label{thm:angular-classification-without-continuity}
Let
\(
h:\mathbb T\rightarrow\mathbb T
\)
be an arbitrary mapping, and let \(E_h\) be defined by
\eqref{eq:radial-extension-arbitrary-boundary-map}. Then
\(
E_h\in\mathcal P\) if and only if
\(
h\in\operatorname{BiLip}(\mathbb T).
\)
\end{thm}

\begin{proof}
We first prove the sufficiency. Suppose that
\(
h\in\operatorname{BiLip}(\mathbb T),
\)
and let \(L\geq 1\) be a bilipschitz constant for \(h\).

Let
\(
z=r\zeta,
\
w=s\eta,
\)
where \(0<r,s<1\) and \(\zeta,\eta\in\mathbb T\). Put
\(
\alpha
:=
d_{\mathbb T}(\zeta,\eta),
\
\alpha_h
:=
d_{\mathbb T}\bigl(h(\zeta),h(\eta)\bigr).
\)
Since \(h\) is bilipschitz,
\(
L^{-1}\alpha
\leq
\alpha_h
\leq
L\alpha.
\)
For \(u\in[0,\pi]\), one has
\(
\frac{u}{\pi}
\leq
\sin\left(\frac{u}{2}\right)
\leq
\frac{u}{2}.
\)
Hence there exists a constant \(K_L\geq 1\), depending only on \(L\),
such that
\[
K_L^{-1}
\sin^2\left(\pi\alpha\right)
\leq
\sin^2\left(\pi{\alpha_h}\right)
\leq
K_L
\sin^2\left(\pi{\alpha}\right).
\]

Set
\(
A:=(r-s)^2,
\
C:=(1-rs)^2,
\)
and
\(
B
:=
4rs\sin^2\left(\pi{\alpha}\right),
\
B_h
:=
4rs\sin^2\left(\pi\alpha_h\right).
\)
Then
\(
A\leq C,
\)
because
\(
C-A
=
(1-r^2)(1-s^2)
\geq 0.
\)
Moreover,
\(
K_L^{-1}B
\leq
B_h
\leq
K_LB.
\)
A direct computation gives
\(
\rho(z,w)^2
=
\frac{A+B}{C+B},
\)
whereas
\(
\rho\bigl(E_h(z),E_h(w)\bigr)^2
=
\frac{A+B_h}{C+B_h}.
\)
Therefore,
\[
K_L^{-1}\rho(z,w)^2
\leq
\rho\bigl(E_h(z),E_h(w)\bigr)^2
\leq
K_L\rho(z,w)^2.
\]
The cases \(r=0\) or \(s=0\) are immediate, since \(E_h\) preserves
the modulus. It follows that \(E_h\) maps pairwise separated
sequences to pairwise separated sequences. Applying the same argument
to \(h^{-1}\), we also obtain the converse implication.

We next verify the Carleson measure condition.
Let \(\Lambda\in\mathcal C\), let \(I\subset\mathbb T\) be an
arbitrary arc, and set
\(
J:=h^{-1}(I).
\)
Since \(h\) is a bilipschitz homeomorphism with constant \(L\), \(J\) is an arc
and
\(
L^{-1}|J|
\leq
|I|
=
|h(J)|
\leq
L|J|.
\)
Let \(J^{(L)}:=\mathbb T\) if \(L|J|\geq1\); otherwise, let
\(J^{(L)}\) be the arc with the same midpoint as \(J\). In either case,
its normalized arc length is
\(
|J^{(L)}|=\min\{1,L|J|\}.
\)
If
\(
\lambda=r\zeta\in\Lambda
\)
satisfies
\(
E_h(\lambda)\in S(I),
\)
then
\(
h(\zeta)\in I,
\
1-r\leq |I|.
\)
Hence
\(
\zeta\in J
\)
and
\(
1-r
\leq
|I|
\leq
L|J|.
\)
Consequently,
\(
\lambda\in S\bigl(J^{(L)}\bigr).
\)

Since \(E_h\) preserves moduli, we obtain
\[
\mu_{E_h(\Lambda)}\bigl(S(I)\bigr)
=
\sum_{\substack{\lambda\in\Lambda\\
E_h(\lambda)\in S(I)}}
\bigl(1-|\lambda|^2\bigr)
\leq
\mu_\Lambda\bigl(S(J^{(L)})\bigr).
\]
Since \(\Lambda\in\mathcal C\), there exists a constant
\(C_\Lambda>0\) such that
\(
\mu_\Lambda\bigl(S(J^{(L)})\bigr)
\leq
C_\Lambda |J^{(L)}|.
\)
Furthermore,
\(
|J^{(L)}|
\leq
L|J|
\leq
L^2|I|.
\)
Therefore,
\(
\mu_{E_h(\Lambda)}\bigl(S(I)\bigr)
\leq
C_\Lambda L^2|I|.
\)
Thus, \(\mu_{E_h(\Lambda)}\) is a Carleson measure. Together with the
pairwise separation already proved, this yields
\(
E_h(\Lambda)\in\mathcal C.
\)
Applying the same argument to \(h^{-1}\), we obtain
\(
\Lambda\in\mathcal C
\Longleftrightarrow
E_h(\Lambda)\in\mathcal C
\)
for every sequence \(\Lambda\subset\mathbb D\).

Finally, for $z\in\mathbb D$,
\(
|E_h(z)|=|z|.
\)
Hence
\(
1-|E_h(z)|^2
=
1-|z|^2.
\)
Thus, both conditions of
Theorem~\ref{thm:universal-orbit-frame-preservation} are satisfied.
Therefore,
\(
E_h\in\mathcal P.
\)
Conversely, suppose that \( E_h\in\mathcal P. \) For every \(\zeta\in\mathbb T\) and \(0<t<1\),
\(
E_h\bigl((1-t)\zeta\bigr) = (1-t)h(\zeta),\)
and hence
\(
\widehat{E_h}\bigl((1-t)\zeta\bigr) = h(\zeta).
\) Thus the canonical radial boundary trace of \(E_h\) is precisely \(h\).
By Theorem~\ref{thm:boundary-shadow-is-bilipschitz},
\(
h=h_{E_h}\in\operatorname{BiLip}(\mathbb T).
\)
\end{proof}

\section{The Boundary Shadow Kernel and Semidirect Product Structure}
Consider the boundary shadow mapping
\(
\mathfrak{S}\colon
\mathcal P
\rightarrow
\operatorname{BiLip}(\mathbb T),
\mathfrak{S}(\Phi):=h_{\Phi}.
\)
By Remark~\ref{rem:boundary-shadow-composition-law},
\(\mathfrak{S}\) is a monoid homomorphism. Its kernel is
\(
\mathcal K
:=
\ker\mathfrak{S}
=
\left\{
\Psi\in\mathcal P:
h_{\Psi}=\operatorname{id}_{\mathbb T}
\right\},
\)
which we call the \emph{boundary shadow kernel}.

We first prove that every \(\Phi\in\mathcal P\) admits a unique
factorization
\(
\Phi=\Psi\circ E_h,
\Psi\in\mathcal K,
h\in\operatorname{BiLip}(\mathbb T),
\)
yielding the semidirect product structure of
\(\mathcal P\). We then give geometric characterizations of \(\mathcal K\), both
within \(\mathcal P\) and for arbitrary mappings of \(\mathbb D\).

\begin{lem}
\label{lem:split-monoid-decomposition}
Let \(M\) be a monoid, let \(G\) be a group, and let
\(
\pi:M\longrightarrow G
\)
be a unital monoid homomorphism. Suppose that there exists a unital
monoid homomorphism
\(
s:G\longrightarrow M
\)
such that
\(
\pi\circ s=\operatorname{id}_{G}.
\)
Set
\(
K:=\ker\pi
=
\left\{
x\in M:\pi(x)=e_G
\right\}.
\)
Then the following assertions hold.

\begin{enumerate}
\item[\textup{(i)}]
For \(g\in G\) and \(k\in K\), define
\(
\alpha_g(k)
:=
s(g)\,k\,s(g^{-1}).
\)
Then
\(
\alpha:G\longrightarrow\operatorname{Aut}(K),
\
g\mapsto\alpha_g,
\)
is an action of \(G\) on \(K\) by monoid automorphisms.

\item[\textup{(ii)}]
Every \(x\in M\) admits a unique factorization
\(
x=k\,s(g),
\
k\in K,\ g\in G.
\)
More precisely,
\(
g=\pi(x),
\
k=x\,s\bigl(\pi(x)^{-1}\bigr).
\)

\item[\textup{(iii)}]
Equip \(K\times G\) with the multiplication
\(
(k_1,g_1)(k_2,g_2)
:=
\left(
k_1\alpha_{g_1}(k_2),
g_1g_2
\right).
\)
Then
\(
\Xi:
K\rtimes_{\alpha}G
\longrightarrow M,
\
\Xi(k,g):=k\,s(g),
\)
is an isomorphism of monoids. In particular,
\(
M\cong K\rtimes_{\alpha}G,
\)
and
\(
1\longrightarrow K
\longrightarrow M
\overset{\pi}{\longrightarrow}G
\longrightarrow1
\)
is a split exact sequence of monoids.
\end{enumerate}
\end{lem}

\begin{proof}
Since \(s\) is a unital homomorphism, \(s(g)^{-1}=s(g^{-1})\).
Conjugation by \(s(g)\) therefore preserves \(K\) and defines the
claimed action \(\alpha\). For \(x\in M\), set
\(g=\pi(x)\) and \(k=x\,s(g^{-1})\). Then \(\pi(k)=e_G\)
and \(x=k\,s(g)\). Conversely, any such factorization forces
\(g=\pi(x)\) and \(k=x\,s(g^{-1})\), proving uniqueness.
Finally,
\[
\bigl(k_1s(g_1)\bigr)\bigl(k_2s(g_2)\bigr)
=k_1\alpha_{g_1}(k_2)s(g_1g_2).
\]
Thus the bijection \(\Xi(k,g)=k\,s(g)\) respects multiplication,
which proves the semidirect product assertion.
\end{proof}

\begin{thm}
\label{thm:kernel-angular-classification-of-universal-preservers}
One has
\(
\mathcal P
=
\left\{
\Psi\circ E_h:
\Psi\in\mathcal K,\
h\in\operatorname{BiLip}(\mathbb T)
\right\}.
\) Moreover, for every \(\Phi\in\mathcal P\), there exists a unique pair
\(
(\Psi,h)
\in
\mathcal K\times\operatorname{BiLip}(\mathbb T)
\)
such that
\(
\Phi=\Psi\circ E_h.
\)
More precisely, \(h=h_{\Phi}\), and
\(
\Psi=\Phi\circ E_{h_{\Phi}^{-1}}
\).
\end{thm}
\begin{proof}
Define
\(
\mathfrak E:
\operatorname{BiLip}(\mathbb T)
\longrightarrow
\mathcal P,
\
\mathfrak E(h):=E_h.
\)
By Theorem~\ref{thm:angular-classification-without-continuity},
\(E_h\in\mathcal P\) and
\(
\mathfrak S(E_h)=h.
\)
Moreover,
\(
E_{h_1}\circ E_{h_2}
=
E_{h_1\circ h_2},
\)
so \(\mathfrak E\) is a monoid homomorphism and
\(
\mathfrak S\circ\mathfrak E
=
\operatorname{id}_{\operatorname{BiLip}(\mathbb T)}.
\)
Applying Lemma~\ref{lem:split-monoid-decomposition} with
\(
M=\mathcal P,
\
G=\operatorname{BiLip}(\mathbb T),
\
\pi=\mathfrak S,
\
s=\mathfrak E,
\)
shows that every \(\Phi\in\mathcal P\) has the unique factorization
\(
\Phi=\Psi\circ E_h,
\
\Psi\in\mathcal K,
\
h\in\operatorname{BiLip}(\mathbb T).
\)
The formulas in the lemma give
\(
h=\mathfrak S(\Phi)=h_\Phi,
\
\Psi
=
\Phi\circ E_{h_\Phi^{-1}}.
\)
\end{proof}

\begin{cor}
\label{cor:semidirect-product-structure-of-P}
For
\(
h\in\operatorname{BiLip}(\mathbb T)
\)
and
\(
\Psi\in\mathcal K,
\)
define
\(
\alpha_h(\Psi)
:=
E_h\circ\Psi\circ E_{h^{-1}}.
\)
Then
\(
\alpha:
\operatorname{BiLip}(\mathbb T)
\rightarrow
\operatorname{Aut}(\mathcal K),
\
h\mapsto\alpha_h,
\)
is an action of the group
\(
\operatorname{BiLip}(\mathbb T)
\)
on the monoid \(\mathcal K\) by monoid automorphisms.

Equip
\(
\mathcal K\times\operatorname{BiLip}(\mathbb T)
\)
with the multiplication
\[
(\Psi_1,h_1)(\Psi_2,h_2)
:=
\left(
\Psi_1\circ\alpha_{h_1}(\Psi_2),
h_1\circ h_2
\right).
\]
Then the mapping
\(
\Xi:
\mathcal K\rtimes_{\alpha}\operatorname{BiLip}(\mathbb T)
\rightarrow
\mathcal P,
\
\Xi(\Psi,h):=\Psi\circ E_h,
\)
is an isomorphism of monoids.
Moreover, the boundary shadow homomorphism
\(
\mathfrak S:
\mathcal P
\rightarrow
\operatorname{BiLip}(\mathbb T),
\
\mathfrak S(\Phi)=h_\Phi,
\)
fits into the following split exact sequence of monoids:
\[
1
\rightarrow
\mathcal K
\rightarrow
\mathcal P
\overset{\mathfrak S}{\rightarrow}
\operatorname{BiLip}(\mathbb T)
\rightarrow
1.
\]
\end{cor}

\begin{proof}
This is precisely Lemma~\ref{lem:split-monoid-decomposition}
applied to
\(
\pi=\mathfrak S
\)
and
\(
s(h)=E_h.
\)
\end{proof}

The following proposition gives geometric characterizations of the
mappings in \(\mathcal P\) that belong to \(\mathcal K\).
\begin{prop}
\label{prop:geometric-characterization-of-K}
Let \(\Psi\in\mathcal P\). Then the following statements are
equivalent:
\begin{enumerate}
\item[\textup{(i)}]
\(
\Psi\in\mathcal K.
\)

\item[\textup{(ii)}]
There exists a constant \(M<1\) such that
\(
\rho\bigl(\Psi(z),z\bigr)\leq M\) for \(z\in\mathbb D\).

\item[\textup{(iii)}]
There exist constants \(M<1\) and \(0\leq r_0<1\) such that
\(
\rho\bigl(\Psi(z),z\bigr)\leq M\) for \(|z|\geq r_0.\)
\end{enumerate}
Equivalently,
\(
\mathcal K
=
\left\{
\Psi\in\mathcal P:
\sup_{z\in\mathbb D}
\rho\bigl(\Psi(z),z\bigr)<1
\right\}.
\)
\end{prop}

\begin{proof}
Assume first that \(\Psi\in\mathcal K\). Then
\(
h_\Psi=\operatorname{id}_{\mathbb T}.
\)
By Corollary~\ref{cor:radial-boundary-trace}, there exists \(C>0\)
such that
\(
\left|
\Psi\bigl((1-t)\zeta\bigr)-\zeta
\right|
\leq Ct
\)
for all \(\zeta\in\mathbb T\) and all sufficiently small \(t>0\).
Hence, for \(z=(1-t)\zeta\),
\[
|\Psi(z)-z|
\leq
|\Psi(z)-\zeta|+|\zeta-z|
\leq
(C+1)t.
\]
Since \(\Psi\in\mathcal P\), there exists \(c>0\) such that
\(
1-|\Psi(z)|^2
\geq
c\bigl(1-|z|^2\bigr).
\)
Using
\[
|1-\overline{\Psi(z)}z|^2
=
|\Psi(z)-z|^2
+
\bigl(1-|\Psi(z)|^2\bigr)\bigl(1-|z|^2\bigr),
\]
we obtain
\(
\rho\bigl(\Psi(z),z\bigr)^2
\leq
\frac{(C+1)^2t^2}
{(C+1)^2t^2+c(2t-t^2)^2}.
\)
Since
\[
\lim_{t\to0^+}
\frac{(C+1)^2t^2}
{(C+1)^2t^2+c(2t-t^2)^2}
=
\frac{(C+1)^2}{(C+1)^2+4c}
<1,
\]
there exist \(t_0>0\) and \(M<1\) such that
\(
\rho\bigl(\Psi((1-t)\zeta),(1-t)\zeta\bigr)
\le M\) for \(0<t<t_0\).
After decreasing \(t_0\), if necessary, so that \(t_0<1\), set
\(
r_0:=1-\frac{t_0}{2}.
\)
We obtain
\(
\rho(\Psi(z),z)\le M\) for \(|z|\ge r_0.\)
Hence condition~\textup{(iii)} holds.

Conversely, suppose that \textup{(iii)} holds. By Lemma~\ref{lem:pseudohyperbolic-polar-estimates}, for
\(z=(1-t)\zeta\),
\(
d_{\mathbb T}
\left(
\frac{\Psi((1-t)\zeta)}
{|\Psi((1-t)\zeta)|},
\zeta
\right)
\rightarrow0
\)
uniformly in \(\zeta\) as \(t\to0^+\). Therefore
\(
h_\Psi(\zeta)=\zeta\) for \(\zeta\in\mathbb T\). Hence \(\Psi\in\mathcal K\).

It remains to show that \textup{(iii)} implies \textup{(ii)}.
Since \(\Psi\in\mathcal P\), there exists \(c>0\) such that
\(
1-|\Psi(z)|^2
\geq
c\bigl(1-|z|^2\bigr).
\)
Consequently, if \(|z|\leq r_0\), then
\(
|\Psi(z)|
\leq
\sqrt{1-c(1-r_0^2)}
=:s_0<1.
\)
It follows that
\(
\rho\bigl(\Psi(z),z\bigr)
\leq
\frac{r_0+s_0}{1+r_0s_0}
<1
\) for \(|z|\leq r_0\).
Combining this estimate with \textup{(iii)} proves
\textup{(ii)}.
\end{proof}
Before characterizing arbitrary mappings in \(\mathcal K\), we present the following lemma.

For an arc \(I\subset\mathbb T\) and a constant \(A\geq1\), define
\(I^{(A)}:=\mathbb T\) if \(A|I|\geq1\); otherwise, let \(I^{(A)}\)
be the arc with the same midpoint as \(I\). In either case, its
normalized arc length is
\(
|I^{(A)}|:=\min\{1,A|I|\}.
\)
\begin{lem}
\label{lem:bounded-displacement-enlarges-carleson-squares}
For each \(M\in(0,1)\), there exists \(A_M\geq1\) such that, for every
arc \(I\subset\mathbb T\) and all \(z,w\in\mathbb D\), if
\(
w\in S(I)\)
and
\(\rho(z,w)\leq M,
\)
then
\(
z\in S\bigl(I^{(A_M)}\bigr).
\)
\end{lem}

\begin{proof}
Set
\(
K_M:=\frac{2M}{1-M},
\
C_M:=1+K_M=\frac{1+M}{1-M},
\
A_M:=\frac{2(1+M)}{1-M}.
\)

If \(A_M|I|\geq1\), then
\(
I^{(A_M)}=\mathbb T,
\)
and hence
\(
S\bigl(I^{(A_M)}\bigr)=\mathbb D.
\)
Thus the conclusion is immediate.

It remains to consider the case
\(
A_M|I|<1.
\)
Put
\(
r:=|w|,
\
R:=|z|.
\)
Since \(w\in S(I)\), one has
\(
1-r\leq |I|<\frac{1}{A_M}.
\)

Since \(\rho(z,w)\leq M\),
\(
|z-w|
\leq
M|1-\overline wz|.
\)
Moreover,
\(
1-\overline wz
=
1-r^2+\overline w(w-z).
\)
Therefore,
\(
|z-w|
\leq
M\bigl(1-r^2\bigr)
+
Mr|z-w|.
\)
Hence
\begin{equation}\label{eq:pseudohyperbolic-euclidean-control}
|z-w|
\leq
\frac{M(1-r^2)}{1-Mr}
\leq
\frac{M(1-r^2)}{1-M}
\leq
\frac{2M}{1-M}(1-r)
=
K_M(1-r).
\end{equation}

Since \(|R-r|\leq |z-w|\), we obtain
\begin{align}\label{eq:pseudohyperbolic-radial-control}
1-R
\leq
1-r+|R-r|
&\leq
1-r+|z-w|
\leq
C_M(1-r).
\end{align}
Since \(1-r<\frac{1}{A_M}\), it follows that
\(
1-R
<
\frac{C_M}{A_M}
=
\frac{1}{2}.
\)
Thus
\(
R>\frac{1}{2}.
\)

Write \(w=r\eta\) and \(z=R\zeta\), where
\(\zeta,\eta\in\mathbb T\).
Then
\[
R|\zeta-\eta|
\leq
|R\zeta-r\eta|
+
|r\eta-R\eta|
=
|z-w|+|r-R|
\leq
2|z-w|.
\]
Since \(R>\frac{1}{2}\), \eqref{eq:pseudohyperbolic-euclidean-control} gives
\(
|\zeta-\eta|
\leq
4|z-w|
\leq
4K_M(1-r).
\)
Consequently,
\begin{align}
d_{\mathbb T}(\zeta,\eta)
\leq
\frac14|\zeta-\eta|
\leq
K_M(1-r)
\leq
K_M|I|.
\label{eq:pseudohyperbolic-angular-control}
\end{align}

Let \(m_I\) denote the midpoint of \(I\). Since \(\eta\in I\), we have
\(
d_{\mathbb T}(\eta,m_I)\leq \frac{|I|}{2}.
\)
Therefore, by \eqref{eq:pseudohyperbolic-angular-control},
\[
d_{\mathbb T}(\zeta,m_I)
\leq
d_{\mathbb T}(\zeta,\eta)
+
d_{\mathbb T}(\eta,m_I)
\leq
\left(K_M+\frac12\right)|I|
\leq
C_M|I|
=
\frac{A_M|I|}{2}.
\]
Since \(A_M|I|<1\), the arc \(I^{(A_M)}\) has normalized length
\(
|I^{(A_M)}|=A_M|I|.
\)
Hence
\(
\zeta\in I^{(A_M)}.
\)

Finally, by \eqref{eq:pseudohyperbolic-radial-control},
\[
1-|z|
=
1-R
\leq
C_M(1-r)
\leq
C_M|I|
\leq
A_M|I|
=
|I^{(A_M)}|.
\]
Thus
\(
z\in S\bigl(I^{(A_M)}\bigr).
\)
\end{proof}

For \(0\le r_0<1\), set
\(
E_{r_0}
:=
\left\{z\in\mathbb D:|z|\ge r_0\right\}.
\) The following theorem gives a geometric characterization of arbitrary mappings in \(\mathcal K\).

\begin{thm} \label{thm:equivalence-of-two-kernel-definitions}  Let
\(
\Psi:\mathbb D\to\mathbb D
\)
be a mapping. Then \(\Psi\in\mathcal K\)  if and only if the following conditions hold:
\begin{enumerate}
\item[\textup{(K0)}]
\(\Psi\) is injective.

\item[\textup{(K1)}]
For every sequence
\(\{z_n\}_{n\geq 1}\subset\mathbb D\),
\(
|z_n|\rightarrow 1
\Longleftrightarrow
|\Psi(z_n)|\rightarrow 1.\)

\item[\textup{(K2)}]
There exist constants \(0<M<r_0<1\) such that
\(
\rho\bigl(\Psi(z),z\bigr)\leq M\) for \(z\in E_{r_0}\).
\item[\textup{(K3)}]
For any two sequences
\(
\{z_n\}_{n\geq1},
\{w_n\}_{n\geq1}
\subset\mathbb D
\)
such that
\(
|z_n|\rightarrow 1
\) and
\(|w_n|\rightarrow 1,
\)
one has
\(
\rho(z_n,w_n)\rightarrow0
\Longleftrightarrow
\rho\bigl(\Psi(z_n),\Psi(w_n)\bigr)\rightarrow0.
\)
\end{enumerate}
\end{thm}
\begin{proof}
Assume first that
\(
\Psi\in\mathcal K.
\)
By the definition of \(\mathcal K\), one has
\(
\Psi\in\mathcal P
\)
and
\(
h_\Psi=\operatorname{id}_{\mathbb T}.
\)
Conditions~\textup{(K0)} and~\textup{(K3)} follow from
Proposition~\ref{prop:basic-properties-of-universal-preservers}
\textup{(i)} and~\textup{(ii)}, respectively.

By
Theorem~\ref{thm:universal-orbit-frame-preservation}\textup{(ii)},
there exist constants \(c,C>0\) such that
\(
c\bigl(1-|z|^2\bigr)
\leq
1-|\Psi(z)|^2
\leq
C\bigl(1-|z|^2\bigr)\) for \(z\in\mathbb D\).
Consequently, for every sequence
\(\{z_n\}_{n\geq1}\subset\mathbb D\),
\(
|z_n|\rightarrow1
\Longleftrightarrow
|\Psi(z_n)|\rightarrow1.
\)
Thus \textup{(K1)} holds.

Finally, by
Proposition~\ref{prop:geometric-characterization-of-K}, there exist
\(M_0<1\) and \(s_0<1\) such that
\(
\rho\bigl(\Psi(z),z\bigr)\leq M_0\) for \(|z|\geq s_0\).
Choose
\(
r_0\in\bigl(\max\{M_0,s_0\},1\bigr)
\) and
\(M\in(M_0,r_0)\). For \(z\in E_{r_0}\) we have
\(
\rho\bigl(\Psi(z),z\bigr)\leq M\). Hence \textup{(K2)} holds.

We now prove the converse. Let \(r_0\) and \(M\) be as in condition \textup{(K2)}, and set
\[
K_M:=\frac{2M}{1-M},
\qquad
D_M:=1+K_M,
\qquad
a_M:=\frac{1}{2D_M},
\qquad
b_M:=2D_M.
\]

We first record two elementary consequences of \textup{(K1)} and
\textup{(K2)}.

By condition~\textup{(K1)}, for every \(0\leq r<1\), there exists \(R_r<1\)
such that
\begin{equation}
|z|\leq r
\quad\Longrightarrow\quad
|\Psi(z)|\leq R_r.
\label{eq:kernel-compact-core-control}
\end{equation}
Next, let \(u\in E_{r_0}\), and put
\(
v:=\Psi(u).
\)
By \textup{(K2)},
\(
\rho(u,v)\leq M.
\)
Thus,
\(
|u-v|
\leq
M|1-\overline{u}v|.
\)
Using
\(
1-\overline{u}v
=
1-|u|^2+\overline{u}(u-v),
\)
we obtain
\(
|u-v|
\leq
M\bigl(1-|u|^2\bigr)
+
M|u|\,|u-v|.
\)
Hence
\[
|u-v|
\leq
\frac{M(1-|u|^2)}{1-M|u|}
\leq
\frac{2M}{1-M}(1-|u|)
=
K_M(1-|u|).
\]
Similarly, using
\(
|1-\overline{u}v|
=
|1-\overline{v}u|
\)
and
\(
1-\overline{v}u
=
1-|v|^2+\overline{v}(v-u),
\)
we obtain
\(
|u-v|
\leq
K_M(1-|v|).
\)
Consequently,
\(
1-|v|
\leq
1-|u|+|u-v|
\leq
D_M(1-|u|),
\)
and, symmetrically,
\(
1-|u|
\leq
D_M(1-|v|).
\)
Since for $\xi\in\mathbb D$,
\(
1-|\xi|
\leq
1-|\xi|^2
\leq
2(1-|\xi|),
\)
it follows that
\begin{equation}
a_M\bigl(1-|u|^2\bigr)
\leq
1-|v|^2
\leq
b_M\bigl(1-|u|^2\bigr),
\qquad
u\in E_{r_0}.
\label{eq:kernel-tail-weight-comparability}
\end{equation}

We first prove the forward implication. Assume that
\(
\Lambda\in\mathcal C.
\)
Set
\(
\Lambda_*:=\Lambda\cap E_{r_0}.
\)
By Lemma~\ref{lem:basic-properties-carleson-sequences}\textup{(ii)},
the set \(\Lambda\setminus\Lambda_*\) is finite. Therefore,
Lemma~\ref{lem:basic-properties-carleson-sequences}\textup{(iii)}
gives
\(
\Lambda_*\in\mathcal C.
\)

We claim that \(\Psi(\Lambda_*)\) is pairwise separated. Otherwise,
there would exist distinct points
\(\lambda_n,\mu_n\in\Lambda_*\) such that
\(
\rho\bigl(\Psi(\lambda_n),\Psi(\mu_n)\bigr)\rightarrow0.
\)
We first show that
\(
|\lambda_n|\rightarrow1,
\
|\mu_n|\rightarrow1.
\)

Suppose that \(|\lambda_n|\not\to1\). Passing to a subsequence, we may
assume that
\(
|\lambda_n|\leq r
\)
for some \(r<1\) and all \(n\ge1\). By
Lemma~\ref{lem:basic-properties-carleson-sequences}\textup{(ii)},
the disk \(\overline{D(0,r)}\) contains only finitely many points of
\(\Lambda_*\). Hence, after passing to a further subsequence, we may
assume that
\(
\lambda_n=\lambda
\)
for all \(n\geq1\). Moreover,
\(
\left|\Psi(\mu_n)-\Psi(\lambda)\right|
\leq
2\rho\bigl(\Psi(\mu_n),\Psi(\lambda)\bigr)
\rightarrow0.
\)
Thus
\(
\Psi(\mu_n)\rightarrow\Psi(\lambda)\in\mathbb D,
\)
so the sequence \(\{\Psi(\mu_n)\}\) is contained in a compact subset
of \(\mathbb D\).

If \(\{\mu_n\}\) contained infinitely many distinct points, local
finiteness of \(\Lambda\) would yield a subsequence, still denoted by
\(\{\mu_n\}\), such that \(|\mu_n|\to1\). By \textup{(K1)}, this would
imply \(|\Psi(\mu_n)|\to1\), contradicting
\(\Psi(\mu_n)\to\Psi(\lambda)\in\mathbb D\). Hence, after passing to a further subsequence, we may assume
that
\(
\mu_n=\mu
\)
is independent of \(n\). Since \(\lambda\neq\mu\), injectivity of
\(\Psi\) gives
\(
\rho\bigl(\Psi(\lambda),\Psi(\mu)\bigr)>0,
\)
contradicting
\(
\rho\bigl(\Psi(\lambda_n),\Psi(\mu_n)\bigr)\rightarrow0.
\)
Therefore,
\(
|\lambda_n|\rightarrow1.
\)
The same argument, with \(\lambda_n\) and \(\mu_n\) interchanged, gives
\(
|\mu_n|\rightarrow1.
\)

Condition \textup{(K3)} now yields
\(
\rho(\lambda_n,\mu_n)\rightarrow0,
\)
which contradicts the uniform separation of \(\Lambda_*\). Hence
\(
\Psi(\Lambda_*)
\)
is pairwise separated.

Let \(I\subset\mathbb T\) be an arc. If
\(
\lambda\in\Lambda_*,
\
\Psi(\lambda)\in S(I),
\)
then
\(
\rho\bigl(\lambda,\Psi(\lambda)\bigr)\leq M.
\)
Applying
Lemma~\ref{lem:bounded-displacement-enlarges-carleson-squares}
we obtain
\(
\lambda\in S\bigl(I^{(A_M)}\bigr).
\)
Therefore, by \eqref{eq:kernel-tail-weight-comparability},
\[
\mu_{\Psi(\Lambda_*)}\bigl(S(I)\bigr)
=
\sum_{\substack{\lambda\in\Lambda_*\\
\Psi(\lambda)\in S(I)}}
\bigl(1-|\Psi(\lambda)|^2\bigr)
\leq
b_M
\sum_{\substack{\lambda\in\Lambda_*\\
\lambda\in S(I^{(A_M)})}}
\bigl(1-|\lambda|^2\bigr)
\leq
b_M\mu_\Lambda\bigl(S(I^{(A_M)})\bigr).
\]
Since \(\Lambda\in\mathcal C\), the measure \(\mu_\Lambda\) is a
Carleson measure. Hence there exists \(C_{\mu_\Lambda}>0\) such that
\(
\mu_\Lambda\bigl(S(I^{(A_M)})\bigr)
\leq
C_{\mu_\Lambda} |I^{(A_M)}|
\leq
C_{\mu_\Lambda} A_M|I|.
\)
Thus
\(
\mu_{\Psi(\Lambda_*)}\bigl(S(I)\bigr)
\leq
b_MC_{\mu_\Lambda} A_M|I|.
\)
Hence \(\mu_{\Psi(\Lambda_*)}\) is a Carleson measure. Since
\(\Psi(\Lambda_*)\) is pairwise separated, we obtain
\(
\Psi(\Lambda_*)\in\mathcal C.
\)

Since \(\Psi\) is injective and \(\Lambda\setminus\Lambda_*\) is
finite, \(\Psi(\Lambda)\) is a finite perturbation of
\(\Psi(\Lambda_*)\). Hence \(\Psi(\Lambda)\in\mathcal C\).

We next prove the reverse implication. Assume that
\(
\Psi(\Lambda)\in\mathcal C.
\)
We claim that
\(
\Lambda\setminus\Lambda_*
\)
is finite. Otherwise, there would exist infinitely many distinct points
\(
\lambda\in\Lambda
\)
satisfying
\(
|\lambda|<r_0.
\)
By \eqref{eq:kernel-compact-core-control}, there exists \(R_{r_0}<1\)
such that
\(
|\Psi(\lambda)|\leq R_{r_0}.
\)
Since \(\Psi\) is injective, this would produce infinitely many distinct
points of \(\Psi(\Lambda)\) in the compact disk
\(
\left\{
z\in\mathbb D:
|z|\leq R_{r_0}
\right\},
\)
contradicting
Lemma~\ref{lem:basic-properties-carleson-sequences}\textup{(ii)}. Thus
\(
\Lambda\setminus\Lambda_*
\)
is finite. Consequently,
\(
\Psi(\Lambda_*)\in\mathcal C.
\)

We now prove that \(\Lambda_*\) is pairwise separated. Suppose, to the
contrary, that there exist distinct points
\(
\lambda_n,\mu_n\in\Lambda_*
\)
such that
\(
\rho(\lambda_n,\mu_n)\rightarrow0.
\)
We claim that
\(
|\lambda_n|\rightarrow1,
\
|\mu_n|\rightarrow1.
\)

Suppose that \(|\lambda_n|\not\to1\). Passing to a subsequence, there
exists \(r<1\) satisfying
\(
|\lambda_n|\leq r,
\) for $n\geq1$.
Since
\(
|\lambda_n-\mu_n|
\leq
2\rho(\lambda_n,\mu_n)
\rightarrow0,
\)
there exists \(r'<1\) such that
\(
|\mu_n|\leq r'
\)
for all \(n\geq1\).
By \eqref{eq:kernel-compact-core-control}, the two image sequences
\(
\{\Psi(\lambda_n)\}_{n\geq1}
\) and
\(
\{\Psi(\mu_n)\}_{n\geq1}
\)
remain in compact disks of \(\mathbb D\). Since
\(
\Psi(\Lambda_*)\in\mathcal C,
\)
it is locally finite. Hence each of the two image sequences takes only
finitely many values. Passing to a common
subsequence, we may therefore assume that
\(
\Psi(\lambda_n)=\xi,
\
\Psi(\mu_n)=\omega\) for $n\geq1.$
 By injectivity of \(\Psi\), it follows that
\(
\lambda_n=\lambda,
\
\mu_n=\mu
\)
are independent of \(n\). Since \(\lambda\neq\mu\), this contradicts
\(
\rho(\lambda_n,\mu_n)\rightarrow0.
\)
Therefore, \(|\lambda_n|\to1\). Since
\(
|\lambda_n-\mu_n|\rightarrow0,
\)
it follows that \(|\mu_n|\to1\).

Condition \textup{(K3)} now implies
\(
\rho\bigl(\Psi(\lambda_n),\Psi(\mu_n)\bigr)\rightarrow0,
\)
which contradicts the uniform separation of \(\Psi(\Lambda_*)\).
Hence
\(
\Lambda_*
\)
is pairwise separated.

Let \(I\subset\mathbb T\) be an arc. If
\(
\lambda\in\Lambda_*,
\
\lambda\in S(I),
\)
then
\(
\rho\bigl(\Psi(\lambda),\lambda\bigr)\leq M.
\)
Applying
Lemma~\ref{lem:bounded-displacement-enlarges-carleson-squares}
we obtain
\(
\Psi(\lambda)\in S\bigl(I^{(A_M)}\bigr).
\)
Using \eqref{eq:kernel-tail-weight-comparability}, we obtain
\[
\mu_{\Lambda_*}\bigl(S(I)\bigr)
=
\sum_{\substack{\lambda\in\Lambda_*\\
\lambda\in S(I)}}
\bigl(1-|\lambda|^2\bigr)
\leq
a_M^{-1}
\sum_{\substack{\lambda\in\Lambda_*\\
\Psi(\lambda)\in S(I^{(A_M)})}}
\bigl(1-|\Psi(\lambda)|^2\bigr)
\leq
a_M^{-1}
\mu_{\Psi(\Lambda_*)}\bigl(S(I^{(A_M)})\bigr).
\]
Since \(\Psi(\Lambda_*)\in\mathcal C\), its associated discrete measure
is a Carleson measure. Therefore, there exists
\(C_{\Psi(\Lambda_*)}>0\) such that
\(
\mu_{\Psi(\Lambda_*)}\bigl(S(I^{(A_M)})\bigr)
\leq
C_{\Psi(\Lambda_*)}A_M|I|.
\)
Thus
\(
\mu_{\Lambda_*}\bigl(S(I)\bigr)
\leq
a_M^{-1}C_{\Psi(\Lambda_*)}A_M|I|.
\)
Hence \(\mu_{\Lambda_*}\) is a Carleson measure. Since
\(\Lambda_*\) is pairwise separated, it follows that
\(
\Lambda_*\in\mathcal C.
\)
Because
\(
\Lambda\setminus\Lambda_*
\)
is finite, we conclude that
\(
\Lambda\in\mathcal C.
\) Hence, \(
\Lambda\in\mathcal C
 \Longleftrightarrow
\Psi(\Lambda)\in\mathcal C.
\)

Thus condition~\textup{(i)} of
Theorem~\ref{thm:universal-orbit-frame-preservation} holds.
It remains to verify condition~\textup{(ii)}. Let \(r_0\) and \(M\)
be as in condition~\textup{(K2)}. By
\eqref{eq:kernel-tail-weight-comparability},
\(
a_M\bigl(1-|z|^2\bigr)
\leq
1-|\Psi(z)|^2
\leq
b_M\bigl(1-|z|^2\bigr)\) for \(z\in E_{r_0}\).

Moreover, by \eqref{eq:kernel-compact-core-control}, there exists
\(R_0<1\) such that
\(
|z|\leq r_0
\Longrightarrow
|\Psi(z)|\leq R_0.
\)
Hence, for \(|z|\leq r_0\),
\(
(1-R_0^2)\bigl(1-|z|^2\bigr)
\leq
1-|\Psi(z)|^2
\leq
\frac{1}{1-r_0^2}\bigl(1-|z|^2\bigr).
\)

Set
\(
c_\Psi
:=
\min\left\{
a_M,
1-R_0^2
\right\},
\
C_\Psi
:=
\max\left\{
b_M,
\frac{1}{1-r_0^2}
\right\}.
\)
Then
\[
c_\Psi\bigl(1-|z|^2\bigr)
\leq
1-|\Psi(z)|^2
\leq
C_\Psi\bigl(1-|z|^2\bigr),
\qquad
z\in\mathbb D.
\]
Condition~\textup{(ii)} of
Theorem~\ref{thm:universal-orbit-frame-preservation} therefore holds.
Hence
\(
\Psi\in\mathcal P.
\)

Finally, condition~\textup{(K2)} implies condition
\textup{(iii)} of
Proposition~\ref{prop:geometric-characterization-of-K}. Since
\(\Psi\in\mathcal P\), that proposition gives
\(
\Psi\in\mathcal K.
\)
This completes the proof.
\end{proof}

\begin{remark}
\label{rem:construction-and-membership-test-for-P}
By Theorem~\ref{thm:kernel-angular-classification-of-universal-preservers}
and the characterization of \(\mathcal K\) in
Theorem~\ref{thm:equivalence-of-two-kernel-definitions},
the preceding classification provides both a direct method for
constructing mappings in \(\mathcal P\) and a criterion with two steps for
determining membership in \(\mathcal P\).

To construct an element of \(\mathcal P\), choose arbitrarily
\(
h\in\operatorname{BiLip}(\mathbb T)\) and
\(\Psi\in\mathcal K,
\)
and define
\(
\Phi:=\Psi\circ E_h.
\)
Then
\(
\Phi\in\mathcal P
\)
by
Theorem~\ref{thm:kernel-angular-classification-of-universal-preservers}.

Conversely, suppose that one wishes to determine whether a mapping
\(
\Phi:\mathbb D\rightarrow\mathbb D
\)
belongs to \(\mathcal P\). It suffices to proceed in two steps.
First verify that the boundary shadow exists and is bilipschitz.
If either condition fails, then
\(
\Phi\notin\mathcal P.
\)
If
\(
h_{\Phi}\in\operatorname{BiLip}(\mathbb T),
\)
define the mapping
\(
\Psi:=\Phi\circ E_{h_{\Phi}^{-1}}.
\)
Then
\(
\Phi\in\mathcal P
\Longleftrightarrow
\Psi\in\mathcal K.
\)
Indeed, in this case
\(
\Phi=\Psi\circ E_{h_{\Phi}}.
\)
Thus, determining whether \(\Phi\) belongs to \(\mathcal P\) reduces
to checking whether its boundary shadow exists and is bilipschitz and whether
its part obtained by straightening the boundary belongs to \(\mathcal K\).
\end{remark}

\begin{cor}
\label{cor:compact-core-surgery-invariance}
Let
\(
\Phi\in\mathcal P.
\) Suppose that
\(
\widetilde{\Phi}:\mathbb D\to\mathbb D
\)
is injective and that there exists \(0<r_0<1\) such that
\(
\widetilde{\Phi}(z)=\Phi(z)
\)
for \(|z|\geq r_0\), while
\(
\sup_{|z|<r_0}
|\widetilde{\Phi}(z)|<1.
\)
Then
\(
\widetilde{\Phi}\in\mathcal P
\)
and
\(
h_{\widetilde{\Phi}}
=
h_\Phi.
\)
\end{cor}

\begin{proof}
Let
\(
s_0:=\sup_{|z|<r_0}|\widetilde{\Phi}(z)|<1.
\)
Since every Carleson sequence has only finitely many terms in a compact
subdisk, and since \(\widetilde{\Phi}=\Phi\) on \(E_{r_0}\), finite
perturbation stability gives
\(
\Lambda\in\mathcal C
\Longleftrightarrow
\widetilde{\Phi}(\Lambda)\in\mathcal C.
\)
Indeed, in the reverse direction, if
\(\widetilde{\Phi}(\Lambda)\in\mathcal C\), then injectivity of
\(\widetilde{\Phi}\) and the bound \(s_0<1\) imply that \(\Lambda\) has
only finitely many terms in \(\{|z|<r_0\}\).

Moreover, the estimate for the boundary defect of \(\Phi\) remains unchanged on
\(E_{r_0}\), while on \(\{|z|<r_0\}\) one has
\(
1-s_0^2
\leq
1-|\widetilde{\Phi}(z)|^2
\leq
1
\)
and
\(
1-r_0^2
\leq
1-|z|^2
\leq
1.
\)
Hence
\(
1-|\widetilde{\Phi}(z)|^2
\asymp
1-|z|^2\) for \(z\in\mathbb D.\)

Theorem~\ref{thm:universal-orbit-frame-preservation} therefore yields
\(\widetilde{\Phi}\in\mathcal P\). Since
\(\widetilde{\Phi}=\Phi\) near \(\mathbb T\), their boundary shadows
coincide:
\(
h_{\widetilde{\Phi}}=h_\Phi.
\)
\end{proof}
The next example gives a discontinuous, nonsurjective element of
\(\mathcal P\). Its radial part has two jumps, while its angular part is bilipschitz.

\begin{example}\label{ex:discontinuous-universal-preserver}
Define
\[
R(r):=
\begin{cases}
\dfrac{r}{2},
& 0\leq r<\dfrac13,\\[2mm]
\dfrac{r}{2}+\dfrac16,
& \dfrac13\leq r<\dfrac23,\\[2mm]
r,
& \dfrac23\leq r<1.
\end{cases}
\]
Let
\[
\Psi(0):=0,
\qquad
\Psi\bigl(re^{i\theta}\bigr):=R(r)e^{i\theta},
\qquad
0<r<1.
\]
The mapping \(\Psi\) is injective, agrees with
\(\operatorname{id}_{\mathbb D}\) on
\(
\left\{
z\in\mathbb D:
|z|\geq\frac{2}{3}
\right\},
\)
and satisfies
\(
|\Psi(z)|\leq\frac{1}{2}
\) for \(|z|<\frac{2}{3}\).
Hence
Corollary~\ref{cor:compact-core-surgery-invariance}
gives
\(
\Psi\in\mathcal K.
\)

Let
\(
h(e^{i\theta})=e^{iH(\theta)},
\)
where
\[
H(\theta):=
\begin{cases}
\dfrac32\theta,
& 0\leq\theta\leq\pi,\\[2mm]
\dfrac12\theta+\pi,
& \pi<\theta<2\pi,
\end{cases}
\]
and extend \(H\) by
\(
H(\theta+2\pi)=H(\theta)+2\pi.
\)
Then
\(
h\in\operatorname{BiLip}(\mathbb T).
\)

Define
\(
\Phi:=\Psi\circ E_h.
\)
By Theorem~\ref{thm:kernel-angular-classification-of-universal-preservers},
\(
\Phi\in\mathcal P.
\)

However, \(\Phi\) is discontinuous on the circles
\(
|z|=\frac{1}{3}
\)
and
\(
|z|=\frac{2}{3},
\)
and its image omits the two annuli
\(
\left\{w\in\mathbb D:\frac{1}{6}<|w|<\frac{1}{3}\right\}
\ \text{and}\
\left\{w\in\mathbb D:\frac{1}{2}<|w|<\frac{2}{3}\right\}.
\)
Thus, \(\Phi\) is neither continuous nor surjective. In particular, it
is neither holomorphic nor quasiconformal.
\end{example}

\section{Preservation of Finitely Many Orbits}
\label{sec:finite-orbits}

We now pass from single orbits to finitely many orbits. We prove
\(\mathcal P_m=\mathcal P\) for every finite \(m\), using the
Szeg\H{o} kernel model for functions with vector values and Hartmann's decomposition.
We then derive invariance of the minimal number of generators and
the Carleson index, and identify the holomorphic members of these
classes.

Recall that
\(
A
=
\sum_{j\geq1}\lambda_jP_j,
\)
where
\(
E_j
:=
\ker\!\left(A-\lambda_jI_H\right)
\)
is the spectral subspace corresponding to \(\lambda_j\), and
\(
P_j:H\rightarrow E_j
\)
is the orthogonal projection onto \(E_j\). Thus,
\(
P_jP_k=0
\
(j\neq k),
\
\sum_{j\geq1}P_j=I_H
\)
in the strong operator topology.

We begin by recalling two necessary spectral conditions for frames generated by finitely many
orbits, which follow from
\cite[Theorem~2.3]{CabrelliMolterPaternostroPhilipp2020}.

\begin{lem}
\label{lem:necessary-spectral-conditions-for-finite-multi-orbits}
Let
\(
A
=
\sum_{j\geq1}\lambda_jP_j
\in\mathcal B(H),
\
F=(f_1,\ldots,f_m)\in H^m,
\)
where \(m\in\mathbb N^{+}\), and suppose that
\(
\mathcal O_m(A,F)
\)
is a frame for \(H\). Then the following assertions hold:
\begin{enumerate}
\item[\textup{(i)}]
For every \(j\geq1\),
\(
\dim E_j\leq m.
\)
\item[\textup{(ii)}]
The spectral set
\(
\Lambda
:=\{\lambda_j:j\geq1\}
\)
can be written as the union of at most \(m\) Carleson sequences.
\end{enumerate}
\end{lem}

\begin{remark}
\label{rem:finite-union-carleson-measure-separation}
The spectral condition in
Lemma~\ref{lem:necessary-spectral-conditions-for-finite-multi-orbits}
admits a formulation parallel to the characterization for a single orbit
above. Let
\(
\mu_\Lambda
:=
\sum_{\lambda\in\Lambda}
(1-|\lambda|^2)\delta_\lambda .
\)
Then \(\Lambda\) is a union of at most \(m\) Carleson sequences if and
only if \(\mu_\Lambda\) is a Carleson measure and \(\Lambda\) can be
partitioned into at most \(m\) pairwise pseudohyperbolically separated
subsets.

Indeed, the forward implication follows because a finite sum of
Carleson measures is again a Carleson measure. Conversely, if
\(
\Lambda
=
\Lambda_1\mathbin{\dot\cup}\cdots
\mathbin{\dot\cup}\Lambda_r\) for \(r\leq m\),

with each \(\Lambda_i\) pairwise separated, then
\(\mu_{\Lambda_i}\leq\mu_\Lambda\), so each
\(\mu_{\Lambda_i}\) is a Carleson measure. By
Lemma~\ref{thm:carleson-interpolation}, every \(\Lambda_i\) is therefore
a Carleson sequence.

Thus the spectral condition for finitely many orbits separates naturally into a
global Carleson measure condition and a finite decomposition into
pairwise separated spectral subsets. The additional difficulty in the
problem involving multiple orbits lies in the coefficient and
cluster geometry in finite dimensions, which is captured below by the
Szeg\H{o} kernel synthesis operator for functions with vector values.
\end{remark}

The kernel formulation of the problem involving a single orbit from
Section~2 suggests the natural model for finitely many generators.
We retain the notation
\(
k_\lambda
\)
and
\(
\widehat{k}_\lambda
\)
for the Szeg\H{o} kernel and its normalized form. In the
setting of a single orbit, the relevant scalar geometry is the Riesz
geometry of the family
\(
\{\widehat{k}_{\lambda_j}\}.
\)
Thus, in the setting of finitely many orbits, the scalar Riesz geometry of
normalized Szeg\H{o} kernels is replaced by the geometry of
kernel synthesis for functions with vector values and coefficient
spaces of finite dimension in
\(
H^2(\mathbb D;\mathbb C^m).
\)

For \(v\in\mathbb C^m\), we identify
\(\widehat{k}_{\lambda}\otimes v\) with the function
\(
z
\longmapsto
\widehat{k}_{\lambda}(z)v
\)
in \(H^2(\mathbb D;\mathbb C^m)\).

Fix
\(
F=(f_1,\ldots,f_m)\in H^m.
\)
For each \(j\geq1\), define
\(
V_j:E_j\rightarrow\mathbb C^m
\)
by
\begin{equation}
\label{eq:local-coefficient-operator}
V_jx
:=
\frac{1}{\sqrt{1-|\lambda_j|^2}}
\left(
\left\langle x,P_jf_1\right\rangle,
\ldots,
\left\langle x,P_jf_m\right\rangle
\right),
\qquad
x\in E_j.
\end{equation}
Then
\(
V_j\in\mathcal B(E_j,\mathbb C^m)
\)
and
\begin{equation}
\label{eq:local-coefficient-operator-norm}
\left\|V_jx\right\|_{\mathbb C^m}^2
=
\frac{1}{1-|\lambda_j|^2}
\sum_{i=1}^{m}
\left|
\left\langle x,P_jf_i\right\rangle
\right|^2,
\qquad
x\in E_j.
\end{equation}
 For \(x\in E_j\),
\(
\sum_{n\geq0}\sum_{i=1}^{m}
\left|
\left\langle x,A^nf_i\right\rangle
\right|^2
=
\left\|V_jx\right\|_{\mathbb C^m}^2.
\)
Indeed,
\(
\left\langle x,A^nf_i\right\rangle
=
\overline{\lambda_j}^{\,n}
\left\langle x,P_jf_i\right\rangle,
\)
and therefore
\[
\begin{aligned}
\sum_{n\geq0}\sum_{i=1}^{m}
\left|
\left\langle x,A^nf_i\right\rangle
\right|^2
=
\sum_{i=1}^{m}
\left|
\left\langle x,P_jf_i\right\rangle
\right|^2
\sum_{n\geq0}|\lambda_j|^{2n}
=
\frac{1}{1-|\lambda_j|^2}
\sum_{i=1}^{m}
\left|
\left\langle x,P_jf_i\right\rangle
\right|^2.
\end{aligned}
\]
If \(\mathcal O_m(A,F)\) is a frame for \(H\) with frame bounds
\(0<a\leq b<\infty\), then
\begin{equation}
\label{eq:uniform-local-coefficient-bounds}
a\|x\|^2
\leq
\|V_jx\|_{\mathbb C^m}^2
\leq
b\|x\|^2,
\qquad
x\in E_j,\quad j\geq1.
\end{equation}
The following proposition is a reformulation of the kernel
characterization in
\cite[Theorem~2.5(ii)]{CabrelliMolterPaternostroPhilipp2020}
using the coefficient operators \(V_j\) and the synthesis operator
\(S_{\Lambda,V}\). We include the proof to fix the notation.

\begin{prop}
\label{prop:orbit-frame-kernel-synthesis-characterization}
The following assertions are equivalent:
\begin{enumerate}
\item[\textup{(i)}]
The family
\(
\mathcal O_m(A,F)
=
\left\{
A^nf_i:
n\in\mathbb N,\ 1\leq i\leq m
\right\}
\)
is a frame for \(H\).

\item[\textup{(ii)}]
\(
S_{\Lambda,V}:
\bigoplus_{j\geq1}E_j
\longrightarrow
H^2(\mathbb D;\mathbb C^m),
\)
defined by
\(
S_{\Lambda,V}(x_j)_{j\geq1}
:=
\sum_{j\geq1}
\widehat{k}_{\lambda_j}\otimes V_jx_j,
\qquad
(x_j)_{j\geq1}
\in
\bigoplus_{j\geq1}E_j,
\)
is a linear operator that is well defined. Moreover, there exist constants
\(0<c\leq C<\infty\) such that
\[
c\sum_{j\geq1}\|x_j\|^2
\leq
\left\|
S_{\Lambda,V}(x_j)_{j\geq1}
\right\|_{H^2(\mathbb D;\mathbb C^m)}^2
\leq
C\sum_{j\geq1}\|x_j\|^2
\]
for every
\(
(x_j)_{j\geq1}
\in
\bigoplus_{j\geq1}E_j.
\)
\end{enumerate}
\end{prop}

\begin{proof}
Let \(U:\bigoplus_jE_j\to H\) be the unitary operator
\(U(x_j)_j=\sum_jx_j\). For finitely supported \(x=(x_j)_j\),
the definition of \(V_j\) gives
\[
\left(\sum_j\widehat{k}_{\lambda_j}\otimes V_jx_j\right)_i(z)
=\sum_{n\geq0}\langle Ux,A^nf_i\rangle z^n.
\]
Parseval's identity consequently yields
\begin{equation}
\label{eq:finite-support-kernel-frame-identity}
\left\|\sum_j\widehat{k}_{\lambda_j}\otimes V_jx_j\right\|^2
=\sum_{n\geq0}\sum_{i=1}^{m}
|\langle Ux,A^nf_i\rangle|^2.
\end{equation}
If the orbit family is a frame, its two bounds therefore hold for
the kernel synthesis operator on the algebraic direct sum.
The upper bound gives a bounded extension to \(\bigoplus_jE_j\),
and continuity preserves the lower bound. Applying the extension
to the finite truncations of \(x\) shows that its value is the
series asserted in \textup{(ii)}.

Conversely, \textup{(ii)} and
\eqref{eq:finite-support-kernel-frame-identity} give both frame
estimates on the dense subspace of vectors with finite spectral
support. The upper estimate extends their analysis operator
boundedly to \(H\). Each coordinate of this extension is
\(h\mapsto\langle h,A^nf_i\rangle\), by continuity of that
functional. The lower estimate then extends by density, proving
\textup{(i)}.
\end{proof}

The following lemma shows that a finite union of Carleson sequences,
together with uniform boundedness of the local coefficient operators,
ensures the boundedness of the associated kernel synthesis operator.

\begin{lem}
\label{lem:finite-union-carleson-bessel-bound}
Suppose that
\(
\Lambda
=
\bigcup_{\nu=1}^{N}\Lambda_{\nu},
\
\Lambda_{\nu}\in\mathcal C,
\)
and
\(
M_V
:=
\sup_{j\geq1}\|V_j\|
<
\infty.
\)
Then the operator
\[
S_{\Lambda,V}:
\bigoplus_{j\geq1}E_j
\rightarrow
H^2(\mathbb D;\mathbb C^m),\qquad
S_{\Lambda,V}(x_j)_{j\geq1}
=
\sum_{j\geq1}
\widehat{k}_{\lambda_j}\otimes V_jx_j
\] is bounded.

\end{lem}

\begin{proof}
By replacing each \(\Lambda_\nu\) with
\(\Lambda_\nu\setminus\bigcup_{\mu<\nu}\Lambda_\mu\),
omitting empty sets, and relabeling, we may assume that the sets
\(\Lambda_1,\ldots,\Lambda_N\) are pairwise disjoint. Each corresponding
normalized scalar kernel family is Bessel, with some bound
\(B_\nu<\infty\). Thus
\begin{equation}
\label{eq:component-kernel-bessel-bound}
\left\|\sum_{\lambda_j\in\Lambda_\nu}
a_j\widehat{k}_{\lambda_j}\right\|_{H^2}^2
\leq B_\nu\sum_{\lambda_j\in\Lambda_\nu}|a_j|^2.
\end{equation}
For finitely supported \(x=(x_j)_j\), put
\(y_\nu=\sum_{\lambda_j\in\Lambda_\nu}
\widehat{k}_{\lambda_j}\otimes V_jx_j\).
Applying this scalar estimate to each coordinate gives
\[
\|y_\nu\|^2
\leq B_\nu\sum_{\lambda_j\in\Lambda_\nu}\|V_jx_j\|^2
\leq B_\nu M_V^2
\sum_{\lambda_j\in\Lambda_\nu}\|x_j\|^2.
\]
Hence
\[
\|S_{\Lambda,V}x\|^2
\leq N\sum_{\nu=1}^{N}\|y_\nu\|^2
\leq NM_V^2\max_\nu B_\nu\sum_j\|x_j\|^2.
\]
The bounded extension from the algebraic direct sum exists by
density, and convergence of finite truncations identifies it with
the displayed kernel series.
\end{proof}

The following decomposition result is
Proposition~1.1 of Hartmann~\cite{Hartmann}.

\begin{lem}
\label{lem:hartmann-cluster-decomposition}
Let
\(
\Lambda=\bigcup_{r=1}^{N}\Lambda_r,
\)
where \(\Lambda_r\in\mathcal C\) for \(1\leq r\leq N\).
Then, for every \(0<\eta<1\), there exists a partition
\(
\Lambda
=
\mathop{\dot\bigcup}_{n\in\mathbb N^+}\sigma_n
\)
with the following properties:
\begin{enumerate}[\rm (i)]
\item
Each \(\sigma_n\) is finite and
\(
\sup_{n\in\mathbb N^+}|\sigma_n|\le N,
\)
where \(|E|\) denotes the cardinality of a finite set \(E\).

\item
For every \(n\in\mathbb N^+\) and all distinct
\(\lambda,\mu\in\sigma_n\),
\(
\rho(\lambda,\mu)
<
\eta.
\)

\item
There exists a constant \(\delta>0\) such that, for every choice of
points
\(
\lambda_n^0\in\sigma_n,
\ n\in\mathbb N^+,
\)
the selector
\[
\Lambda_0:=\{\lambda_n^0:n\in\mathbb N^+\}
\]
belongs to \(\mathcal C\) and satisfies
\(\delta_{\Lambda_0}\geq\delta\).

\item
There exist a constant \(M_D<\infty\) and functions
\(
D_n\in H^\infty(\mathbb D),
\ n\in\mathbb N^+,
\)
such that
\[
D_n(\lambda)
=
\begin{cases}
1, & \lambda\in\sigma_n,\\
0, & \lambda\in\Lambda\setminus\sigma_n,
\end{cases}
\]
and
\(
\sum_{n\in\mathbb N^+}|D_n(z)|
\le M_D\) for \( z\in\mathbb D\).
\end{enumerate}
\end{lem}
The following lemma reduces the existence of a global lower bound for the
kernel synthesis operator to uniform lower bounds on the finite
clusters provided by Hartmann's decomposition.

\begin{lem}
\label{lem:finite-cluster-localization}
Let
\(
\Lambda
=
\mathop{\dot\bigcup}_{q\in\mathbb N^{+}}\sigma_q
\)
be the decomposition given by Lemma~\ref{lem:hartmann-cluster-decomposition}. For each
\(q\in\mathbb N^{+}\), set
\(
J_q
:=
\{j\geq1:\lambda_j\in\sigma_q\},
\
\mathcal E_q
:=
\bigoplus_{j\in J_q}E_j,
\
S_q
:=
S_{\Lambda,V}|_{\mathcal E_q}.
\)
Assume that
\(
S_{\Lambda,V}:
\bigoplus_{j\geq1}E_j
\rightarrow
H^2(\mathbb D;\mathbb C^m)
\)
is bounded. Then \(S_{\Lambda,V}\) is bounded below if and only if
\(
\inf_{q\in\mathbb N^{+}}s_{\min}(S_q)>0,
\)
where
\(
s_{\min}(T):=\inf_{\|x\|=1}\|Tx\|.
\)

More precisely, if
\(
c
:=
\inf_{q\in\mathbb N^{+}}
s_{\min}(S_q)>0,
\)
then \(S_{\Lambda,V}\) has lower bound at least \(\frac{c}{M_D}\), where
\(M_D\) is the constant in Lemma~\ref{lem:hartmann-cluster-decomposition}.
\end{lem}

\noindent\emph{Proof outline.}
Let \(M_{D_q}\) denote multiplication by the functions in
Lemma~\ref{lem:hartmann-cluster-decomposition}\textup{(iv)}.
The row operator \((g_q)_q\mapsto\sum_q D_qg_q\) has norm at most
\(M_D\). If \(x=\bigoplus_qx_q\) has finitely many nonzero
cluster components, the reproducing property gives
\[
M_{D_q}^*S_{\Lambda,V}x=S_qx_q.
\]
Consequently, if \(c=\inf_qs_{\min}(S_q)>0\), then
\[
c^2\|x\|^2
\leq\sum_q\|S_qx_q\|^2
\leq M_D^2\|S_{\Lambda,V}x\|^2.
\]
Density gives the asserted global lower bound; the converse follows
by restriction. Appendix~\ref{app:finite-cluster-proofs} supplies the
estimate for the row operator and the complete extension argument.

For the next lemma, fix \(r,m\in\mathbb{N}^{+}\) and integers
\(
1\leq d_k\leq m
\)
for \(1\leq k\leq r\), and set
\(
E
:=
\bigoplus_{k=1}^{r}\mathbb{C}^{d_k}.
\)
For every \(n\in\mathbb{N}^{+}\) and \(1\leq k\leq r\), let
\(
z_{n,k},w_{n,k}\in\mathbb{D},
\
\alpha_{n,k}>0,
\
V_{n,k}:\mathbb{C}^{d_k}\rightarrow\mathbb{C}^{m}.
\)
Assume that, for each fixed \(n\), the points
\(
z_{n,1},\ldots,z_{n,r}
\)
are pairwise distinct, and so are
\(
w_{n,1},\ldots,w_{n,r}.
\)
Define
\(
S_n^z:E\rightarrow H^2(\mathbb{D};\mathbb{C}^{m})
\)
by
\(
S_n^z(u_k)_{k=1}^{r}
:=
\sum_{k=1}^{r}
\widehat{k}_{z_{n,k}}\otimes V_{n,k}u_k,
\)
and define
\(
S_n^w:E\rightarrow H^2(\mathbb{D};\mathbb{C}^{m})
\)
by
\(
S_n^w(u_k)_{k=1}^{r}
:=
\sum_{k=1}^{r}
\widehat{k}_{w_{n,k}}
\otimes
\alpha_{n,k}V_{n,k}u_k.
\)

\begin{lem}
\label{lem:bad-cluster-subsequence-transfer}
Suppose that there exist constants
\(
b<\infty,
\
0<\alpha_{-}\leq\alpha_{+}<\infty
\)
such that
\(
\|V_{n,k}\|\leq b,
\
\alpha_{-}\leq\alpha_{n,k}\leq\alpha_{+},
\)
for all \(n\in\mathbb{N}^{+}\) and \(1\leq k\leq r\).

Assume further that, for every subsequence
\(
(n_j)_{j\geq1}
\)
and every \(1\leq k,\ell\leq r\),
\(
\rho(z_{n_j,k},z_{n_j,\ell})\rightarrow0\)
if and only if
\(\rho(w_{n_j,k},w_{n_j,\ell})\rightarrow0.
\)
Then
\(
s_{\min}(S_n^z)\rightarrow0\) if and only if
\(s_{\min}(S_n^w)\rightarrow0.
\)
\end{lem}

\noindent\emph{Proof outline.}
The fixed finite dimensions allow simultaneous subsequential limits
of the coefficient operators, the unit vectors witnessing
degeneration, and the scalar kernel Gram matrices. The hypothesis
gives the same equivalence classes of colliding points on the two
sides. On each side, the unit vectors realizing the limiting kernel Gram
matrix agree up to unimodular factors within each collision class;
representatives of distinct classes have a positive definite Gram matrix. Thus a vanishing synthesis
norm produces a separate cancellation relation in each class.
Adjusting the components by the relative unimodular factors and the
positive limiting scalars transfers these cancellations to the
other side without annihilating the vector. Interchanging the two
families proves the reverse implication. The full argument is in
Appendix~\ref{app:finite-cluster-proofs}. Fixed finite dimensions are
essential to this compactness argument, so the lemma does not by
itself yield a preservation theorem for countably many orbits.

\begin{thm}
\label{thm:finite-multi-orbit-universal-preservers}
For every \(m\in\mathbb N^{+}\), one has
\(
\mathcal P_m=\mathcal P.
\)
\end{thm}

\begin{proof}
Fix \(m\in\mathbb N^{+}\).
We first prove that \(\mathcal P_m\subseteq\mathcal P\). Let
\(\Phi\in\mathcal P_m\), let \(A\) be an arbitrary diagonal
operator of the form considered above, and let \(f\in H\). Set
\(
F_f:=(f,0,\ldots,0)\in H^m.
\)
Adding zero vectors does not affect the frame inequalities, and hence
\(
f\in\mathcal G_1(A)
\Longleftrightarrow
F_f\in\mathcal G_m(A).
\)
Since \(\Phi\in\mathcal P_m\), it follows that
\[
f\in\mathcal G_1(A)
\Longleftrightarrow
F_f\in\mathcal G_m(A)
 \Longleftrightarrow
F_f\in\mathcal G_m\bigl(\Phi(A)\bigr)
 \Longleftrightarrow
f\in\mathcal G_1\bigl(\Phi(A)\bigr).
\]
Thus
\(
\mathcal G_1(A)=\mathcal G_1(\Phi(A))
\)
for every \(A\), so \(\Phi\in\mathcal P\).

We now prove the reverse inclusion. Let \(\Phi\in\mathcal P\). By
Theorem~\ref{thm:universal-orbit-frame-preservation}, \(\Phi\)
preserves Carleson sequences in both directions and there exist
constants \(0<c_\Phi\leq C_\Phi<\infty\) such that
\begin{equation}
\label{eq:finite-multi-boundary-defect-comparability}
c_\Phi(1-|z|^2)
\leq
1-|\Phi(z)|^2
\leq
C_\Phi(1-|z|^2),
\qquad z\in\mathbb D.
\end{equation}
Moreover, \(\Phi\) is injective by
Proposition~\ref{prop:basic-properties-of-universal-preservers}. The
same proposition shows that if
\(
|z_n|\rightarrow1,
\
|w_n|\rightarrow1,
\)
then
\begin{equation}
\label{eq:finite-multi-asymptotic-collision-preservation}
\rho(z_n,w_n)\rightarrow0
\quad\Longleftrightarrow\quad
\rho\bigl(\Phi(z_n),\Phi(w_n)\bigr)\rightarrow0.
\end{equation}

Fix a diagonal operator and a tuple of \(m\) vectors
\(
A=\sum_{j\geq1}\lambda_jP_j,
\
F=(f_1,\ldots,f_m)\in H^m.
\)
Write
\(
E_j:=P_jH,
\
\mu_j:=\Phi(\lambda_j),
\
\Lambda:=\{\lambda_j:j\geq1\},
\
M:=\{\mu_j:j\geq1\}.
\)
Since \(\Phi\) is injective, the points \(\mu_j\) are pairwise
distinct, and the spectral subspace of \(\Phi(A)\) associated with
\(\mu_j\) is again \(E_j\).

Let \(V_j,W_j:E_j\to\mathbb C^m\) be the local coefficient operators
associated with \((A,F)\) and \((\Phi(A),F)\), respectively.
Thus
\[
V_jx
=
\frac{1}{\sqrt{1-|\lambda_j|^2}}
\bigl(
\langle x,P_jf_1\rangle,
\ldots,
\langle x,P_jf_m\rangle
\bigr)\]
and
\[W_jx
=
\frac{1}{\sqrt{1-|\mu_j|^2}}
\bigl(
\langle x,P_jf_1\rangle,
\ldots,
\langle x,P_jf_m\rangle
\bigr).
\]
Consequently,
\begin{equation}
\label{eq:finite-multi-local-coefficient-scaling}
W_j=\gamma_jV_j,
\qquad
\gamma_j
:=
\left(
\frac{1-|\lambda_j|^2}{1-|\mu_j|^2}
\right)^{1/2}.
\end{equation}
By \eqref{eq:finite-multi-boundary-defect-comparability},
\begin{equation}
\label{eq:finite-multi-scaling-bounds}
C_\Phi^{-1/2}
\leq
\gamma_j
\leq
c_\Phi^{-1/2},
\qquad j\geq1.
\end{equation}

Suppose first that \(\mathcal O_m(A,F)\) is a frame. By
Lemma~\ref{lem:necessary-spectral-conditions-for-finite-multi-orbits},
\(\dim E_j\leq m\) for every \(j\), and \(\Lambda\) is the union of
finitely many Carleson sequences. Moreover,
\eqref{eq:uniform-local-coefficient-bounds} yields constants
\(0<a\leq b<\infty\) such that
\begin{equation}
\label{eq:finite-multi-source-local-bounds}
\sqrt a\|x\|
\leq
\|V_jx\|
\leq
\sqrt b\|x\|,
\qquad x\in E_j,\quad j\geq1.
\end{equation}
The bidirectional preservation of Carleson sequences implies that
\(M\) is also a finite union of Carleson sequences. Equations
\eqref{eq:finite-multi-local-coefficient-scaling}--
\eqref{eq:finite-multi-scaling-bounds} show that
\(\sup_j\|W_j\|<\infty\). Hence
Lemma~\ref{lem:finite-union-carleson-bessel-bound} implies that the
synthesis operator
\(
S_{M,W}(x_j)_{j\geq1}
:=
\sum_{j\geq1}
\widehat{k}_{\mu_j}\otimes W_jx_j
\)
is bounded. On the other hand,
Proposition~\ref{prop:orbit-frame-kernel-synthesis-characterization}
shows that
\(
S_{\Lambda,V}(x_j)_{j\geq1}
:=
\sum_{j\geq1}
\widehat{k}_{\lambda_j}\otimes V_jx_j
\)
is bounded below. Fix \(c_0>0\) such that
\begin{equation}
\label{eq:finite-multi-source-global-lower-bound}
\|S_{\Lambda,V}x\|
\geq
c_0\|x\|,
\qquad
x\in\bigoplus_{j\geq1}E_j.
\end{equation}

We claim that \(S_{M,W}\) is also bounded below. Suppose otherwise.
Apply Lemma~\ref{lem:hartmann-cluster-decomposition} to \(M\), and write the resulting Hartmann
partition as
\(
M
=
\mathop{\dot\bigcup}_{q\in\mathbb N^{+}}\tau_q.
\)
For each \(q\), set
\(
\mathcal E_q
:=
\bigoplus_{\{j:\,\mu_j\in\tau_q\}}E_j,
\
T_q
:=
S_{M,W}|_{\mathcal E_q}.
\)
Lemma~\ref{lem:finite-cluster-localization} gives
\begin{equation}
\label{eq:finite-multi-image-local-degeneration}
\inf_{q\in\mathbb N^{+}}s_{\min}(T_q)=0.
\end{equation}

Every \(T_q\) is injective. Indeed, the normalized Szeg\H{o} kernels
indexed by the finitely many distinct points of \(\tau_q\) are
linearly independent, while
\eqref{eq:finite-multi-source-local-bounds}--
\eqref{eq:finite-multi-scaling-bounds} show that every \(W_j\) is
injective. Therefore, \(T_qx=0\) implies that \(W_jx_j=0\), and hence
\(x_j=0\), for every \(j\) with \(\mu_j\in\tau_q\). Since
\(\mathcal E_q\) has finite dimension, it follows that
\begin{equation}
\label{eq:finite-multi-fixed-local-positivity}
s_{\min}(T_q)>0,
\qquad q\in\mathbb N^{+}.
\end{equation}
Combining \eqref{eq:finite-multi-image-local-degeneration} and
\eqref{eq:finite-multi-fixed-local-positivity}, we may choose pairwise
distinct indices \(q_n\) such that
\begin{equation}
\label{eq:finite-multi-bad-image-clusters}
s_{\min}(T_{q_n})\rightarrow0.
\end{equation}

By Lemma~\ref{lem:hartmann-cluster-decomposition}, the cardinalities \(|\tau_q|\) are uniformly
bounded. After passing to a subsequence, reordering within each
cluster, and passing to a further subsequence, there exist
\(r\in\mathbb N^{+}\) and integers \(1\leq d_k\leq m\) such that
\(
\tau_{q_n}
=
\{\mu_{j_{n,1}},\ldots,\mu_{j_{n,r}}\},
\
\dim E_{j_{n,k}}=d_k
\)
for every \(n\) and \(1\leq k\leq r\). Put
\(
\mu_{n,k}:=\mu_{j_{n,k}},
\
\lambda_{n,k}:=\lambda_{j_{n,k}},
\
\alpha_{n,k}:=\gamma_{j_{n,k}}.
\)
Choose a unitary operator
\(
U_{n,k}:\mathbb C^{d_k}\to E_{j_{n,k}}
\)
and define
\(
\widetilde V_{n,k}
:=
V_{j_{n,k}}U_{n,k}:
\mathbb C^{d_k}\rightarrow\mathbb C^m.
\)
On
\(
E_0:=\bigoplus_{k=1}^{r}\mathbb C^{d_k},
\)
define
\(
\widetilde S_n^{\lambda}(u_k)_{k=1}^{r}
:=
\sum_{k=1}^{r}
\widehat{k}_{\lambda_{n,k}}
\otimes
\widetilde V_{n,k}u_k
\)
and
\(
\widetilde S_n^{\mu}(u_k)_{k=1}^{r}
:=
\sum_{k=1}^{r}
\widehat{k}_{\mu_{n,k}}
\otimes
\alpha_{n,k}\widetilde V_{n,k}u_k.
\)
These are unitarily equivalent to the corresponding restrictions of
\(S_{\Lambda,V}\) and \(S_{M,W}\), respectively. Hence
\eqref{eq:finite-multi-source-global-lower-bound} and
\eqref{eq:finite-multi-bad-image-clusters} give
\begin{equation}
\label{eq:finite-multi-opposite-local-behavior}
s_{\min}(\widetilde S_n^{\lambda})
\geq
c_0,
\qquad
s_{\min}(\widetilde S_n^{\mu})
\rightarrow0.
\end{equation}

For every fixed \(k\), extend the choices
\(
\mu_{n,k}\in\tau_{q_n}
\)
to a selector of the entire Hartmann partition. By
Lemma~\ref{lem:hartmann-cluster-decomposition}\textup{(iii)}, this selector is a Carleson sequence.
Since a Carleson sequence has only finitely many points in each compact
subset of \(\mathbb D\), and since the indices \(q_n\) are pairwise
distinct, Lemma~\ref{lem:basic-properties-carleson-sequences}\textup{(ii)}
yields
\(
|\mu_{n,k}|\rightarrow1.
\)
Equation \eqref{eq:finite-multi-boundary-defect-comparability} then
implies that
\begin{equation}
\label{eq:finite-multi-cluster-points-approach-boundary}
|\lambda_{n,k}|\rightarrow1,
\qquad
|\mu_{n,k}|\rightarrow1,
\qquad
1\leq k\leq r.
\end{equation}
Furthermore, for every subsequence \((n_h)_{h\geq1}\) and every
\(1\leq k,\ell\leq r\),
\eqref{eq:finite-multi-asymptotic-collision-preservation} gives
\begin{equation}
\label{eq:finite-multi-cluster-collision-equivalence}
\rho(\lambda_{n_h,k},\lambda_{n_h,\ell})\rightarrow0
\quad\Longleftrightarrow\quad
\rho(\mu_{n_h,k},\mu_{n_h,\ell})\rightarrow0.
\end{equation}

The operators \(\widetilde V_{n,k}\) are uniformly bounded by
\eqref{eq:finite-multi-source-local-bounds}, and the scalars
\(\alpha_{n,k}\) satisfy uniform positive upper and lower bounds by
\eqref{eq:finite-multi-scaling-bounds}. Thus all the hypotheses of
Lemma~\ref{lem:bad-cluster-subsequence-transfer} are satisfied with
\(
z_{n,k}=\lambda_{n,k},
\
w_{n,k}=\mu_{n,k},
\
V_{n,k}=\widetilde V_{n,k}.
\)
That lemma and \eqref{eq:finite-multi-opposite-local-behavior} imply
\(
s_{\min}(\widetilde S_n^{\lambda})\rightarrow0,
\)
contradicting the first inequality in
\eqref{eq:finite-multi-opposite-local-behavior}. Therefore
\(S_{M,W}\) is bounded below. Since it is also bounded,
Proposition~\ref{prop:orbit-frame-kernel-synthesis-characterization}
shows that \(\mathcal O_m(\Phi(A),F)\) is a frame.

It remains to prove the converse implication. Suppose that
\(\mathcal O_m(\Phi(A),F)\) is a frame. Then \(\dim E_j\leq m\) for
every \(j\), \(M\) is a finite union of Carleson sequences, the
operators \(W_j\) have uniform positive
lower and upper bounds, and \(S_{M,W}\) is bounded below. If
\(
M=\bigcup_{\nu=1}^{N}\Gamma_\nu\) for \(\Gamma_\nu\in\mathcal C,\)
set
\(
\Lambda_\nu
:=
\{\lambda_j:\mu_j\in\Gamma_\nu\}.
\)
The injectivity of \(\Phi\) gives
\(
\Phi(\Lambda_\nu)=\Gamma_\nu,
\)
and the reverse preservation of Carleson sequences gives
\(
\Lambda_\nu\in\mathcal C.
\)
Hence
\(
\Lambda=\bigcup_{\nu=1}^{N}\Lambda_\nu
\)
is a finite union of Carleson sequences. By
\eqref{eq:finite-multi-local-coefficient-scaling}--
\eqref{eq:finite-multi-scaling-bounds}, the operators
\(
V_j=\gamma_j^{-1}W_j
\)
have uniform positive lower and upper bounds. Consequently,
Lemma~\ref{lem:finite-union-carleson-bessel-bound} shows that
\(S_{\Lambda,V}\) is bounded.

If \(S_{\Lambda,V}\) were not bounded below, apply the preceding
argument using Hartmann clusters to a Hartmann partition of \(\Lambda\). As
before, every fixed local restriction is injective and hence has
positive minimum modulus. We can therefore select pairwise distinct
clusters whose minimum moduli tend to zero and, after passing to a
subsequence, obtain local data in fixed dimensions. For these data use
in Lemma~\ref{lem:bad-cluster-subsequence-transfer}
\(
z_{n,k}=\mu_{j_{n,k}},
\
w_{n,k}=\lambda_{j_{n,k}},
\
V_{n,k}=W_{j_{n,k}}U_{n,k},
\
\alpha_{n,k}=\gamma_{j_{n,k}}^{-1}.
\)
A Hartmann selector for the \(\lambda\) points tends to the boundary;
\eqref{eq:finite-multi-boundary-defect-comparability} forces the
corresponding selector for the \(\mu\) points to tend to the boundary, and
\eqref{eq:finite-multi-asymptotic-collision-preservation} gives the
required collision equivalence along every subsequence. The local operators
for the \(w\) points are precisely the degenerating restrictions of
\(S_{\Lambda,V}\), whereas those for the \(z\) points are restrictions
of \(S_{M,W}\). Lemma~\ref{lem:bad-cluster-subsequence-transfer}
would therefore produce restrictions of \(S_{M,W}\) whose minimum
moduli tend to zero, contradicting the global lower bound of
\(S_{M,W}\). Thus \(S_{\Lambda,V}\) is bounded below, and
Proposition~\ref{prop:orbit-frame-kernel-synthesis-characterization}
yields that \(\mathcal O_m(A,F)\) is a frame.

We have proved, for every diagonal \(A\) and every \(F\in H^m\),
\(
F\in\mathcal G_m(A)
\Longleftrightarrow
F\in\mathcal G_m\bigl(\Phi(A)\bigr).
\)
Hence \(\Phi\in\mathcal P_m\), so
\(\mathcal P\subseteq\mathcal P_m\). Together with the first
inclusion, this proves \(\mathcal P_m=\mathcal P\).
\end{proof}

Define
\(
\mathfrak{s}(A)
:=
\inf\left\{
m\in\mathbb N^{+}:
\mathcal G_m(A)\neq\varnothing
\right\},
\)
with the convention that \(\mathfrak{s}(A)=\infty\) if
\(\mathcal G_m(A)=\varnothing\) for every \(m\in\mathbb N^{+}\).

\begin{cor}
\label{cor:minimal-sampling-number}
Let \(A\) be a diagonal operator and let \(\Phi\in\mathcal P\). Then
\(
\mathfrak{s}\bigl(\Phi(A)\bigr)=\mathfrak{s}(A).
\)
\end{cor}

\begin{proof}
Fix \(m\in\mathbb N^{+}\). By
Theorem~\ref{thm:finite-multi-orbit-universal-preservers},
\(
\mathcal P_m=\mathcal P.
\)
Hence \(\Phi\in\mathcal P_m\), and therefore
\(
\mathcal G_m\bigl(\Phi(A)\bigr)=\mathcal G_m(A).
\)
In particular,
\(
\mathcal G_m(A)\neq\varnothing
\Longleftrightarrow
\mathcal G_m\bigl(\Phi(A)\bigr)\neq\varnothing.
\)
Taking the infimum over \(m\in\mathbb N^{+}\) proves the result.
\end{proof}

Motivated by the condition on finite decompositions in \cite{CabrelliMolterPaternostroPhilipp2020}, let
\(
\Lambda=(\lambda_j)_{j\geq1}\subset\mathbb D.
\)
For the purposes of this definition and the following corollary, we
identify a sequence with its underlying set; in particular,
repetitions are disregarded. Every finite set of pairwise distinct
points is regarded as a Carleson sequence. Define the Carleson index of \(\Lambda\) by
\[
\operatorname{ind}_{\mathcal C}(\Lambda)
:=
\inf\left\{
m\in\mathbb N^{+}:
\Lambda=\bigcup_{i=1}^{m}\Lambda_i,
\quad
\Lambda_i\in\mathcal C
\ \text{for }1\leq i\leq m
\right\}.
\]
If no such finite decomposition exists, we set
\(
\operatorname{ind}_{\mathcal C}(\Lambda)=\infty.
\)

\begin{cor}
\label{cor:carleson-index-invariance}
Let \(\Phi\in\mathcal P\). Then, for every sequence
\(\Lambda\subset\mathbb D\),
\(
\operatorname{ind}_{\mathcal C}\bigl(\Phi(\Lambda)\bigr)
=
\operatorname{ind}_{\mathcal C}(\Lambda).
\)
\end{cor}

\begin{proof} By Proposition~\ref{prop:basic-properties-of-universal-preservers},
\(\Phi\) is injective. Therefore, under the convention above, the
bidirectional preservation of Carleson sequences from
Theorem~\ref{thm:universal-orbit-frame-preservation} also extends to
finite sets.
Suppose first that
\(
\operatorname{ind}_{\mathcal C}(\Lambda)=m<\infty.
\)
Then
\(
\Lambda=\bigcup_{i=1}^{m}\Lambda_i
\)
for some \(\Lambda_i\in\mathcal C\). By
Theorem~\ref{thm:universal-orbit-frame-preservation},
\(
\Phi(\Lambda_i)\in\mathcal C,
\ 1\leq i\leq m.
\)
Since
\(
\Phi(\Lambda)=\bigcup_{i=1}^{m}\Phi(\Lambda_i),
\)
we obtain
\(
\operatorname{ind}_{\mathcal C}\bigl(\Phi(\Lambda)\bigr)
\leq
\operatorname{ind}_{\mathcal C}(\Lambda).
\)

Conversely, suppose that
\(
\operatorname{ind}_{\mathcal C}\bigl(\Phi(\Lambda)\bigr)
=
m<\infty.
\)
Then
\(
\Phi(\Lambda)=\bigcup_{i=1}^{m}\Gamma_i
\)
for some \(\Gamma_i\in\mathcal C\).
Hence, for
\(
\Lambda_i
:=
\left\{
\lambda\in\Lambda:
\Phi(\lambda)\in\Gamma_i
\right\},
\)
one has
\(
\Lambda=\bigcup_{i=1}^{m}\Lambda_i\) and
\(\Phi(\Lambda_i)=\Gamma_i.
\)
Another application of
Theorem~\ref{thm:universal-orbit-frame-preservation} gives
\(
\Lambda_i\in\mathcal C,
\ 1\leq i\leq m.
\)
Therefore,
\(
\operatorname{ind}_{\mathcal C}(\Lambda)
\leq
\operatorname{ind}_{\mathcal C}\bigl(\Phi(\Lambda)\bigr).
\)
Combining the two inequalities proves the result, including the case
in which either index is infinite.
\end{proof}

We conclude this section by identifying the holomorphic members
of the classes for finitely many orbits. Although \(\mathcal P\) contains
discontinuous and nonholomorphic mappings, its holomorphic members
are precisely the disk automorphisms.

\begin{thm}
\label{thm:holomorphic-rigidity}
 Let
\(
\varphi:\mathbb D\rightarrow\mathbb D
\)
be holomorphic, and let \(m\in\mathbb N^+\). Then the following assertions are equivalent:
\begin{enumerate}
\item[\textup{(i)}]
\(
\varphi\in\mathcal P_m.
\)

\item[\textup{(ii)}]
\(
\varphi\in\operatorname{Aut}(\mathbb D).
\)
\end{enumerate}
\end{thm}
\begin{proof}
By Theorem~\ref{thm:finite-multi-orbit-universal-preservers},
it suffices to prove that \(\varphi\in\mathcal P\) if and only if \(\varphi\in\operatorname{Aut}(\mathbb D).\) First, suppose that
\(\varphi\in\operatorname{Aut}(\mathbb D)\).
Then there exist
\(
a\in\mathbb D,
\
\vartheta\in\mathbb R,
\)
satisfying \(
\varphi(z)
=
e^{i\vartheta}
\frac{z-a}{1-\overline{a}z}
\) for $z\in\mathbb D.$
Since every disk automorphism preserves the pseudohyperbolic distance,
\(
\rho\bigl(\varphi(z),\varphi(w)\bigr)
=
\rho(z,w),
\
z,w\in\mathbb D.
\)
Therefore, for every sequence
\(
\Lambda=\{\lambda_j\}_{j\geq1}\subset\mathbb D,
\)
one has
\(
\Lambda\in\mathcal C
 \Longleftrightarrow
\varphi(\Lambda)\in\mathcal C.
\)

Moreover, for $z\in\mathbb D,$
\(
1-|\varphi(z)|^2
=
\frac{1-|a|^2}{|1-\overline{a}z|^2}
\bigl(1-|z|^2\bigr).
\)
Since
\(
1-|a|
\leq
|1-\overline{a}z|
\leq
1+|a|,
\)
we obtain
\(
\frac{1-|a|}{1+|a|}
\bigl(1-|z|^2\bigr)
\leq
1-|\varphi(z)|^2
\leq
\frac{1+|a|}{1-|a|}
\bigl(1-|z|^2\bigr),
\
z\in\mathbb D.
\)
By Theorem~\ref{thm:universal-orbit-frame-preservation},
\(
\varphi\in\mathcal P.
\)

Conversely, suppose that \(\varphi\in\mathcal P\). By
Proposition~\ref{prop:basic-properties-of-universal-preservers},
\(\varphi\) is injective.
Set
\(
\Omega
:=
\varphi(\mathbb D).
\)
Since \(\varphi\) is injective and holomorphic,
\(
\varphi:\mathbb D\rightarrow\Omega
\)
is a conformal bijection. We claim that
\(
\Omega=\mathbb D.
\)

Suppose, to the contrary, that
\(
\Omega\neq\mathbb D.
\)
Since \(\Omega\) is a nonempty proper open subset of the connected
domain \(\mathbb D\), its relative boundary in \(\mathbb D\) is
nonempty. Hence there exists
\(
w_0\in\partial\Omega\cap\mathbb D.
\)
Choose a sequence
\(
\{w_n\}_{n\geq1}\subset\Omega
\)
such that
\(
w_n\rightarrow w_0.
\)
For every \(n\geq1\), let
\(
z_n:=\varphi^{-1}(w_n)\in\mathbb D.
\)
We claim that
\(
|z_n|\rightarrow1.
\)

Indeed, suppose otherwise. Then, after passing to a subsequence, there
exists \(0<r<1\) such that for $n\geq1,$
\(
|z_n|\leq r.
\)
By compactness of
\(
\overline{D(0,r)},
\)
after passing to a further subsequence, we may assume that
\(
z_n\rightarrow z_0\in\mathbb D.
\)
By continuity of \(\varphi\),
\(
w_n
=
\varphi(z_n)
\rightarrow
\varphi(z_0).
\)
Since also
\(
w_n\rightarrow w_0,
\)
we obtain
\(
w_0=\varphi(z_0)\in\Omega,
\)
which contradicts
\(
w_0\in\partial\Omega.
\)
Thus,
\(
|z_n|\rightarrow1.
\)

On the other hand, since
\(
\varphi\in\mathcal P,
\)
condition \textup{(ii)} of
Theorem~\ref{thm:universal-orbit-frame-preservation} yields a constant
\(
C_\varphi>0
\)
such that for $z\in\mathbb D,$
\(
1-|\varphi(z)|^2
\leq
C_\varphi\bigl(1-|z|^2\bigr).
\)
Applying this estimate to \(z=z_n\), we obtain
\(
1-|w_n|^2
=
1-|\varphi(z_n)|^2
\leq
C_\varphi\bigl(1-|z_n|^2\bigr)
\rightarrow0.
\)
Therefore,
\(
|w_n|\rightarrow1.
\)
But
\(
w_n\rightarrow w_0\in\mathbb D,
\)
and hence
\(
|w_n|\rightarrow|w_0|<1,
\)
which is a contradiction.

Therefore,
\(
\Omega=\mathbb D.
\)
Thus \(\varphi\) is a bijective holomorphic mapping from \(\mathbb D\) onto itself,
and hence
\(
\varphi\in\operatorname{Aut}(\mathbb D).
\)
This proves the theorem.
\end{proof}

\section{Preservation of Countably Many Orbits}
\label{sec:countable-orbits}

For countably many orbits, we prove that the corresponding preservation class \(\mathcal P_\omega\) is strictly contained in \(\mathcal P\). We first establish uniform continuity of
preservers and their inverses in the pseudohyperbolic metric, and
identify the holomorphic members of the class. We then give an
explicit criterion for boundary perturbations of disk automorphisms,
prove invariance under compact modifications among disk
homeomorphisms, and construct nonholomorphic radial examples.
We finally prove a uniform matrix characterization for arbitrary
mappings of \(\mathbb D\) and apply it to complex conjugation.

For
\(
F=(f_i)_{i\geq1}\in H^{\mathbb N^{+}},
\)
set
\(
\mathcal O_\omega(A,F)
:=
\{A^n f_i\}_{n\geq0,\ i\geq1}
\)
and
\[
\mathcal G_\omega(A)
:=
\left\{
F\in H^{\mathbb N^{+}}:
\mathcal O_\omega(A,F)
\text{ is a frame for }H
\right\}.
\]
Define
\[
\mathcal P_\omega
:=
\left\{
\Phi:\mathbb D\rightarrow\mathbb D:
\mathcal G_\omega(A)
=
\mathcal G_\omega\bigl(\Phi(A)\bigr)
\text{ for every diagonal operator \(A\) as above}
\right\}.
\]

The class \(\mathcal G_\omega(A)\) is nonempty for every
diagonal operator considered in this paper. Indeed, such an operator
is a normal contraction with no unitary part, and
\cite[Theorem~5.2(b)]{IterativeNormalOperators} provides a countable
generating family whose iterates form a Parseval frame. The two geometric viewpoints developed above persist in the
setting of countably many orbits. The vector-valued kernel model now uses
\(H^2(\mathbb D;\ell^2)\), while every
\(\Phi\in\mathcal P_\omega\subseteq\mathcal P\) inherits the
Carleson measure estimates in both directions and pairwise pseudohyperbolic
separation preservation established for the class for a single orbit.
The new feature is that preservation for countably many orbits is sensitive to arbitrarily large spectral multiplicities and kernel clusters.
\begin{prop}
\label{prop:countable-preservers-are-uniform-homeomorphisms}
Let \(\Phi\in\mathcal P_\omega\). Then \(\Phi\in\mathcal P\). Moreover, \(\Phi\) is a bijection of \(\mathbb D\) and both
\(\Phi\) and \(\Phi^{-1}\) are uniformly continuous with respect to
the pseudohyperbolic metric.
\end{prop}

\begin{proof}
We first show that \(\Phi\in\mathcal P\). Fix a nonzero sequence
\(c=(c_i)_{i\geq1}\in\ell^2\). For \(f\in H\), set
\(
F_f:=(c_if)_{i\geq1}.
\)
Then
\(
\sum_{i\geq1}\sum_{k\geq0}
|
\left\langle x,A^k(c_if)\right\rangle
|^2
=
\|c\|_{\ell^2}^2
\sum_{k\geq0}
|
\left\langle x,A^kf\right\rangle
|^2.
\)
Hence
\(
F_f\in\mathcal G_\omega(A)
\Longleftrightarrow
f\in\mathcal G_1(A).
\)
Applying the defining property of \(\mathcal P_\omega\) gives
\(
f\in\mathcal G_1(A)
\Longleftrightarrow
f\in\mathcal G_1\bigl(\Phi(A)\bigr).
\)
Thus \(\Phi\in\mathcal P\).

In particular, \(\Phi\) is injective and there exist constants
\(0<c_\Phi\leq C_\Phi<\infty\) such that
\begin{equation}
\label{eq:countable-defect-comparability}
c_\Phi(1-|z|^2)
\leq
1-|\Phi(z)|^2
\leq
C_\Phi(1-|z|^2),
\qquad z\in\mathbb D.
\end{equation}

We prove that, for every pair of sequences
\(
(z_n)_{n\geq1},(w_n)_{n\geq1}\subset\mathbb D,
\)
one has
\begin{equation}
\label{eq:uniform-separation-preservation}
\inf_{n\geq1}\rho(z_n,w_n)>0
\quad\Longleftrightarrow\quad
\inf_{n\geq1}
\rho\bigl(\Phi(z_n),\Phi(w_n)\bigr)>0.
\end{equation}
Let
\(
S:=\{z_n:n\geq1\}\cup\{w_n:n\geq1\}.
\)
If \(S\) is finite, then the desired equivalence follows immediately
from the injectivity of \(\Phi\), since only finitely many ordered
pairs occur. Hence it remains to consider the case in which \(S\) is
infinite.
For each occurrence of a point in the pairs \((z_n,w_n)\), choose a
separate orthonormal vector in the corresponding spectral subspace.
Thus we may write
\(
H=\bigoplus_{n\geq1}H_n,
\
H_n=\operatorname{span}\{e_n^0,e_n^1\},
\)
where vectors corresponding to equal spectral values are understood
as belonging to the same spectral
subspace, whose dimension may be infinite. Define
\(
Ae_n^0=z_ne_n^0,
\
Ae_n^1=w_ne_n^1,
\)
and
\(
f_n
:=
\sqrt{1-|z_n|^2}\,e_n^0
+
\sqrt{1-|w_n|^2}\,e_n^1.
\)
Let \(F=(f_n)_{n\geq1}\).

The orbit generated by \(f_n\) remains in \(H_n\), and the subspaces
\(H_n\) are mutually orthogonal. Relative to
\(\{e_n^0,e_n^1\}\), the frame operator of
\(\{A^kf_n\}_{k\geq0}\) has matrix
\(
G_n
=
\begin{pmatrix}
1 & \xi_n\\
\overline{\xi_n} & 1
\end{pmatrix},
\)
where
\(
|\xi_n|
=
\frac{
\sqrt{(1-|z_n|^2)(1-|w_n|^2)}
}{
|1-\overline{z_n}w_n|
}
=
\sqrt{1-\rho(z_n,w_n)^2}.
\)
Therefore,
\(
\lambda_{\min}(G_n)
=
1-\sqrt{1-\rho(z_n,w_n)^2},
\
\lambda_{\max}(G_n)
\leq2.
\)
It follows that
\begin{equation}
\label{eq:source-two-point-frame}
F\in\mathcal G_\omega(A)
\quad\Longleftrightarrow\quad
\inf_{n\geq1}\rho(z_n,w_n)>0.
\end{equation}

For \(\Phi(A)\), the corresponding block matrix is
\(
\widetilde G_n
=
D_n
\begin{pmatrix}
1 & \eta_n\\
\overline{\eta_n} & 1
\end{pmatrix}
D_n,
\)
where
\[
|\eta_n|
=
\sqrt{
1-\rho\bigl(\Phi(z_n),\Phi(w_n)\bigr)^2
}
\]
and
\[
D_n
=
\begin{pmatrix}
\displaystyle
\sqrt{
\frac{1-|z_n|^2}
     {1-|\Phi(z_n)|^2}
}
&0\\[2mm]
0&
\displaystyle
\sqrt{
\frac{1-|w_n|^2}
     {1-|\Phi(w_n)|^2}
}
\end{pmatrix}.
\]
By \eqref{eq:countable-defect-comparability}, the matrices \(D_n\)
and \(D_n^{-1}\) are uniformly bounded. Consequently,
\begin{equation}
\label{eq:image-two-point-frame}
F\in\mathcal G_\omega\bigl(\Phi(A)\bigr)
\quad\Longleftrightarrow\quad
\inf_{n\geq1}
\rho\bigl(\Phi(z_n),\Phi(w_n)\bigr)>0.
\end{equation}
Since \(\Phi\in\mathcal P_\omega\), equations
\eqref{eq:source-two-point-frame} and
\eqref{eq:image-two-point-frame} prove
\eqref{eq:uniform-separation-preservation}. Repetitions among the
points cause no difficulty, since separate occurrences are placed in
orthogonal directions of the same spectral subspace.

We next prove uniform continuity. If \(\Phi\) were not uniformly
continuous, there would exist sequences \(z_n,w_n\) and
\(\varepsilon>0\) such that
\(
\rho(z_n,w_n)\rightarrow0,
\
\rho\bigl(\Phi(z_n),\Phi(w_n)\bigr)\geq\varepsilon.
\)
This contradicts \eqref{eq:uniform-separation-preservation}. The same
argument, with the roles of the two sides reversed, shows that
\(\Phi^{-1}\) is uniformly continuous on \(\Phi(\mathbb D)\).

It remains to prove surjectivity. Since \(\Phi\) is continuous and
injective, invariance of domain shows that \(\Phi(\mathbb D)\) is open
in \(\mathbb D\). On the other hand,
\eqref{eq:countable-defect-comparability} implies that \(\Phi\) is
proper. Indeed, if \(K\subset\mathbb D\) is compact and
\(\Phi(z)\in K\), then, for some \(R<1\),
\(
|\Phi(z)|\leq R.
\)
Hence
\(
1-|z|^2
\geq
\frac{1-R^2}{C_\Phi},
\)
so \(\Phi^{-1}(K)\) is contained in a compact subdisk. It is closed by
continuity and is therefore compact.

A proper continuous mapping has closed image. Thus
\(\Phi(\mathbb D)\) is both open and closed in the connected set
\(\mathbb D\). Therefore
\(
\Phi(\mathbb D)=\mathbb D.
\)
This completes the proof.
\end{proof}
\begin{prop}
\label{prop:automorphisms-belong-to-countable-preservers}
\(
\operatorname{Aut}(\mathbb D)
\subseteq
\mathcal P_\omega.
\)
\end{prop}

\begin{proof}
Let
\(
\varphi\in\operatorname{Aut}(\mathbb D)
\).
Write
\(
\varphi(z)
=
e^{i\vartheta}
\frac{z-a}{1-\overline{a}z},
\
a\in\mathbb D,\quad \vartheta\in\mathbb R.
\)
Fix a diagonal operator of the above form
\(
A
=
\sum_{j\geq1}\lambda_jP_j
\)
and a countable family
\(
F=(f_i)_{i\geq1}\in H^{\mathbb N^+}
\).
Set
\(
\Lambda:=(\lambda_j)_{j\geq1},
\
\mu_j:=\varphi(\lambda_j),
\
M:=(\mu_j)_{j\geq1}.
\)
Thus
\(
\varphi(A)
=
\sum_{j\geq1}\mu_jP_j.
\)

Suppose first that
\(
F\in\mathcal G_\omega(A)
\).
For \(x\in E_j:=P_jH\), define
\(
V_jx
:=
\frac{1}{\sqrt{1-|\lambda_j|^2}}
\bigl(
\langle x,P_jf_i\rangle
\bigr)_{i\geq1}
\in\ell^2.
\)
The countable analogue of
Proposition~\ref{prop:orbit-frame-kernel-synthesis-characterization}
follows by replacing \(\mathbb C^m\) with \(\ell^2\); the Bessel bound
ensures that \(V_jx\in\ell^2\), and the proof is otherwise unchanged. Hence
\(
F\in\mathcal G_\omega(A)
\)
if and only if the synthesis operator
\(
S_{\Lambda,V}:
\bigoplus_{j\geq1}E_j
\rightarrow
H^2(\mathbb D;\ell^2),
\
S_{\Lambda,V}(x_j)_{j\geq1}
=
\sum_{j\geq1}
\widehat{k}_{\lambda_j}\otimes V_jx_j,
\)
is bounded and bounded below.
For the operator \(\varphi(A)\), the corresponding local coefficient
operators are
\(
W_jx
:=
\frac{1}{\sqrt{1-|\mu_j|^2}}
\bigl(
\langle x,P_jf_i\rangle
\bigr)_{i\geq1}
=
d_jV_jx,
\)
where
\(
d_j
:=
\sqrt{
\frac{1-|\lambda_j|^2}
     {1-|\varphi(\lambda_j)|^2}
}.
\)
The standard defect identity for disk automorphisms gives
\(
1-|\varphi(z)|^2
=
\frac{1-|a|^2}{|1-\overline{a}z|^2}
\bigl(1-|z|^2\bigr).
\)
Consequently,
\(
d_j
=
\frac{|1-\overline{a}\lambda_j|}
     {\sqrt{1-|a|^2}},
\)
and therefore
\begin{equation}
\label{eq:automorphism-countable-defect-ratio}
\frac{1-|a|}{\sqrt{1-|a|^2}}
\leq
d_j
\leq
\frac{1+|a|}{\sqrt{1-|a|^2}},
\qquad j\geq1.
\end{equation}

The standard unitary action of
\(\operatorname{Aut}(\mathbb D)\) on \(H^2(\mathbb D)\) provides a
unitary operator
\(
U_\varphi:H^2(\mathbb D)\to H^2(\mathbb D)
\)
and unimodular constants
\(
\varepsilon_j\in\mathbb T
\)
such that
\(
U_\varphi\widehat{k}_{\lambda_j}
=
\varepsilon_j\widehat{k}_{\mu_j},
\ j\geq1.
\)
Define
\(
Q_\varphi:
\bigoplus_{j\geq1}E_j
\rightarrow
\bigoplus_{j\geq1}E_j,
\
Q_\varphi(x_j)_{j\geq1}
:=
\bigl(
d_j\overline{\varepsilon_j}x_j
\bigr)_{j\geq1}.
\)
By \eqref{eq:automorphism-countable-defect-ratio},
\(Q_\varphi\) is boundedly invertible. Define
\(
T_\varphi
:=
\bigl(U_\varphi\otimes I_{\ell^2}\bigr)
S_{\Lambda,V}Q_\varphi.
\)
Since \(U_\varphi\otimes I_{\ell^2}\) is unitary,
\(S_{\Lambda,V}\) is bounded and bounded below, and \(Q_\varphi\) is
boundedly invertible, the operator \(T_\varphi\) is bounded and
bounded below.

For every finitely supported
\(
x=(x_j)_{j\geq1},
\)
we have
\[
T_\varphi(x_j)_{j\geq1}
=
\bigl(U_\varphi\otimes I_{\ell^2}\bigr)
S_{\Lambda,V}Q_\varphi(x_j)_{j\geq1}
=
\sum_{j\geq1}
\varepsilon_j\widehat{k}_{\mu_j}
\otimes
d_j\overline{\varepsilon_j}V_jx_j
=
\sum_{j\geq1}
\widehat{k}_{\mu_j}\otimes W_jx_j.
\]
Now let
\(
x=(x_j)_{j\geq1}
\in
\bigoplus_{j\geq1}E_j
\)
and set
\(
x^{(N)}
:=
(x_1,\ldots,x_N,0,\ldots).
\)
Since \(x^{(N)}\to x\) and \(T_\varphi\) is bounded,
\(
\sum_{j=1}^{N}
\widehat{k}_{\mu_j}\otimes W_jx_j
=
T_\varphi x^{(N)}
\rightarrow
T_\varphi x
\)
in \(H^2(\mathbb D;\ell^2)\). Hence the series
\(
\sum_{j\geq1}
\widehat{k}_{\mu_j}\otimes W_jx_j
\)
converges for every
\(
x\in\bigoplus_{j\geq1}E_j
\),
and therefore defines the synthesis operator \(S_{M,W}\). Moreover,
\(
S_{M,W}
=
T_\varphi
=
\bigl(U_\varphi\otimes I_{\ell^2}\bigr)
S_{\Lambda,V}Q_\varphi.
\)
Thus \(S_{M,W}\) is bounded and bounded below. By the countable
analogue of
Proposition~\ref{prop:orbit-frame-kernel-synthesis-characterization},
\(
F\in\mathcal G_\omega\bigl(\varphi(A)\bigr).
\)
Applying the same argument to \(\varphi^{-1}\) and \(\varphi(A)\)
gives the reverse implication. Hence
\(
\mathcal G_\omega(A)
=
\mathcal G_\omega\bigl(\varphi(A)\bigr)
\)
for every such \(A\), and therefore
\(
\varphi\in\mathcal P_\omega
\).
\end{proof}

Together with Theorem~\ref{thm:holomorphic-rigidity}, the preceding
two propositions identify the holomorphic members of
\(\mathcal P_\omega\).

\begin{cor}
\label{cor:countable-holomorphic-rigidity}
Let
\(\varphi:\mathbb D\to\mathbb D\)
be holomorphic. Then
\(
\varphi\in\mathcal P_\omega
\) if and only if
\(\varphi\in\operatorname{Aut}(\mathbb D).
\)
\end{cor}

\begin{proof}
If \(\varphi\in\mathcal P_\omega\), then
Proposition~\ref{prop:countable-preservers-are-uniform-homeomorphisms}
gives \(\varphi\in\mathcal P\). Hence
Theorem~\ref{thm:holomorphic-rigidity} implies that
\(\varphi\in\operatorname{Aut}(\mathbb D)\).
The converse follows from
Proposition~\ref{prop:automorphisms-belong-to-countable-preservers}.
\end{proof}

The next theorem gives an explicit criterion for mappings whose
difference from the identity is \(o(1-|z|)\) uniformly near
\(\mathbb T\).

\begin{thm}
\label{thm:countable-boundary-perturbation}
Let \(\Psi:\mathbb D\to\mathbb D\) be a mapping. Suppose that
\begin{equation}
|\Psi(z)-z|=o(1-|z|)
\qquad\text{uniformly as }|z|\to1.
\label{eq:countable-boundary-perturbation}
\end{equation}
Then \(\Psi\in\mathcal P_\omega\) if and only if \(\Psi\) is a
homeomorphism of \(\mathbb D\) onto itself.
\end{thm}

\begin{proof}
The necessity follows from
Proposition~\ref{prop:countable-preservers-are-uniform-homeomorphisms}.
For the converse, we first establish an estimate for finite kernel
sums. Let \(E\) be a Hilbert space, let
\(Z=(z_j)_{j=1}^N\) and \(W=(w_j)_{j=1}^N\) be lists in
\(\mathbb D\), and let \(v_j\in E\). Define
\(\mathcal S_Zc=\sum_{j=1}^N c_jk_{z_j}\otimes v_j\) as an
operator from \(\mathbb C^N\) to \(H^2(\mathbb D;E)\), and
define \(\mathcal S_W\) analogously. If
\(\|\mathcal S_Z\|^2\le B\) and
\(|w_j-z_j|\le\eta(1-|z_j|)\), where \(0\le\eta\le\frac{1}{4}\), then
\begin{equation}
\|\mathcal S_W-\mathcal S_Z\|\le12\eta\sqrt B.
\label{eq:countable-kernel-perturbation-estimate}
\end{equation}
The constant is independent of \(N\), the lists, and \(E\).

To prove this estimate, put \(\delta_j=1-|z_j|\) and
\(D_j=D(z_j,\frac{\delta_j}{2})\), and let \(dA\) denote normalized
area measure. If \(u\in D_j\), then
\(\delta_j\le2(1-|u|)\) and
\(|1-\bar u z_j|\le3(1-|u|)\). Hence
\(|\widehat k_u(z_j)|^2\ge[9(1-|u|)]^{-1}\).
Testing the bound for \(\mathcal S_Z^*\) on
\(\widehat k_u\otimes e\) gives
\(
\sum_{j:\,u\in D_j}|\langle e,v_j\rangle|^2
\le9B(1-|u|)\|e\|^2\) for \(u\in\mathbb D, e\in E\).
For \(g\in H^2(\mathbb D;E)\), every point of the segment
joining \(z_j\) to \(w_j\) is the center of a disk of radius
\(\frac{\delta_j}{4}\) contained in \(D_j\). The mean value inequality
for \(|\langle g'(\cdot),v_j\rangle|^2\), followed by
Cauchy--Schwarz along the segment, therefore yields
\(
|\langle g(w_j)-g(z_j),v_j\rangle|^2
\le16\eta^2\int_{D_j}|\langle g'(u),v_j\rangle|^2\,dA(u).
\)
Summing and applying the preceding local estimate with \(e=g'(u)\),
we obtain
\[
\sum_{j=1}^N|\langle g(w_j)-g(z_j),v_j\rangle|^2
\le144\eta^2B\int_{\mathbb D}(1-|u|)\|g'(u)\|^2\,dA(u)
\le144\eta^2B\|g\|_{H^2(\mathbb D;E)}^2.
\]
For the last inequality, expand \(g\) in its Taylor series.
If \(g(z)=\sum_{n\ge0}g_nz^n\), the integral is
\(\sum_{n\ge1}\frac{n}{2n+1}\|g_n\|^2\), which is at most
\(\|g\|_{H^2(\mathbb D;E)}^2\). Taking adjoints proves
\eqref{eq:countable-kernel-perturbation-estimate}.

Now suppose that \(\Psi\) is a disk homeomorphism. Fix a
diagonal operator \(A\) and a family
\(F=(f_i)_{i\ge1}\in\mathcal G_\omega(A)\), with frame bounds
\(0<a\le b<\infty\). Since \(F\) is the subfamily corresponding
to \(n=0\), it is Bessel. Thus
\(Cx=(\langle x,f_i\rangle)_i\) defines a bounded operator
\(C:H\to\ell^2\). Choose an orthonormal eigenbasis \((e_j)\)
with \(Ae_j=\lambda_je_j\), retaining eigenvalues according to
multiplicity, and put \(v_j=Ce_j\). Set \(T=A^*\),
\(\psi(z)=\overline{\Psi(\bar z)}\),
\(S=\psi(T)=\Psi(A)^*\), and
\(P_R=\mathbf1_{\{|z|>R\}}(T)\).
Write \(\mathcal O_Tx=(CT^nx)_{n\ge0}\); the original frame
inequalities give
\(\sqrt a\|x\|\le\|\mathcal O_Tx\|\le\sqrt b\|x\|\).
For \(x\) in a finite span of eigenvectors, also write
\(\mathcal O_Sx=(CS^nx)_{n\ge0}\).
Under the unitary identification
\((y_n)_{n\ge0}\mapsto\sum_{n\ge0}y_nz^n\) of
\(\ell^2(\mathbb N;\ell^2)\) with \(H^2(\mathbb D;\ell^2)\),
the observation sequences for \(T\) and \(S\) on each finite
span of eigenvectors are the kernel sums at \(\lambda_j\) and
\(\Psi(\lambda_j)\), respectively, with the same coefficients
\(v_j\).

By \eqref{eq:countable-boundary-perturbation},
\(\eta_R:=\sup_{R<|z|<1}\frac{|\Psi(z)-z|}{1-|z|}\to0\) as
\(R\uparrow1\). For \(R\) sufficiently close to \(1\), apply
\eqref{eq:countable-kernel-perturbation-estimate} with \(B=b\)
to every finite list of eigenvalues satisfying \(|\lambda_j|>R\).
This gives
\(\|\mathcal O_Sx-\mathcal O_Tx\|
\le12\eta_R\sqrt b\|x\|\)
on the finite spans of the corresponding eigenvectors.
By density, \(\mathcal O_SP_R\) is bounded; its coordinates
remain \(CS^nP_R\), since each coordinate operator is continuous.
Moreover,
\[
(\sqrt a-12\eta_R\sqrt b)\|x\|
\le\|\mathcal O_Sx\|
\le\sqrt b(1+12\eta_R)\|x\|,
\qquad x\in P_RH.
\]
Choose \(R\) so that \(12\eta_R\sqrt b<\sqrt a\).
Lemma~\ref{lem:countable-boundary-localization}, proved in
Appendix~\ref{app:countable-observation-localization}, then gives
both bounds on all of \(H\). Thus
\(F\in\mathcal G_\omega(\Psi(A))\).

For the reverse implication,
\eqref{eq:countable-boundary-perturbation} gives
\(\frac{1-|\Psi(z)|}{1-|z|}\to1\) uniformly as \(|z|\to1\).
Since \(\Psi\) is a disk homeomorphism,
\(|\Psi^{-1}(w)|\to1\) as \(|w|\to1\). Writing
\(z=\Psi^{-1}(w)\), we obtain
\[
\frac{|\Psi^{-1}(w)-w|}{1-|w|}
=
\frac{|\Psi(z)-z|}{1-|z|}
\frac{1-|z|}{1-|\Psi(z)|}
\longrightarrow0
\]
uniformly as \(|w|\to1\). Applying the forward implication to
\(\Psi^{-1}\) and \(\Psi(A)\) proves the reverse frame implication
for the same family \(F\). Therefore \(\Psi\in\mathcal P_\omega\).
\end{proof}

By its defining preservation property and
Proposition~\ref{prop:countable-preservers-are-uniform-homeomorphisms},
\(\mathcal P_\omega\) is closed under composition and taking
inverses. The preceding theorem therefore extends to boundary
perturbations of an arbitrary disk automorphism.

\begin{cor}
\label{cor:countable-automorphism-perturbation}
Let \(\varphi\in\operatorname{Aut}(\mathbb D)\) and let
\(\Phi:\mathbb D\to\mathbb D\) be a mapping. Suppose that
\(|\Phi(z)-\varphi(z)|=o(1-|z|)\) uniformly as \(|z|\to1\).
Then \(\Phi\in\mathcal P_\omega\) if and only if \(\Phi\) is a
homeomorphism of \(\mathbb D\) onto itself.
\end{cor}

\begin{proof}
Set \(\Psi=\varphi^{-1}\circ\Phi\). Since \(\varphi^{-1}\)
is Lipschitz on \(\overline{\mathbb D}\), with constant \(L\),
we have
\( |\Psi(z)-z|
\le L|\Phi(z)-\varphi(z)|=o(1-|z|)\)
uniformly as \(|z|\to1\). By
Proposition~\ref{prop:automorphisms-belong-to-countable-preservers}
and closure under composition,
\(\Phi\in\mathcal P_\omega\) if and only if
\(\Psi\in\mathcal P_\omega\). Moreover, \(\Phi\) is a disk
homeomorphism if and only if \(\Psi\) is. The conclusion follows
from Theorem~\ref{thm:countable-boundary-perturbation}.
\end{proof}

The same theorem gives a countable version of
Corollary~\ref{cor:compact-core-surgery-invariance} within the class
of disk homeomorphisms.

\begin{cor}
\label{cor:countable-compact-modification}
Let \(\Phi_1,\Phi_2:\mathbb D\to\mathbb D\) be disk
homeomorphisms. Suppose that there exists \(0\le r_0<1\) such that
\(\Phi_1(z)=\Phi_2(z)\) whenever \(|z|>r_0\). Then
\(
\Phi_1\in\mathcal P_\omega\) if and only if
\(
\Phi_2\in\mathcal P_\omega.
\)
\end{cor}

\begin{proof}
The mapping \(\chi=\Phi_2\circ\Phi_1^{-1}\) is a disk
homeomorphism and equals the identity outside the compact set
\(\Phi_1(\overline{D(0,r_0)})\). Hence it equals the identity on
some boundary annulus and satisfies
\eqref{eq:countable-boundary-perturbation}.
Theorem~\ref{thm:countable-boundary-perturbation} gives
\(\chi\in\mathcal P_\omega\). Since \(\Phi_2=\chi\circ\Phi_1\),
closure under composition and taking inverses proves the equivalence.
\end{proof}

We now apply Theorem~\ref{thm:countable-boundary-perturbation} to
an explicit family of radial homeomorphisms. These examples show that
countable orbit preservation does not force holomorphicity on any
boundary annulus.

\begin{example}\label{ex:countable-nonholomorphic-preserver}
For \(0<c<1\) and \(\beta>1\), define
\(K_{c,\beta}(z)=[1-c(1-|z|)^\beta]z\) for \(z\in\mathbb D\).
Then \(K_{c,\beta}\in\mathcal P_\omega\), but
\(K_{c,\beta}\) is not holomorphic on any boundary annulus.
In particular, \(K_\varepsilon:=K_{\varepsilon,2}\) has these
properties for \(0<\varepsilon<\frac{1}{2}\). Consequently,
\(\operatorname{Aut}(\mathbb D)\subsetneq\mathcal P_\omega\).
\end{example}

\begin{proof}
The radial function \(p(r)=r[1-c(1-r)^\beta]\) satisfies
\(p(0)=0\), \(p(r)\to1\) as \(r\uparrow1\), and
\(p'(r)=1-c(1-r)^\beta+c\beta r(1-r)^{\beta-1}\ge1-c>0\)
for \(0<r<1\). Hence \(K_{c,\beta}\) is a homeomorphism of
\(\mathbb D\) onto itself. Moreover,
\(
\frac{|K_{c,\beta}(z)-z|}{1-|z|}
=c|z|(1-|z|)^{\beta-1}
\le c(1-|z|)^{\beta-1}\rightarrow0
\)
uniformly as \(|z|\to1\). Theorem~\ref{thm:countable-boundary-perturbation}
therefore gives \(K_{c,\beta}\in\mathcal P_\omega\).
For \(z\ne0\), a direct calculation yields
\(
\frac{\partial K_{c,\beta}}{\partial\bar z}(z)
=\frac{c\beta}{2}(1-|z|)^{\beta-1}\frac{z^2}{|z|}\ne0.
\)
Thus \(K_{c,\beta}\) is not holomorphic on any boundary annulus,
and \(\operatorname{Aut}(\mathbb D)\subsetneq\mathcal P_\omega\).
\end{proof}

We next give a necessary and sufficient criterion for an arbitrary
disk mapping. The criterion uses finite matrices, but requires bounds
independent of their dimension, spectral points, and spectral
multiplicities.

For a finite indexed list
\(\Lambda=(\lambda_1,\ldots,\lambda_N)\in\mathbb D^N\), with
repetitions allowed, and a positive semidefinite matrix
\(Q=(q_{jk})\), set
\[
\mathsf S_\Lambda(Q):=
\left[\frac{q_{jk}}{1-\lambda_j\overline{\lambda_k}}\right]_{j,k=1}^N,
\qquad D_\Lambda:=\operatorname{diag}(\lambda_1,\ldots,\lambda_N).
\]
The matrix \(\mathsf S_\Lambda(Q)\) is the orbit frame operator
obtained by iterating the columns of \(Q^{\frac{1}{2}}\) under
\(D_\Lambda\). Indeed, a calculation using a geometric series gives
\begin{equation}\label{eq:countable-finite-realization}
\mathsf S_\Lambda(Q)
=\sum_{n\ge0}D_\Lambda^nQ(D_\Lambda^*)^n,
\qquad 0\le Q\le\mathsf S_\Lambda(Q).
\end{equation}
The series converges in operator norm because the list is finite.
All matrix inequalities below are in the positive semidefinite order.

\begin{thm}
\label{thm:countable-uniform-matrix-characterization}
Let \(\Phi:\mathbb D\to\mathbb D\) be an arbitrary mapping.
Then \(\Phi\in\mathcal P_\omega\) if and only if, for every
\(\kappa\ge1\), there exist constants
\(0<a_\kappa\le b_\kappa<\infty\) such that for every
\(N\ge1\), every \(\Lambda\in\mathbb D^N\), and every \(Q\ge0\),
\begin{align}
I_N\le\mathsf S_\Lambda(Q)\le\kappa I_N
&\ \Longrightarrow\quad
a_\kappa I_N\le\mathsf S_{\Phi(\Lambda)}(Q)\le b_\kappa I_N,
\label{eq:countable-matrix-forward}\\
I_N\le\mathsf S_{\Phi(\Lambda)}(Q)\le\kappa I_N
&\ \Longrightarrow\quad
a_\kappa I_N\le\mathsf S_\Lambda(Q)\le b_\kappa I_N.
\label{eq:countable-matrix-reverse}
\end{align}
The constants depend only on \(\Phi\) and \(\kappa\), not on
\(N\), \(\Lambda\), or \(Q\).
\end{thm}

\begin{proof}
Suppose first that \(\Phi\in\mathcal P_\omega\), and fix
\(\kappa\ge1\). If no uniform upper bound exists in
\eqref{eq:countable-matrix-forward}, choose finite data
\((\Lambda_m,Q_m)\) with
\(
I\le\mathsf S_{\Lambda_m}(Q_m)\le\kappa I, \|\mathsf S_{\Phi(\Lambda_m)}(Q_m)\|\rightarrow\infty.
\)
If no uniform positive lower bound exists, choose such data with
\(\lambda_{\min}(\mathsf S_{\Phi(\Lambda_m)}(Q_m))\to0\)
instead. Realize each \(Q_m\) by the columns of \(Q_m^{\frac{1}{2}}\),
and place these finite families in mutually orthogonal summands.
On their Hilbert direct sum, set
\(A_0=\bigoplus_m D_{\Lambda_m}\).
By \eqref{eq:countable-finite-realization}, the resulting countable
family generates an orbit frame under \(A_0\) with bounds \(1,\kappa\),
whereas its transformed orbit system has no finite upper bound or no
positive lower bound. The generating family is Bessel, since
\(0\le Q_m\le\kappa I\).

To ensure infinitely many distinct eigenvalues, append an orthogonal
diagonal summand with distinct eigenvalues \(\tau_l\in\mathbb D\)
and generators \(\sqrt{1-|\tau_l|^2}\,e_l\). Their orbits under the original operator form a Parseval frame,
so the enlarged orbit system still has bounds \(1,\kappa\). The failure in the transformed finite blocks
persists, contradicting \(\Phi\in\mathcal P_\omega\).

For \eqref{eq:countable-matrix-reverse}, repeat the argument with
\(I\le\mathsf S_{\Phi(\Lambda_m)}(Q_m)\le\kappa I\), and append
generators \(\sqrt{1-|\Phi(\tau_l)|^2}\,e_l\) instead. Now the
full transformed orbit system is a frame, while the original blocks fail
one of the required bounds. This also contradicts preservation.
This argument uses no inverse map and does not require
\(\Phi(A)\) to have infinitely many distinct eigenvalues.
Taking the smaller of the two positive lower constants and the
larger of the two upper constants gives common
\(a_\kappa,b_\kappa\).

Conversely, suppose the two conditions on finite matrices hold. Let
\(Ae_j=\lambda_je_j\) in an orthonormal eigenbasis, retaining
repeated eigenvalues according to multiplicity, and suppose
\((A^nf_i)_{n,i}\) is a frame with bounds \(a,b\).
The generating family \(F\) is Bessel, with bounded positive operator
\(Q_F=\sum_i f_if_i^*\), where the sum converges strongly.
For a finite coordinate set \(J\), let \(Q_J\) be the compression
of \(Q_F\), and let \(\Lambda_J=(\lambda_j)_{j\in J}\).
The coordinate projection commutes with \(A\), so
\eqref{eq:countable-finite-realization} gives
\(
aI_J\le\mathsf S_{\Lambda_J}(Q_J)\le bI_J.
\)
Apply \eqref{eq:countable-matrix-forward} to \(\frac{Q_J}{a}\), with
\(\kappa=\frac{b}{a}\). For every finite coordinate set,
\(
a\,a_{\frac{b}{a}}I_J\le\mathsf S_{\Phi(\Lambda_J)}(Q_J)
\le a\,b_{\frac{b}{a}}I_J.
\)
Equivalently, for every finitely supported \(x\),
\[
a\,a_{b/a}\|x\|^2
\le\sum_{n\ge0}\sum_{i\ge1}
|\langle x,\Phi(A)^nf_i\rangle|^2
\le a\,b_{b/a}\|x\|^2.
\]
The upper estimate extends the analysis operator boundedly from
this dense subspace to \(H\); each coordinate of the extension is
the displayed inner product. The lower estimate extends by
continuity. Thus the transformed orbit system is a frame.
Starting with a transformed orbit frame and using
\eqref{eq:countable-matrix-reverse} proves the converse.
\end{proof}

The uniformity in this theorem is essential to the passage from finite data to frames formed by countably many operator orbits. The argument using direct sums
forces common bounds for all finite configurations, while finite
compression recovers the full orbit inequalities from those bounds.
Corollary~\ref{cor:countable-automorphism-perturbation} gives an
explicit criterion under a prescribed boundary estimate. The present
theorem applies to arbitrary disk mappings, but does not give an
explicit geometric classification of the full class
\(\mathcal P_\omega\).

The preceding criterion also shows that membership in \(\mathcal P\),
even for a pseudohyperbolic isometry, does not imply preservation of
frames generated by countably many orbits.

\begin{example}
\label{ex:countable-conjugation-not-preserver}
The mapping \(J:\mathbb D\to\mathbb D\), \(J(z)=\bar z\), belongs
to \(\mathcal P\setminus\mathcal P_\omega\), although it is a
pseudohyperbolic isometry.
\end{example}

\begin{proof}
Since \(|J(z)|=|z|\) and
\(\rho(J(z),J(w))=\rho(z,w)\), complex conjugation preserves
Carleson sequences in both directions and leaves the boundary defect
unchanged. Theorem~\ref{thm:universal-orbit-frame-preservation}
therefore gives \(J\in\mathcal P\).

Suppose that \(J\in\mathcal P_\omega\), and let \(b_2\) be the
upper constant in
Theorem~\ref{thm:countable-uniform-matrix-characterization} for
\(\kappa=2\). For a finite list \(\Lambda=(\lambda_j)_{j=1}^N\),
write \(\bar\Lambda=(\bar\lambda_j)_{j=1}^N\) and set
\(\Delta_\Lambda=\operatorname{diag}(1-|\lambda_j|^2)_{j=1}^N\).
Then \(\mathsf S_\Lambda(\Delta_\Lambda)
=\mathsf S_{\bar\Lambda}(\Delta_\Lambda)=I_N\).
For \(Q\ge0\), \(Q\ne0\), put
\(B=\|\mathsf S_\Lambda(Q)\|>0\). Since
\(I_N\le\mathsf S_\Lambda(\frac{Q}{B}+\Delta_\Lambda)\le2I_N\),
positivity and the uniform upper estimate give
\(
\|\mathsf S_{\bar\Lambda}(Q)\|
\le b_2\|\mathsf S_\Lambda(Q)\|.
\)
This also holds for \(Q=0\).

For a Hilbert space \(E_0\) and vectors \(v_j\in E_0\), define
\(T_\Lambda g=(\langle g(\lambda_j),v_j\rangle)_{j=1}^N\)
on \(H^2(\mathbb D;E_0)\). With
\(Q=(\langle v_k,v_j\rangle)_{j,k=1}^N\ge0\), the reproducing
property gives \(T_\Lambda T_\Lambda^*=\mathsf S_\Lambda(Q)\)
and \(T_{\bar\Lambda}T_{\bar\Lambda}^*
=\mathsf S_{\bar\Lambda}(Q)\). Hence
\(\|\mathsf S_{\bar\Lambda}(Q)\|
\le b_2\|\mathsf S_\Lambda(Q)\|\) implies
\(\|T_{\bar\Lambda}\|^2\le b_2\|T_\Lambda\|^2\).
Decomposing positive matrix weights into sums of rank one positive
operators, and retaining repeated sample points, yields the same
comparison for finite sums of evaluations with positive matrix weights.

By \cite[Corollary~1.11]{Rydhe2017}, there are a separable Hilbert
space \(E\) and a bounded analytic function
\(\vartheta:\mathbb D\to\mathcal B(E)\) such that, with
\(B_0(z)=\vartheta(z)+z\vartheta'(z)\) and
\(W(z)=(1-|z|^2)B_0(z)^*B_0(z)\), one has
\(
\int_{\mathbb D}\langle W(z)g(z),g(z)\rangle\,dA(z)
\le C_0\|g\|_{H^2(\mathbb D;E)}^2,
\)
whereas
\begin{equation}
\sup_{\|g\|_{H^2(\mathbb D;E)}=1}
\int_{\mathbb D}\langle W(z)g(\bar z),g(\bar z)\rangle\,dA(z)
=\infty.
\label{eq:countable-conjugation-reflected-embedding}
\end{equation}
Here \(dA\) is normalized area measure. The first inequality extends
from analytic polynomials to all of \(H^2(\mathbb D;E)\) by
approximation and Fatou's lemma. The density \(W\) is positive and
continuous in operator norm on every compact subdisk.

Fix \(r<1\) and a finite dimensional subspace \(E_0\subset E\).
Let \(P_0:E\to E_0\) be the orthogonal projection, put
\(W_0(z)=P_0W(z)|_{E_0}\), and denote evaluation at \(z\) on
\(H^2(\mathbb D;E_0)\) by \(\mathcal E_z\). Consider the positive
operators
\(
G_r^+=\int_{|z|\le r}\mathcal E_z^*W_0(z)\mathcal E_z\,dA(z),
G_r^-=\int_{|z|\le r}\mathcal E_{\bar z}^*W_0(z)
\mathcal E_{\bar z}\,dA(z).
\)
Both integrands are continuous in operator norm, since
\(z\mapsto k_z\) is continuous in \(H^2(\mathbb D)\).
Their Riemann sums therefore converge in operator norm. Each pair
of sums uses the same positive matrix weights, so the preceding
comparison for finite evaluations passes to the limit and gives
\(\|G_r^-\|\le b_2\|G_r^+\|\le b_2C_0\).
The constant is independent of \(r\) and \(E_0\).

For any analytic polynomial \(g\) with coefficients in \(E\), take
\(E_0\) to contain those coefficients. Letting \(r\uparrow1\),
monotone convergence gives
\(\int_{\mathbb D}\langle W(z)g(\bar z),g(\bar z)\rangle\,dA(z)
\le b_2C_0\|g\|_{H^2(\mathbb D;E)}^2\).
Taylor polynomial approximation and Fatou's lemma extend this bound
to every \(g\in H^2(\mathbb D;E)\), contradicting
\eqref{eq:countable-conjugation-reflected-embedding}.
Thus \(J\notin\mathcal P_\omega\).
\end{proof}

Combining Examples~\ref{ex:countable-nonholomorphic-preserver}
and~\ref{ex:countable-conjugation-not-preserver} with
Proposition~\ref{prop:countable-preservers-are-uniform-homeomorphisms},
we obtain
\(
\operatorname{Aut}(\mathbb D)\subsetneq\mathcal P_\omega
\subsetneq\mathcal P\cap\operatorname{Homeo}(\mathbb D),
\)
where \(\operatorname{Homeo}(\mathbb D)\) denotes the homeomorphisms
of \(\mathbb D\) onto itself. Thus the distinction between finite
and countable orbit preservation persists even among disk
homeomorphisms.

\appendix
\section{Auxiliary Geometric Estimates and Compactness}
\label{app:geometric-estimates-compactness}
We give the complete proofs of
Lemmas~\ref{lem:pseudohyperbolic-polar-estimates}
and~\ref{lem:asymptotic-arzela-ascoli}, using the notation of their
statements in Sections~3 and~4.

\begin{proof}[Proof of Lemma~\ref{lem:pseudohyperbolic-polar-estimates}]
After discarding finitely many terms, we may assume that \(z_n,w_n\neq0\) for every \(n\).
First we prove \textup{(i)}. Write
\(
z_n=r_n\zeta_n,
\
w_n=s_n\eta_n,
\
\alpha_n:=d_{\mathbb T}(\zeta_n,\eta_n),
\)
and set
\(
q_n:=\rho(z_n,w_n)^2.
\)
For convenience, put
\(
a_n:=(r_n-s_n)^2,
\
b_n:=4r_ns_n\sin^2(\pi\alpha_n),
\
c_n:=(1-r_ns_n)^2.
\)
By \eqref{eq:pseudohyperbolic-distance-polar-form},
\(
q_n=\frac{a_n+b_n}{c_n+b_n}.
\)
Since \(q_n\to0\), we may assume that \(q_n<1\). From
\(
a_n+b_n=q_n(c_n+b_n)
\)
we obtain
\(
(1-q_n)b_n\leq q_nc_n,
\)
and therefore
\begin{equation}
\label{eq:angular-term-relative-to-radial-term}
\frac{b_n}{c_n}
\leq
\frac{q_n}{1-q_n}
\rightarrow0.
\end{equation}
Moreover,
\(
a_n
\leq
q_n(c_n+b_n),
\)
so that
\(
\frac{a_n}{c_n}
\leq
q_n
\left(
1+\frac{b_n}{c_n}
\right)
\rightarrow0.
\)
Consequently,
\(
\frac{|r_n-s_n|}{1-r_ns_n}
\rightarrow0.
\)
Set
\(
x_n:=1-r_n=\delta(z_n),
\
y_n:=1-s_n=\delta(w_n).
\)
Then
\(
|r_n-s_n|=|x_n-y_n|
\)
and
\(
1-r_ns_n
=
x_n+y_n-x_ny_n.
\)
Thus we have
\(
\varepsilon_n
:=
\frac{|x_n-y_n|}
{x_n+y_n-x_ny_n}
\rightarrow0.
\)
Since
\(
x_n+y_n-x_ny_n
\leq
x_n+y_n
\leq
2y_n+|x_n-y_n|,
\)
we have
\(
|x_n-y_n|
\leq
\varepsilon_n
\left(
2y_n+|x_n-y_n|
\right).
\)
For all sufficiently large \(n\), \(\varepsilon_n<1\), and hence
\(
\frac{|x_n-y_n|}{y_n}
\leq
\frac{2\varepsilon_n}{1-\varepsilon_n}
\rightarrow0.
\)
It follows that
\(
\frac{x_n}{y_n}\rightarrow1,
\)
which proves
\eqref{eq:asymptotic-boundary-defect-comparability}.

It remains to prove
\eqref{eq:asymptotic-angular-control}. By
\eqref{eq:angular-term-relative-to-radial-term},
\(
\frac{
2\sqrt{r_ns_n}\sin(\pi\alpha_n)
}{
1-r_ns_n
}
\rightarrow0.
\)
Since \(r_n,s_n\to1\), it follows that
\(
\frac{\sin(\pi\alpha_n)}{1-r_ns_n}
\rightarrow0.
\)
Because \(0\leq\alpha_n\leq\frac{1}{2}\),
\(\sin(\pi\alpha_n)\geq 2\alpha_n,
\)
thus
\(
\frac{\alpha_n}{1-r_ns_n}
\leq
\frac{\sin(\pi\alpha_n)}{2(1-r_ns_n)}
\rightarrow0.
\)
Furthermore,
\(
\frac{1-r_ns_n}{x_n}
=
\frac{x_n+y_n-x_ny_n}{x_n}
=
1+\frac{y_n}{x_n}-y_n
\rightarrow2,
\)
because \(\frac{x_n}{y_n}\to1\) and \(y_n\to0\). Therefore,
\(
\frac{\alpha_n}{x_n}
=
\frac{\alpha_n}{1-r_ns_n}
\frac{1-r_ns_n}{x_n}
\rightarrow0.
\)
Since
\(
\alpha_n
=
d_{\mathbb T}
\left(
\frac{z_n}{|z_n|},
\frac{w_n}{|w_n|}
\right),
\
x_n=\delta(z_n),
\)
this proves \eqref{eq:asymptotic-angular-control}.

We finally prove \textup{(ii)}. Let
\(
r:=|z|,
\
s:=|w|,
\
\alpha
:=
d_{\mathbb T}
\left(
\frac{z}{|z|},
\frac{w}{|w|}
\right).
\)
Set
\(
a:=(r-s)^2,
\
b:=4rs\sin^2(\pi\alpha),
\
c:=(1-rs)^2.
\)
Since \(\rho(z,w)\leq M\), formula
\eqref{eq:pseudohyperbolic-distance-polar-form} yields
\(
a+b
\leq
M^2(c+b).
\)
Equivalently,
\begin{equation}
\label{eq:basic-bounded-pseudohyperbolic-estimate}
a+(1-M^2)b
\leq
M^2c.
\end{equation}

In particular,
\(
|r-s|
\leq
M(1-rs).
\)
Let
\(
x:=1-r=\delta(z),
\
y:=1-s=\delta(w).
\)
Then
\(
|x-y|
\leq
M(x+y-xy)
\leq
M(x+y).
\)
The inequalities
\(
x-y\leq M(x+y)
\)
and
\(
y-x\leq M(x+y)
\)
give
\(
(1-M)x\leq(1+M)y
\)
and
\(
(1-M)y\leq(1+M)x.
\)
Consequently,
\begin{equation}
\label{eq:explicit-radial-comparability}
\frac{1-M}{1+M}\,x
\leq
y
\leq
\frac{1+M}{1-M}\,x.
\end{equation}
We next estimate the angular distance. From
\eqref{eq:basic-bounded-pseudohyperbolic-estimate},
\(
4rs(1-M^2)\sin^2(\pi\alpha)
\leq
M^2(1-rs)^2.
\)
Hence
\(
2\sqrt{rs}\sin(\pi\alpha)
\leq
\frac{M}{\sqrt{1-M^2}}(1-rs).
\)
Since \(r,s\geq\frac{1}{2}\), we have \(2\sqrt{rs}\geq1\), and therefore
\(
\sin(\pi\alpha)
\leq
\frac{M}{\sqrt{1-M^2}}(1-rs).
\)
Using again
\(
\sin(\pi\alpha)\geq2\alpha,
\)
we obtain
\(
\alpha
\leq
\frac{M}{2\sqrt{1-M^2}}(1-rs).
\)
By \eqref{eq:explicit-radial-comparability},
\(
y
\leq
\frac{1+M}{1-M}x.
\)
Therefore,
\(
1-rs
=
x+ry
\leq
x+y
\leq
\left(
1+\frac{1+M}{1-M}
\right)x
=
\frac{2}{1-M}x.
\)
It follows that
\begin{equation}
\label{eq:explicit-angular-bound}
\alpha
\leq
\frac{M}
{(1-M)\sqrt{1-M^2}}
\,x.
\end{equation}

Thus we may take
\(
A_M
:=
\max
\left\{
\frac{1+M}{1-M},
\frac{M}
{(1-M)\sqrt{1-M^2}}
\right\}.
\)
Then \eqref{eq:explicit-radial-comparability} gives
\eqref{eq:bounded-pseudohyperbolic-radial-control}, while
\eqref{eq:explicit-angular-bound} gives
\eqref{eq:bounded-pseudohyperbolic-angular-control}.
\end{proof}

\begin{proof}[Proof of Lemma~\ref{lem:asymptotic-arzela-ascoli}]
For each \(m,k\geq1\), choose a finite net
\(
D_{m,k}\subset K_m
\) of mesh \(2^{-k}\),
and set
\(
D:=\bigcup_{m,k\geq1}D_{m,k}.
\)
Then \(D\) is countable, and \(D\cap K_m\) is dense in \(K_m\) for
every \(m\).

For each \(p\in D\), choose \(m\) such that \(p\in K_m\). By
condition~\textup{(i)}, the sequence
\(
\{F_n(p)\}_{n\geq m}
\)
is contained in a relatively compact subset of \(Y\). A diagonal
argument therefore yields a subsequence
\(
\{F_{n_j}\}_{j\geq1}
\)
such that
\(
F_{n_j}(p)
\)
converges for every \(p\in D\).

We claim that the sequence \(\{F_{n_j}\}\) is uniformly Cauchy on each \(K_m\).
Fix \(m\geq1\) and \(\varepsilon>0\). By condition~\textup{(ii)},
there exist \(\delta>0\) and \(N\geq m\) such that, whenever
\(n\geq N\),
\(
d_X(p,q)<\delta,\quad p,q\in K_m
\)
imply
\(
d_Y\bigl(F_n(p),F_n(q)\bigr)<\frac{\varepsilon}{3}.
\)
Choose \(k\) so large that \(2^{-k}<\delta\), and write
\(
D_{m,k}=\{p_1,\ldots,p_r\}.
\)
Since \(F_{n_j}(p_\ell)\) converges for every \(\ell\), there exists
\(J\) such that, whenever \(j,s\geq J\),
\(
n_j,n_s\geq N
\)
and
\(
d_Y\bigl(F_{n_j}(p_\ell),F_{n_s}(p_\ell)\bigr)
<
\frac{\varepsilon}{3}\) for \(1\leq\ell\leq r\).
For any \(p\in K_m\), choose \(p_\ell\in D_{m,k}\) satisfying
\(
d_X(p,p_\ell)<\delta
\).
Then, for \(j,s\geq J\),
\[
d_Y\bigl(F_{n_j}(p),F_{n_s}(p)\bigr)
\leq
d_Y\bigl(F_{n_j}(p),F_{n_j}(p_\ell)\bigr)
+
d_Y\bigl(F_{n_j}(p_\ell),F_{n_s}(p_\ell)\bigr)
+
d_Y\bigl(F_{n_s}(p_\ell),F_{n_s}(p)\bigr)
<\varepsilon.
\]
Thus \(\{F_{n_j}\}\) is uniformly Cauchy on \(K_m\).

By condition~\textup{(i)}, the set
\(
C_m:=
\overline{\bigcup_{n\geq m}F_n(K_m)}
\)
is compact, and hence complete. Therefore there exists a mapping
\(
F_m:K_m\rightarrow C_m\subset Y
\)
such that
\(
F_{n_j}\rightarrow F_m
\)
uniformly on \(K_m\). Since \(K_m\subset K_{m+1}\), uniqueness of
limits gives
\(
F_{m+1}|_{K_m}=F_m.
\)
Hence the mappings \(F_m\) define a mapping
\(
F:X\rightarrow Y.
\)

We next prove that \(F_m\) is continuous. Fix \(\varepsilon>0\). By
condition~\textup{(ii)}, there exist \(\delta>0\) and \(N\geq m\) such
that, whenever
\(
n\geq N
\)
and
\(
p,q\in K_m
\)
satisfy
\(
d_X(p,q)<\delta,
\)
one has
\(
d_Y\bigl(F_n(p),F_n(q)\bigr)<\frac{\varepsilon}{2}.
\)
Letting \(j\to\infty\) along indices satisfying \(n_j\geq N\), we obtain
\(
d_Y\bigl(F_m(p),F_m(q)\bigr)
\leq
\frac{\varepsilon}{2}
<
\varepsilon.
\)
Thus \(F_m\) is continuous on \(K_m\).

To see that \(F\) is continuous on \(X\), let \(x_j\to x\) in \(X\).
The set
\(
\{x\}\cup\{x_j:j\geq1\}
\)
is compact, and hence is contained in some \(K_m\). Since
\(F|_{K_m}=F_m\) is continuous, it follows that
\(
F(x_j)\to F(x)
\).

Finally, if \(K\subset X\) is compact, then \(K\subset K_m\) for some
\(m\). Since \(F_{n_j}\to F_m=F|_{K_m}\) uniformly on \(K_m\), the
convergence is uniform on \(K\). Therefore,
\(
F_{n_j}\rightarrow F
\)
uniformly on every compact subset of \(X\).
\end{proof}

\section{Localization and Degeneration of Finite Kernel Clusters}
\label{app:finite-cluster-proofs}
This appendix supplies the complete proofs of
Lemmas~\ref{lem:finite-cluster-localization}
and~\ref{lem:bad-cluster-subsequence-transfer}. We retain the
Hartmann partition, coefficient operators, and
data in finite dimensions introduced in Section~\ref{sec:finite-orbits}.

\begin{proof}[Proof of Lemma~\ref{lem:finite-cluster-localization}]
Suppose first that \(S_{\Lambda,V}\) is bounded below. Then there
exists \(a>0\) such that
\(
\|S_{\Lambda,V}x\|
\geq
a\|x\|,
\
x\in\bigoplus_{j\geq1}E_j.
\)
Restricting this inequality to \(\mathcal E_q\), we obtain
\(
\|S_qx\|
\geq
a\|x\|,
\
x\in\mathcal E_q,
\)
for every \(q\in\mathbb N^{+}\). Hence
\(
s_{\min}(S_q)\geq a,
\
q\in\mathbb N^{+},
\)
and therefore
\(
\inf_{q\in\mathbb N^{+}}s_{\min}(S_q)
\geq
a>0.
\)

Conversely, suppose that
\(
c
:=
\inf_{q\in\mathbb N^{+}}s_{\min}(S_q)
>0.
\)
Let
\(
\{D_q\}_{q\in\mathbb N^{+}}
\subset H^\infty(\mathbb D)
\)
and \(M_D<\infty\) be the functions and the constant provided by
Lemma~\ref{lem:hartmann-cluster-decomposition}\textup{(iv)} for the partition
\(
\Lambda
=
\mathop{\dot\bigcup}_{q\in\mathbb N^{+}}\sigma_q.
\)
Thus,
\[
D_q(\lambda)
=
\begin{cases}
1,
& \lambda\in\sigma_q,\\
0,
& \lambda\in\Lambda\setminus\sigma_q,
\end{cases}
\]
and
\(
\sum_{q\in\mathbb N^{+}}|D_q(z)|
\leq
M_D,
\
z\in\mathbb D.
\)

For each \(q\in\mathbb N^{+}\), let
\(
M_{D_q}:
H^2(\mathbb D;\mathbb C^m)
\rightarrow
H^2(\mathbb D;\mathbb C^m)
\)
be the multiplication operator induced by \(D_q\), that is,
\(
(M_{D_q}g)(z)
:=
D_q(z)g(z).
\)

Set
\(
\mathscr X:=H^2\bigl(\mathbb D;
\ell^2(\mathbb N^{+};\mathbb C^m)\bigr)
\cong\bigoplus_{q\in\mathbb N^{+}}H^2(\mathbb D;\mathbb C^m).
\)
Under this identification, let \(\mathscr X_0\) be the dense
subspace of sequences \(g=(g_q)_{q\in\mathbb N^{+}}\) with only
finitely many nonzero components.

Define
\(
\mathcal M_Dg
:=
\sum_{q\in\mathbb N^{+}}M_{D_q}g_q,
\
g\in\mathscr X_0.
\)
Since \(g\) is finitely supported, the sum is finite and belongs to
\(H^2(\mathbb D;\mathbb C^m)\).

Let \(m_{\mathbb T}\) denote normalized arc length measure on
\(\mathbb T\). For \(g\in\mathscr X_0\), the Cauchy--Schwarz
inequality and Lemma~\ref{lem:hartmann-cluster-decomposition}\textup{(iv)} give
\[
\begin{aligned}
\|\mathcal M_Dg\|_{H^2(\mathbb D;\mathbb C^m)}^2
&=
\sup_{0<r<1}
\int_{\mathbb T}
\left\|
\sum_{q\in\mathbb N^{+}}
D_q(r\zeta)g_q(r\zeta)
\right\|_{\mathbb C^m}^2
\,dm_{\mathbb T}(\zeta)\\
&\leq
\sup_{0<r<1}
\int_{\mathbb T}
\left(
\sum_{q\in\mathbb N^{+}}|D_q(r\zeta)|^2
\right)
\left(
\sum_{q\in\mathbb N^{+}}
\|g_q(r\zeta)\|_{\mathbb C^m}^2
\right)
\,dm_{\mathbb T}(\zeta)\\
&\leq
M_D^2
\sup_{0<r<1}
\int_{\mathbb T}
\sum_{q\in\mathbb N^{+}}
\|g_q(r\zeta)\|_{\mathbb C^m}^2
\,dm_{\mathbb T}(\zeta)\\
&\leq
M_D^2
\sum_{q\in\mathbb N^{+}}
\|g_q\|_{H^2(\mathbb D;\mathbb C^m)}^2\\
&=
M_D^2\|g\|_{\mathscr X}^2.
\end{aligned}
\]
Therefore, \(\mathcal M_D\) extends uniquely to a bounded operator
\(
\mathcal M_D:
\mathscr X
\rightarrow
H^2(\mathbb D;\mathbb C^m)
\)
satisfying
\(
\|\mathcal M_D\|
\leq
M_D.
\)

For \(g=(g_q)_{q\in\mathbb N^{+}}\in\mathscr X_0\) and
\(y\in H^2(\mathbb D;\mathbb C^m)\), we have
\(
\langle \mathcal M_Dg,y\rangle
=
\sum_{q\in\mathbb N^{+}}
\langle M_{D_q}g_q,y\rangle
=
\sum_{q\in\mathbb N^{+}}
\langle g_q,M_{D_q}^*y\rangle.
\)
By density of \(\mathscr X_0\), it follows that
\(
\mathcal M_D^*y
=
\bigl(M_{D_q}^*y\bigr)_{q\in\mathbb N^{+}},
\
y\in H^2(\mathbb D;\mathbb C^m).
\)

Let
\(
x
=
\bigoplus_{q\in\mathbb N^{+}}x_q,
\
x_q\in\mathcal E_q,
\)
and first assume that only finitely many \(x_q\) are nonzero. Put
\(
y
:=
S_{\Lambda,V}x.
\)
For \(\lambda_j\in\Lambda\) and \(v\in\mathbb C^m\), the reproducing
property gives
\(
M_{D_q}^*
\bigl(
\widehat{k}_{\lambda_j}\otimes v
\bigr)
=
\overline{D_q(\lambda_j)}
\widehat{k}_{\lambda_j}\otimes v.
\)
Since \(D_q\) equals \(1\) on \(\sigma_q\) and \(0\) on
\(\Lambda\setminus\sigma_q\), we obtain
\(
M_{D_q}^*y
=
S_qx_q.
\)
Consequently,
\[
\sum_{q\in\mathbb N^{+}}\|S_qx_q\|^2
=
\sum_{q\in\mathbb N^{+}}\|M_{D_q}^*y\|^2
=
\|\mathcal M_D^*y\|_{\mathscr X}^2
\leq
M_D^2\|y\|^2
=
M_D^2\|S_{\Lambda,V}x\|^2.
\]

On the other hand, since
\(
s_{\min}(S_q)\geq c,
\
q\in\mathbb N^{+},
\)
we have
\(
\|S_qx_q\|
\geq
c\|x_q\|.
\)
Therefore,
\[
c^2\|x\|^2
=
c^2
\sum_{q\in\mathbb N^{+}}\|x_q\|^2
\leq
\sum_{q\in\mathbb N^{+}}\|S_qx_q\|^2
\leq
M_D^2\|S_{\Lambda,V}x\|^2.
\]
Thus
\(
\|S_{\Lambda,V}x\|
\geq
\frac{c}{M_D}\|x\|.
\)

The vectors having only finitely many nonzero cluster components are
dense in
\(
\bigoplus_{j\geq1}E_j
=
\bigoplus_{q\in\mathbb N^{+}}\mathcal E_q.
\)
Since \(S_{\Lambda,V}\) is bounded, the preceding estimate extends to
every
\(
x\in\bigoplus_{j\geq1}E_j.
\)
Hence \(S_{\Lambda,V}\) is bounded below, with lower bound at least
\(\frac{c}{M_D}\).
\end{proof}

\begin{proof}[Proof of Lemma~\ref{lem:bad-cluster-subsequence-transfer}]
We first prove the implication from the family of \(w\) points to the
family of \(z\) points. Suppose that
\(
s_{\min}(S_n^w)\rightarrow0.
\)
Assume, to the contrary, that
\(
s_{\min}(S_n^z)
\)
does not converge to \(0\). Then there exist \(\varepsilon>0\) and a
subsequence \((n_j)_{j\geq1}\) such that
\(
s_{\min}(S_{n_j}^z)\geq\varepsilon,
\
j\geq1.
\)
Restricting all the data to this subsequence and relabeling, we may
assume that
\begin{equation}
\label{eq:z-uniformly-bounded-below}
s_{\min}(S_n^z)\geq\varepsilon,
\qquad
n\geq1,
\end{equation}
while
\(
s_{\min}(S_n^w)\to0.
\)

Since \(E\) has finite dimension, for every \(n\in\mathbb N^+\)
there exists
\(
u_n=(u_{n,1},\ldots,u_{n,r})\in E,
\ \|u_n\|=1,
\)
such that
\(
\|S_n^w u_n\|
=
s_{\min}(S_n^w)
\rightarrow0.
\)
The unit sphere of \(E\) is compact. Moreover, for each \(k\), the
closed ball of radius \(b\) in
\(
\mathcal L(\mathbb C^{d_k},\mathbb C^m)
\)
is compact because this operator space has finite dimension.
After passing to a subsequence, without changing notation, we may
therefore assume simultaneously that
\(
u_n\rightarrow u=(u_1,\ldots,u_r),
\)
\(
V_{n,k}\rightarrow V_k
\quad\text{in operator norm},
\)
and
\(
\alpha_{n,k}\rightarrow
\alpha_k\in[\alpha_{-},\alpha_{+}]
\)
for every \(1\leq k\leq r\). In particular,
\(
\|u\|=1.
\)

Passing to a further subsequence, we may also assume that the Gram
matrices
\(
G_n^z
:=
[
\langle
\widehat{k}_{z_{n,k}},
\widehat{k}_{z_{n,\ell}}
\rangle
]_{k,\ell=1}^{r}
\)
and
\(
G_n^w
:=
[
\langle
\widehat{k}_{w_{n,k}},
\widehat{k}_{w_{n,\ell}}
\rangle
]_{k,\ell=1}^{r}
\)
converge to matrices \(G^z\) and \(G^w\), respectively, and that all
the limits
\(
\delta_{k\ell}^z
:=
\lim_{n\to\infty}
\rho(z_{n,k},z_{n,\ell}),
\
\delta_{k\ell}^w
:=
\lim_{n\to\infty}
\rho(w_{n,k},w_{n,\ell})
\)
exist. Applying the collision hypothesis to the
subsequence currently under consideration gives
\begin{equation}
\label{eq:common-zero-collision-relation}
\delta_{k\ell}^z=0
\quad\Longleftrightarrow\quad
\delta_{k\ell}^w=0.
\end{equation}

Define a relation on \(\{1,\ldots,r\}\) by
\(
k\sim\ell
\Longleftrightarrow
\delta_{k\ell}^z=0.
\)
Because \(\rho\) is a metric, its triangle inequality shows that
\(\sim\) is an equivalence relation. By
\eqref{eq:common-zero-collision-relation}, the same relation is
obtained from the family of \(w\) points. Denote its equivalence classes by
\(
C_1,\ldots,C_s.
\)

The matrices \(G^z\) and \(G^w\) are positive semidefinite and have
diagonal entries equal to \(1\). Hence there exist Hilbert spaces \(\mathcal H_z,\mathcal H_w\)
of finite dimension and unit vectors
\(
e_1^z,\ldots,e_r^z\in\mathcal H_z,
\
e_1^w,\ldots,e_r^w\in\mathcal H_w
\)
whose Gram matrices are \(G^z\) and \(G^w\), respectively.

Recall that
\(
1-
|
\langle
\widehat{k}_{\lambda},
\widehat{k}_{\mu}
\rangle
|^2
=
\rho(\lambda,\mu)^2.
\)
Consequently, if \(k,\ell\in C_a\), then
\(
|\langle e_k^z,e_\ell^z\rangle|
=
|\langle e_k^w,e_\ell^w\rangle|
=
1.
\)
Fix a representative \(k_a\in C_a\), and set
\(
f_a^z:=e_{k_a}^z,
\
f_a^w:=e_{k_a}^w.
\)
Equality in the Cauchy--Schwarz inequality shows that, for every
\(k\in C_a\), there exist constants
\(
c_k^z,c_k^w\in\mathbb T
\)
such that
\begin{equation}
\label{eq:within-cluster-phase-representation}
e_k^z=c_k^zf_a^z,
\qquad
e_k^w=c_k^wf_a^w.
\end{equation}

We next show that
\(
f_1^z,\ldots,f_s^z
\ \text{and}\
f_1^w,\ldots,f_s^w
\)
are linearly independent. For every \(n\), the Cauchy determinant
formula for normalized Szeg\H{o} kernels, applied to the pairwise
distinct points
\(
z_{n,k_1},\ldots,z_{n,k_s},
\)
gives
\(
\det
[
\langle
\widehat{k}_{z_{n,k_a}},
\widehat{k}_{z_{n,k_b}}
\rangle
]_{a,b=1}^{s}
=
\prod_{1\leq a<b\leq s}
\rho(z_{n,k_a},z_{n,k_b})^2.
\)
Taking the limit as \(n\to\infty\), we obtain
\(
\det
\left[
\langle f_a^z,f_b^z\rangle
\right]_{a,b=1}^{s}
=
\prod_{1\leq a<b\leq s}
(\delta_{k_ak_b}^z)^2
>0,
\)
because representatives of distinct equivalence classes have
strictly positive limiting distance. The same argument, applied to
\(
w_{n,k_1},\ldots,w_{n,k_s},
\)
gives
\(
\det
\left[
\langle f_a^w,f_b^w\rangle
\right]_{a,b=1}^{s}
=
\prod_{1\leq a<b\leq s}
(\delta_{k_ak_b}^w)^2
>0.
\)
Thus both representative families are linearly independent.

Set
\(
y_{n,k}
:=
\alpha_{n,k}V_{n,k}u_{n,k},
\
y_k
:=
\alpha_kV_ku_k.
\)
Then
\(
y_{n,k}\rightarrow y_k\) for \(1\leq k\leq r.\)
Using the convergence of \(G_n^w\), we obtain
\[
0
=
\lim_{n\to\infty}\|S_n^wu_n\|^2
=
\lim_{n\to\infty}
\sum_{k,\ell=1}^{r}
\left\langle
\widehat{k}_{w_{n,k}},
\widehat{k}_{w_{n,\ell}}
\right\rangle
\left\langle
y_{n,k},y_{n,\ell}
\right\rangle
=
\left\|
\sum_{k=1}^{r}
e_k^w\otimes y_k
\right\|^2.
\]
Therefore,
\(
0
=
\sum_{k=1}^{r}
e_k^w\otimes\alpha_kV_ku_k.
\)
By \eqref{eq:within-cluster-phase-representation}, this identity can
be written as
\(
0
=
\sum_{a=1}^{s}
f_a^w\otimes
\left(
\sum_{k\in C_a}
c_k^w\alpha_kV_ku_k
\right).
\)
Since \(f_1^w,\ldots,f_s^w\) are linearly independent, it follows
that
\begin{equation}
\label{eq:clusterwise-cancellation-on-w-side}
\sum_{k\in C_a}
c_k^w\alpha_kV_ku_k
=
0,
\qquad
1\leq a\leq s.
\end{equation}

For \(k\in C_a\), define
\(
v_k
:=
(c_k^z)^{-1}c_k^w\alpha_ku_k,
\
v:=(v_1,\ldots,v_r)\in E.
\)
Since
\(
|c_k^z|=|c_k^w|=1
\)
and
\(
\alpha_k\geq\alpha_{-},
\)
we have
\(
\|v\|^2
=
\sum_{k=1}^{r}
\alpha_k^2\|u_k\|^2
\geq
\alpha_{-}^2\|u\|^2
=
\alpha_{-}^2.
\)
In particular, \(v\neq0\).

Moreover,
\eqref{eq:clusterwise-cancellation-on-w-side} gives
\(
\sum_{k\in C_a}
c_k^zV_kv_k
=
\sum_{k\in C_a}
c_k^w\alpha_kV_ku_k
=
0\) for \(1\leq a\leq s.\)
Using \eqref{eq:within-cluster-phase-representation}, we conclude
that
\(
\sum_{k=1}^{r}
e_k^z\otimes V_kv_k
=
\sum_{a=1}^{s}
f_a^z\otimes
(
\sum_{k\in C_a}
c_k^zV_kv_k
)
=
0.
\)

The convergence of \(G_n^z\) and the convergence in operator norm
\(
V_{n,k}\to V_k
\)
now yield
\[
\lim_{n\to\infty}\|S_n^zv\|^2
=
\lim_{n\to\infty}
\sum_{k,\ell=1}^{r}
\left\langle
\widehat{k}_{z_{n,k}},
\widehat{k}_{z_{n,\ell}}
\right\rangle
\left\langle
V_{n,k}v_k,V_{n,\ell}v_\ell
\right\rangle
=
\left\|
\sum_{k=1}^{r}
e_k^z\otimes V_kv_k
\right\|^2
=
0.
\]
Consequently,
\(
\varepsilon
\leq
s_{\min}(S_n^z)
\leq
\frac{\|S_n^zv\|}{\|v\|}
\rightarrow0,
\)
contradicting \eqref{eq:z-uniformly-bounded-below}. This proves that
\(
s_{\min}(S_n^w)\rightarrow0
\Longrightarrow
s_{\min}(S_n^z)\rightarrow0.
\)

For the converse, define
\(
z_{n,k}'=w_{n,k},
\
w_{n,k}'=z_{n,k},
\
V_{n,k}'=\alpha_{n,k}V_{n,k},
\
\alpha_{n,k}'=\alpha_{n,k}^{-1}.
\)
Then
\(
\|V_{n,k}'u\|
\leq
\alpha_{+}b\|u\|,
\
\alpha_{+}^{-1}
\leq
\alpha_{n,k}'
\leq
\alpha_{-}^{-1}.
\)
Thus the primed data satisfy all the hypotheses of the implication
just proved. The corresponding operators satisfy
\(
S_n^{w,\prime}=S_n^z,
\
S_n^{z,\prime}=S_n^w.
\)
It follows that
\(
s_{\min}(S_n^z)\rightarrow0
\Longrightarrow
s_{\min}(S_n^w)\rightarrow0,
\)
which completes the proof.
\end{proof}

\section{Observation Localization for Countable Orbits}
\label{app:countable-observation-localization}

For a bounded operator \(C:H\to E\), where \(E\) is an arbitrary
Hilbert space, write \(\mathcal O_Tx=(CT^nx)_{n\ge0}\).
Boundedness of this observation operator is asserted only after the
corresponding upper estimate has been established. The arguments
below impose no restriction on the dimension of \(E\).

\begin{lem}\label{lem:countable-compact-observation}
Suppose that \(T\) is normal, \(\|T\|<1\), and \(\eta\) is a
continuous injection on \(\sigma(T)\). If \(S=\eta(T)\) satisfies
\(\|S\|<1\), then \(\mathcal O_T\) is bounded below if and only if
\(\mathcal O_S\) is bounded below. Both operators are bounded.
\end{lem}

\begin{proof}
For a strict contraction \(V\), the observation tail has the uniform
bound \(\frac{\|C\|^2\|V\|^{2M+2}}{1-\|V\|^2}\). Thus a lower bound
for \(\mathcal O_V\) is equivalent to an estimate over finitely many iterates
\begin{equation}\label{eq:countable-finite-horizon}
\sum_{n=0}^M\|CV^nx\|^2\ge a\|x\|^2.
\end{equation}
For any bounded \(V\), this is equivalent to the existence of
\(b>0\) such that
\begin{equation}\label{eq:countable-uniform-hautus}
\|Cx\|+\|(V-\lambda I)x\|\ge b\|x\|,
\qquad x\in H,\quad\lambda\in\mathbb C.
\end{equation}
Indeed, \eqref{eq:countable-finite-horizon} implies
\eqref{eq:countable-uniform-hautus} by expanding
\(V^n-\lambda^nI\) when \(\lambda\) is bounded; large
\(\lambda\) are handled by the second term alone.

For the converse, if every estimate over finitely many iterates fails, choose
unit vectors \(x_j\) such that
\(
\sum_{n=0}^j\|CV^nx_j\|^2\rightarrow0.
\)
Set
\(
\widehat H:=\ell^\infty(H)/c_0(H),
\widehat E:=\ell^\infty(E)/c_0(E).
\)
Let \(\widehat V:\widehat H\to\widehat H\) and
\(\widehat C:\widehat H\to\widehat E\) be the bounded operators
induced coordinatewise by \(V\) and \(C\), respectively.
The subspace
\(
M:=\bigcap_{n\ge0}\ker(\widehat C\widehat V^n)
\subset\widehat H
\)
is closed and invariant under \(\widehat V\). It is nonzero because
the class of \((x_j)\) has norm one and belongs to \(M\).
Estimate~\eqref{eq:countable-uniform-hautus} passes to quotient norms and
gives \(\|(\widehat V|_M-\lambda I)u\|\ge b\|u\|\) for every
\(\lambda\in\mathbb C\). This is impossible, since a bounded
operator on a nonzero complex Banach space has nonempty spectrum
and every boundary point of its spectrum is an approximate
eigenvalue. This proves the equivalence.

If \eqref{eq:countable-uniform-hautus} fails, its spectral parameters remain
bounded and may be passed to a convergent subsequence. Failure is
therefore equivalent to the existence of unit vectors \(y_j\) and
a fixed \(\lambda\) such that
\(Cy_j\to0\) and \((V-\lambda I)y_j\to0\).
For normal \(T\), the spectral theorem and continuity show that
approximate eigenvectors at \(\lambda\in\sigma(T)\) are
approximate eigenvectors of \(\eta(T)\) at \(\eta(\lambda)\).
The reverse implication follows from the continuous inverse of
\(\eta\) on its compact image. Explicitly, split the spectral
integral into a small neighborhood of \(\lambda\) and its
complement; on the latter, \(|\eta(z)-\eta(\lambda)|\) has a
positive minimum. Hence \eqref{eq:countable-uniform-hautus} is equivalent
for \((C,T)\) and \((C,S)\), proving the lemma.
\end{proof}

\begin{lem}\label{lem:countable-boundary-localization}
Let \(T\) be diagonal with eigenvalues in \(\mathbb D\),
let \(\psi\) be a disk homeomorphism, and put \(S=\psi(T)\).
Suppose \(\mathcal O_T\) is bounded and bounded below. If, for
some \(r<1\), the restriction of \(\mathcal O_S\) to
\(PH=\mathbf1_{\{|z|>r\}}(T)H\) has upper and positive lower
bounds, then \(\mathcal O_S\) has both bounds on all of \(H\).
\end{lem}

\begin{proof}
Put \(Q=I-P\) and
\(s=\max_{|z|\le r}|\psi(z)|<1\). On \(QH\), the norm of
\(S\) is at most \(s\), so a geometric series gives boundedness
of \(\mathcal O_SQ\). Together with the hypothesis this proves
global boundedness; write \(W=\mathcal O_S\).
Choose \(\alpha>0\) with \(\|WPx\|\ge\alpha\|Px\|\).
Since \(WS^kx\) is a tail of the observation sequence,
\[
\alpha\|PS^kx\|
\le\|WS^kx\|+\|WQS^kx\|
\le\|Wx\|+\|W\|s^k\|x\|.
\]
Consequently,
\begin{equation}\label{eq:countable-boundary-localization-estimate}
\|S^kx\|\le\alpha^{-1}\|Wx\|
+(1+\alpha^{-1}\|W\|)s^k\|x\|.
\end{equation}

If \(W\) is not bounded below, take unit vectors \(x_j\) with
\(Wx_j\to0\). Fix \(t\in(s,1)\) and let
\(R_t=\mathbf1_{\{|z|>t\}}(S)\). Normality and
\eqref{eq:countable-boundary-localization-estimate} give, for every \(k\),
\(
\limsup_j\|R_tx_j\|
\le(1+\alpha^{-1}\|W\|)\left(\frac{s}{t}\right)^k.
\)
Thus \(R_tx_j\to0\). On \(H_t=(I-R_t)H\), both \(S\) and
\(T=\psi^{-1}(S)\) are strict normal contractions: the latter
assertion uses compactness of
\(\psi^{-1}(t\overline{\mathbb D})\) in \(\mathbb D\).
The restriction of \(\mathcal O_T\) remains bounded below, so
Lemma~\ref{lem:countable-compact-observation} makes \(W|_{H_t}\) bounded
below. But
\[
\|(I-R_t)x_j\|\to1,\qquad
\|W(I-R_t)x_j\|\le\|Wx_j\|+\|W\|\|R_tx_j\|\to0,
\]
a contradiction.
\end{proof}

\section*{Declaration of competing interest}
The author declares that there are no competing interests
that could have influenced the work reported in this paper.

\section*{Funding}
This research did not receive any specific grant from funding
agencies in the public, commercial, or nonprofit sectors.

\section*{Data availability}
No data were used for the research described in this article.

\end{document}